\documentclass{article}
\usepackage[utf8]{inputenc}

 \usepackage{geometry}
\usepackage{amsmath}
\usepackage{comment}
\usepackage{leftidx}
\usepackage{multicol}
\usepackage{amssymb}
\usepackage{stmaryrd}
\usepackage{accents}
\usepackage{mathrsfs}
\usepackage{xfrac}
\usepackage{amsthm}
\usepackage{xcolor}
\usepackage{sectsty}
\usepackage{relsize}
\usepackage{float}

\usepackage{authblk}

\usepackage{indentfirst}
\usepackage{tikz}
\tikzset{shorten <>/.style={shorten >=#1,shorten <=#1}}

\usepackage{tikz-cd}
\newcounter{nodemaker}
\tikzset{Rightarrow/.style={double equal sign distance,>={Implies},->},
triple/.style={-,preaction={draw,Rightarrow}},
quadruple/.style={preaction={draw,Rightarrow,shorten >=0pt},shorten >=1pt,-,double,double
distance=0.2pt}}

\tikzset{%
    symbol/.style={%
        draw=none,
        every to/.append style={%
            edge node={node [sloped, allow upside down, auto=false]{$#1$}}}
    }
}
 \usepackage{quiver}
\makeatletter
\newcommand{\bigdoublevee}{\big@doubleop{\bigvee}}
\newcommand{\bigdoublewedge}{\big@doubleop{\bigwedge}}
\newcommand{\big@doubleop}[1]{%
  \DOTSB\mathop{\mathpalette\big@doubleop@aux{#1}}\slimits@
}

\newcommand\big@doubleop@aux[2]{%
  \sbox\z@{$\m@th#1#2$}%
  \makebox[1.35\wd\z@][s]{$\m@th#1#2\hss#2$}%
}
\makeatother

\makeatletter
\newcommand*{\doublerightarrow}[2]{\mathrel{
  \settowidth{\@tempdima}{$\scriptstyle#1$}
  \settowidth{\@tempdimb}{$\scriptstyle#2$}
  \ifdim\@tempdimb>\@tempdima \@tempdima=\@tempdimb\fi
  \mathop{\vcenter{
    \offinterlineskip\ialign{\hbox to\dimexpr\@tempdima+1em{##}\cr
    \rightarrowfill\cr\noalign{\kern.5ex}
    \rightarrowfill\cr}}}\limits^{\!#1}_{\!#2}}}
\newcommand*{\triplerightarrow}[1]{\mathrel{
  \settowidth{\@tempdima}{$\scriptstyle#1$}
  \mathop{\vcenter{
    \offinterlineskip\ialign{\hbox to\dimexpr\@tempdima+1em{##}\cr
    \rightarrowfill\cr\noalign{\kern.5ex}
    \rightarrowfill\cr\noalign{\kern.5ex}
    \rightarrowfill\cr}}}\limits^{\!#1}}}
\makeatother

\makeatletter
\newcommand{\xRrightarrow}[2][]{\ext@arrow 0359\Rrightarrowfill@{#1}{#2}}
\newcommand{\Rrightarrowfill@}{\arrowfill@\equiv\equiv\Rrightarrow}
\newcommand{\xLleftarrow}[2][]{\ext@arrow 3095\Lleftarrowfill@{#1}{#2}}
\newcommand{\Lleftarrowfill@}{\arrowfill@\Lleftarrow\equiv\equiv}
\makeatother

\DeclareFontFamily{U}{min}{}
\DeclareFontShape{U}{min}{m}{n}{<-> udmj30}{}

\DeclareFontFamily{U}{min}{}
\DeclareFontShape{U}{min}{m}{n}{<-> udmj30}{}

\DeclareFontFamily{U}{min}{}
\DeclareFontShape{U}{min}{m}{n}{<-> udmj30}{}

\DeclareUnicodeCharacter{3088}{\yo}

\newcommand{\Spec}
 {{\bf Spec}}
 
 \newcommand{\Alg}
 {{\bf Alg}}
 
 \newcommand{\psAlg}
 {{\bf psAlg}}
 
 \newcommand{\Lex}
 {{\bf Lex}}
 
 \newcommand{\biLex}
 {{\bf biLex}}
 
 \newcommand{\Et}
 {{\bf Et}}
 
\newcommand{\Loc}
 {{\bf Loc}}
 
  \newcommand{\Diag}
 {{\bf Diag}}
 
 \newcommand{\GTop}
 {{\bf GTop}}
 
 \newcommand{\Geom}
 {{\bf Geom}}
 
 \newcommand{\Cat}
 {{\bf Cat}}

\newcommand{\cod}
 {{\rm cod}}

\newcommand{\comp}
 {\circ}

\newcommand{\Cont}
 {{\bf Cont}}

\newcommand{\dom}
 {{\rm dom}}
 
 \newcommand{\id}
 {{\rm id}}
 
\newcommand{\fp}
{{\rm fp }}

\newcommand{\li}
{{\textup{lim} }}

\newcommand{\oplaxlim}
{{\textup{oplaxlim} }}

\newcommand{\oplaxcolim}
{{\textup{oplaxcolim} }}

\newcommand{\oplaxbicolim}
{{\textup{oplaxbicolim} }}

\newcommand{\oplaxpscolim}
{{\textup{oplaxpscolim} }}

\newcommand{\laxcolim}
{{\textup{laxcolim} }}

\newcommand{\bicolim}
{{\textup{bicolim} }}

\newcommand{\bilim}
{{\textup{bilim} }}

\newcommand{\pscolim}
{{\textup{pscolim} }}

\newcommand{\pslim}
{{\textup{pslim} }}

\newcommand{\lan}
{{\textup{lan} }}

\newcommand{\ran}
{{\textup{ran} }}

\newcommand{\colim}
{{\textup{colim}}}

\newcommand{\coeq}
{{\textup{coeq}}}

\newcommand{\bicoeq}
{{\textup{bicoeq}}}

\newcommand{\eq}
{{\textup{eq}}}

\newcommand{\comma}[2]
{\mbox{$(#1\!\downarrow\!#2)$}}

\newcommand{\empstg}
 {[\,]}

\newcommand{\epi}
 {\twoheadrightarrow}

\newcommand{\hy}
 {\mbox{-}}

\newcommand{\im}
 {{\rm im}}

\newcommand{\imp}
 {\!\Rightarrow\!}

\newcommand{\Ind}
 {{\rm Ind}}
 
 \newcommand{\Pro}
 {{\rm Pro}}

\newcommand{\mono}
 {\rightarrowtail}

\newcommand{\ob}
 {{\rm ob}}
 
 \newcommand{\can}
 {{\rm can}}
 
 \newcommand{\Hom}
 {{\rm Hom}}

\newcommand{\op}
 {{\rm op}}
 
 \newcommand{\pt}
 {{\bf pt}}

\newcommand{\Set}
 {{\bf Set }}

\newcommand{\Sh}
 {{\bf Sh}}

\newcommand{\sh}
 {{\bf sh}}

\newcommand{\Sub}
 {{\rm Sub}}

\newcommand{\Flat}
 {{\bf Flat}}
\newcommand{\biIns}
{{\textup{biIns} }}

\newcommand{\biInv}
{{\textup{biInv} }}

\newcommand{\biend}
{{\textup{bi}\int }}

\newcommand{\biLan}
{{\textup{biLan} }}

\newcommand{\biRan}
{{\textup{biRan} }}

 \newcommand{\Cart}
 {{\bf Cart}}

  \newcommand{\ps}
 {{\bf ps}}

 \newcommand{\St}
 {{\bf St}}

\usepackage{hyperref}
\hypersetup{
    colorlinks = true,
    linkbordercolor = {red},
   	linkcolor ={teal},
	anchorcolor = {pink},
	citecolor =  {orange},
	filecolor = {teal},
	menucolor = {teal},
	runcolor =  {teal},
	urlcolor = {teal},
}

\usepackage{cleveref}
\newtheorem{theorem}{Theorem}[section]
\newtheorem*{theorem*}{Theorem}
\newtheorem{proposition}[theorem]{Proposition}
\newtheorem{corollary}[theorem]{Corollary}
\newtheorem{corollary'}[theorem]{Corollary}
\newtheorem{lemma}[theorem]{Lemma}
\theoremstyle{definition}
\newtheorem{definition}[theorem]{Definition}
\newtheorem{example}[theorem]{Example}
\theoremstyle{definition}
\newtheorem{remark}[theorem]{Remark}
\theoremstyle{definition}
\newtheorem{division}[theorem]{}

\usepackage{imakeidx}
\usepackage{authblk}

\title{Codescent and bicolimits of pseudo-algebras}
\author{Axel Osmond\thanks{Istituto Grothendieck, Mondovì, Italy}}
\date{April 2022}

\begin{document}

\maketitle

\begin{abstract}
    We categorify cocompleteness results of monad theory, in the context of pseudomonads. We first prove a general result establishing that, in any 2-category, weighted bicolimits can be constructed from oplax bicolimits and bicoequalizers of codescent objects. After prerequisites on pseudomonads and their pseudo-algebras, we give a 2-dimensional Linton theorem reducing bicocompleteness of 2-categories of pseudo-algebras to existence of bicoequalizers of codescent objects. Finally we prove this condition to be fulfilled in the case of a bifinitary pseudomonad, ensuring bicocompleteness.   
    
\end{abstract}

\section*{Introduction}



This work provides a categorification of two classical results of monad theory. For a monad on a cocomplete category, it is well known that existence of colimits in the category of algebras reduces to existence of coequalizers, since the latter can be used in a first step to construct coproducts of algebras, and secondly, arbitrary colimits as quotients of such coproducts. In the case of a finitary monad (or more generally a monad with rank) such coequalizers can always be forced to exist thanks to a transfinite process, ensuring the cocompletness of categories of algebras of finitary monads.\\  

  
In 2-category theory, 2-dimensional analogs of monads and their algebras decline in several flavours of strictness. As well as in the 1-categorical case, 2-categories of algebras of a 2-monad - whatever the notion of strictness one would consider - may lack some colimits without specific assumptions; however, in the case of \emph{finitary} 2-monads (or more generally for a 2-monad with rank) on a 2-complete and 2-cocomplete 2-category, it is known that both the 2-categories of \emph{strict algebras} with either \emph{strict morphisms} or \emph{pseudomorphisms} are known to have all bicolimits: this result, originating in \cite{blackwell1989two}, first construct 2-colimits for strict morphisms, then uses strictification techniques to extract those bicolimits from existing stricter 2-colimits of strict algebras and strict morphisms through some adjunction.\\

However, this cannot be used for pseudo-algebras; it is established that not all pseudoalgebras of a 2-monad can be strictified into strict algebras (see for instance \cite{shulman2012not}), even in the case of well-behaved 2-monad: this complicates any attempt to extract bicolimits from the 2-category of strict algebras and pseudomorphisms.\\

In this work, we consider the more general setting of \emph{pseudomonads}, where the canonical natural transformations coding for the unit and the multiplication are related only through invertible natural modifications rather than strict equalities. Our motivation to consider pseudomonad rather than 2-monad originates in \cite{ODL}, where we prove that 2-categories of \emph{$\Phi$-exact} categories relative to a class of finite weights $\Phi$ are finitely bipresentable: our proof implicitly requires bicocompletness of the 2-category of pseudo-algebras of a certain pseudomonad  \cite{garner2012lex} associated to $\Phi$. In the context of pseudomonads, the stricter notions of algebras are less meaningful (for instance, the free construction does not return strict algebras): for this reason, it is far from clear whether strictification results of the kind of \cite{blackwell1989two} may help. \\

This paper proposes hence a ``brute force" proof of the bicocompleteness of the 2-category of pseudo-algebras of bifinitary pseudomonads. Our strategy, inspired by the classics of 1-dimensional monad theory as \cite{borceux1994handbook}[Section 4] or \cite{barr2000toposes}, reduces existence of arbitrary bicolimits to existence of bicolimits of more specific shape, which can be more directly ensured, in particular for finitary pseudomonads. \\

In 1-dimension, it is known that cocompleteness of categories of algebras depends on the sole existence of coequalizers, and that those latter exist in the case of a finitary monad thanks to a famous, yet arcane strategy relying on a transfinite induction. In 2-dimension, though several shapes of bicolimits could provide generalizations of coequalizers of parallel pairs, we claim that in our context their correct analogs are \emph{bicoequalizers of codescent objects} in the sense of \cite{le2002beck} (also known as \emph{codescent objects of coherence data}). Those latter encode additional data, which are akin to internal categories and also 2-dimensional notions of congruences. The role of codescent objects in the theory of 2-monads has been established for some time. Many results of monad theory involving instances of reflexives and split coequalizers generalize into pseudocoequalizing statements relative to some codescent objects in the 2-dimensional context. For instance, in \cite{bourke2010codescent} and \cite{le2002beck}, a 2-dimensional version of the \emph{Barr construction} (also: the \emph{resolution} or \emph{reflexive coherence data}) of pseudo-algebras of pseudomonad is given, involving a canonical codescent object made of free pseudo-algebras.\\

In this paper, we prove that codescent objects are, more generally, useful to generate more general bicolimits. In 1-category theory, colimits can be constructed from coproducts and coequalizer; in 2-category theory, we often use the fact that weighted bicolimits can be constructed from coproduct, coinserted and coequifiers. We prove another result of this kind in our section 1: we reduce existence of weighted bicolimits to existence of oplax bicolimits and bicoequalizers of codescent objects, see \cref{bicolim is a bicoeq}. Our argument relies on the more recent notion of \emph{$\sigma$-bicolimit}, intermediate between bicolimits and oplax bicolimits, which allows to turn weighted bicolimit into conical ones, see \cref{deweighting}. Then, for a given functor over a marked 2-category, we construct at \cref{codescent object at the oplax colimit} a certain codescent object from its oplax bicolimit, whose higher data encode the maps we are going to invert in the $\sigma$-bicolimit, and exhibit the $\sigma$-bicolimit as the bicoequalizer of this codescent diagram. \\

Then we apply this result in the context of pseudo-algebras of pseudomonads. We prove in section 3 that one can construct the oplax bicolimit of a diagram of pseudoalgebras as the bicoequalizer of a certain codescent diagram constructed from the oplax bicolimit of the objects underlying the algebras, see \cref{Linton}, categorifying a famous result from \cite{linton1969coequalizers}. As a consequence, the sole existence of bicoequalizers of codescent objects in the 2-category of pseudo-algebras becomes sufficient to ensure existence of oplax bicolimits, and hence, following the result of our section 1, of all bicolimit. \\

Finally, in section 4, we establish the existence of bicoequalizers of codescent diagrams in the 2-category of pseudo-algebras for a bifinitary pseudomonad -- see \cref{existence of codescent objec for a finitary pseudomonad}. Our construction is very close in spirit to \cite{borceux1994handbook}[Theorem 4.3.6]. This ensures from what precedes that the 2-category of pseudo-algebras of a bifinitary pseudomonad is always bicocomplete, which is \cref{The theorem}. From this we deduce cocompleteness of various examples as $\Lex$ and 2-categories of $\Phi$-exact categories for classes of finite weights.  

\newpage

\section{An observation on oplax bicolimits and codescent objects}

First recall that bicolimits in $\Cat$ are pseudocolimits. Moreover, conical bicolimits are computed as localizations of oplax-colimits at cartesian morphisms. We would like to extend this result to arbitrary weighted bicolimits. We recall here the notion of \emph{$\sigma$-bicolimit} (also \emph{marked bicolimit} in \cite{gagna2021bilimits}), the main reference for it being \cite{descotte2018sigma}


\begin{definition}[$\sigma$-natural transformations]
Let be $I$ a 2-category and $ \Sigma$ a class of maps in $I$ containing equivalences and closed under composition and invertible 2-cells; let be $ \mathcal{C}$ a category and $ F,G : I \rightarrow \mathcal{C}$ a pair of 2-functors. A $\sigma$-\emph{natural transformation} relatively to $\Sigma$ is a lax natural transformation $ f: F \Rightarrow G$ whose lax naturality squares
\[\begin{tikzcd}
	{F(i)} & {F(j)} \\
	{G(i)} & {G(j)}
	\arrow["{F(s)}", from=1-1, to=1-2]
	\arrow["{f_i}"', from=1-1, to=2-1]
	\arrow[""{name=0, anchor=center, inner sep=0}, "{f_j}", from=1-2, to=2-2]
	\arrow[""{name=1, anchor=center, inner sep=0}, "{G(s)}"', from=2-1, to=2-2]
	\arrow["{f_s}", curve={height=-6pt}, shorten <=4pt, shorten >=4pt, Rightarrow, from=1, to=0]
\end{tikzcd}\]
at an arrow $ s$ in $\Sigma$ are invertible 2-cells of $\mathcal{C}$. Similarly an \emph{op}$\sigma$-\emph{natural transformation} is an oplax natural transformation whose oplax naturality squares at arrow at arrows in $\Sigma$ are invertible.   We denote as $[I, \mathcal{C}]_\Sigma $ the 2-category of strict 2-functors and $ \sigma$-natural transformations relatively to $ \Sigma$, with no restriction on 2-cells, and $[I, \mathcal{C}]_{\op\Sigma} $ for op$\sigma$-natural transformations.  \\

A \emph{$\sigma$-cocone} is an \emph{op$\sigma$}-natural transformation $ q: F \Rightarrow \Delta_B$, with oplax naturality triangles - with the $q_s$ for $s \in \Sigma$ invertible:
\[\begin{tikzcd}
	{F(i)} && {F(j)} \\
	& B
	\arrow["{F(d)}", from=1-1, to=1-3]
	\arrow[""{name=0, anchor=center, inner sep=0}, "{q_j}", from=1-3, to=2-2]
	\arrow[""{name=1, anchor=center, inner sep=0}, "{q_i}"', from=1-1, to=2-2]
	\arrow["{q_d}"', shorten <=6pt, shorten >=6pt, Rightarrow, from=0, to=1]
\end{tikzcd}\]
Now the \emph{$\sigma$-bicolimit} of a 2-functor $F : I \rightarrow \mathcal{C}$ relative to $\Sigma$ is a $\sigma$-cocone $ q : F \Rightarrow \Sigma\bicolim_I \, F $ such that one has a natural equivalence
\[ \mathcal{C}[\underset{i \in I}{\Sigma\bicolim} \, F(i), -]
\simeq  [I,\mathcal{C}]_{\op\Sigma} [F,\Delta] \]
\end{definition}

The interest of $\sigma$-bicolimit is that they allow to ``deweight" weighted bicolimits (see more generally \cite{descotte2018sigma}[2.4.9]):

\begin{proposition}[Deweighting lemma]\label{deweighting}
Let $\mathcal{C}$ be a 2-category. If $\mathcal{C}$ has (conical) $\sigma$-bicolimits, then it has all weighted bicolimits. 
\end{proposition}

\begin{proof}
Let be $F : I \rightarrow \mathcal{C}$ with $I$ a small 2-category and $ W : I^{\op} \rightarrow \Cat$ a 2-functor. One can take the 2-category of elements $ \int W$ and consider the composite 
\[\begin{tikzcd}
	{\displaystyle \int W} & I & {\mathcal{C}}
	\arrow["{\pi_W}", from=1-1, to=1-2]
	\arrow["F", from=1-2, to=1-3]
\end{tikzcd}\]
The 2-category of elements $ \int W$ is endowed with the class of cartesian arrows $\Cart_W$: we claim that the weighted bicolimit of $F$ can be recovered as the $\sigma$-bicolimit of $F\pi_W$ at the cartesian arrows, that is 
\[  \underset{I}{\bicolim^W} \; F \simeq \Cart_W\underset{I}{\bicolim} \; F\pi_W \]
\end{proof}

\begin{division}\label{sigmacolim in Cat}
In $\Cat$, $\sigma$-bicolimits are $\sigma$-pseudocolimits and can be constructed in the same way as pseudocolimits by localizing oplax bicolimits, but this time only at cartesian arrows over arrows in the marked class. More precisely, for a $\sigma$- pair $ (I,\Sigma)$ and a 2-functor $ F : I \rightarrow \Cat $, the $\sigma$-colimit is obtained as the localization of the oplax colimit at cartesian lifts of $\Sigma$-arrows. If one defines
\[ \Sigma_{(F,\Sigma)} = \textbf{Cart}_F \cap \pi_F^{-1}(\Sigma)  \]
with $ \pi_F : \oplaxcolim_{i \in I} F(i) \rightarrow I$ the associated fibration (recall that $\oplaxcolim_{i \in I} F(i)$ is the underlying category of the Grothendieck construction of $ F$), one has the equation below. Moreover this $\sigma$-bicolimit can be chosen as a $\sigma$-bicolimit. 
\[   \underset{i \in I}{\Sigma\bicolim} \, F(i) \simeq \underset{i \in I}{\oplaxcolim}\,  F(i)  [\Sigma_{(F,\Sigma)}^{-1}].  \]
\end{division}

We want to generalize this statement in arbitrary 2-categories: this implies to understand what a correct generalization of localizations is. While this operation could be described in term of coinverters and coequifiers, we shall be interested in a special class of bicolimits, which could really be seen as a colimit of internal categories - here encoding the arrows one wants to localize as if they lived \emph{inside} an object. The following notion will be central to our work; we follow here mostly defintions and notations from \cite{le2002beck}.

\begin{division}[Codescent diagram]\label{codescent diagram}
In the following $ \mathbb{X}$ will denote the following truncated simplicial object 
\[\begin{tikzcd}
	2 & 1 & 0
	\arrow["{d_0}", shift left=2, from=1-2, to=1-3]
	\arrow["{d_1}"', shift right=2, from=1-2, to=1-3]
	\arrow["i"{description}, from=1-3, to=1-2]
	\arrow["{p_0}", shift left=2, from=1-1, to=1-2]
	\arrow["{p_1}"{description}, from=1-1, to=1-2]
	\arrow["{p_2}"', shift right=2, from=1-1, to=1-2]
\end{tikzcd}\]
where we shall denote the $ d_0,i,d_1$ as the \emph{lower codescent data} and the $p_0,p_1,p_2$ as the \emph{higher codescent data}, together with the following invertible 2-cells exhibiting $ i$ as a common pseudosection of $d_0,d_1$:
\[\begin{tikzcd}
	& 0 \\
	1 & 0
	\arrow["i", from=2-2, to=2-1]
	\arrow[""{name=0, anchor=center, inner sep=0}, "{d_0}", from=2-1, to=1-2]
	\arrow[Rightarrow, no head, from=1-2, to=2-2]
	\arrow["{n_0 \atop\simeq}"{description, pos=0.6}, Rightarrow, draw=none, from=0, to=2-2]
\end{tikzcd} \hskip1cm \begin{tikzcd}
	& 0 \\
	1 & 0
	\arrow["i", from=2-2, to=2-1]
	\arrow[""{name=0, anchor=center, inner sep=0}, "{d_1}", from=2-1, to=1-2]
	\arrow[Rightarrow, no head, from=1-2, to=2-2]
	\arrow["{n_1 \atop\simeq}"{description, pos=0.6}, Rightarrow, draw=none, from=0, to=2-2]
\end{tikzcd}\]
and the following invertible 2-cells:
\[\begin{tikzcd}[sep=small]
	& 1 \\
	2 && 0 \\
	& 1
	\arrow["{p_0}", from=2-1, to=1-2]
	\arrow["{d_0}", from=1-2, to=2-3]
	\arrow["{p_1}"', from=2-1, to=3-2]
	\arrow["{d_0}"', from=3-2, to=2-3]
	\arrow["{\theta_{01} \atop \simeq}"{description}, shorten <=10pt, shorten >=10pt, Rightarrow, draw=none, from=1-2, to=3-2]
\end{tikzcd} \hskip1cm 
\begin{tikzcd}[sep=small]
	& 1 \\
	2 && 0 \\
	& 1
	\arrow["{p_0}", from=2-1, to=1-2]
	\arrow["{d_1}", from=1-2, to=2-3]
	\arrow["{p_2}"', from=2-1, to=3-2]
	\arrow["{d_0}"', from=3-2, to=2-3]
	\arrow["{\theta_{02}\atop \simeq}"{description}, shorten <=10pt, shorten >=10pt, Rightarrow, draw=none, from=1-2, to=3-2]
\end{tikzcd} \hskip1cm
\begin{tikzcd}[sep=small]
	& 1 \\
	2 && 0 \\
	& 1
	\arrow["{p_1}", from=2-1, to=1-2]
	\arrow["{d_1}", from=1-2, to=2-3]
	\arrow["{p_2}"', from=2-1, to=3-2]
	\arrow["{d_1}"', from=3-2, to=2-3]
	\arrow["{\theta_{12}\atop \simeq}"{description}, shorten <=10pt, shorten >=10pt, Rightarrow, draw=none, from=1-2, to=3-2]
\end{tikzcd}
\]
\end{division}

\begin{definition}
A \emph{codescent object} in a 2-category $\mathcal{C}$ is a 2-functor $ \mathscr{X} : \mathbb{X} \rightarrow \mathcal{C}$. A \emph{morphism of codescent object} is a pseudonatural transformation in $[\mathbb{X}, \mathcal{C}]_\ps$. 
\end{definition}

Codescent diagrams define shapes over which we can compute the following kind of weighted bicolimit, their \emph{bicoequalizer}:

\begin{definition}\label{bicoequalization of codescent diagram}
We shall say that a morphism $ q: \mathscr{X}(0) \rightarrow C$ \emph{pseudocoequalizes} $ \mathbb{X}$ if it inserts an invertible 2-cell 
\[\begin{tikzcd}
	{\mathscr{X}(1)} & {\mathscr{X}(0)} \\
	{\mathscr{X}(0)} & C
	\arrow["{\mathscr{X}(d_0)}", from=1-1, to=1-2]
	\arrow["q", from=1-2, to=2-2]
	\arrow["{\mathscr{X}(d_1)}"', from=1-1, to=2-1]
	\arrow["q"', from=2-1, to=2-2]
	\arrow["{\xi \atop \simeq}"{description}, draw=none, from=1-1, to=2-2]
\end{tikzcd}\]
satisfying moreover the following identity, which we shall refer to as the \emph{lower coherence condition}:
\[\begin{tikzcd}[sep=large]
	{\mathscr{X}(0)} \\
	& {\mathscr{X}(1)} & {\mathscr{X}(0)} \\
	& {\mathscr{X}(0)} & C
	\arrow["{\mathscr{X}(d_0)}"{description}, from=2-2, to=2-3]
	\arrow["{\mathscr{X}(d_1)}"{description}, from=2-2, to=3-2]
	\arrow["q"', from=3-2, to=3-3]
	\arrow["q", from=2-3, to=3-3]
	\arrow["{\mathscr{X}(i)}"{description}, from=1-1, to=2-2]
	\arrow[""{name=0, anchor=center, inner sep=0}, curve={height=18pt}, Rightarrow, no head, from=2-3, to=1-1]
	\arrow[""{name=1, anchor=center, inner sep=0}, curve={height=22pt}, Rightarrow, no head, from=1-1, to=3-2]
	\arrow["{\xi \atop \simeq}"{description}, draw=none, from=2-2, to=3-3]
	\arrow["{\mathscr{X}(n_0) \atop \simeq}"{description}, Rightarrow, draw=none, from=0, to=2-2]
	\arrow["{\mathscr{X}(n_1) \atop \simeq}"{description, pos=0.4}, Rightarrow, draw=none, from=2-2, to=1]
\end{tikzcd}
=\begin{tikzcd}
	{\mathscr{X}(0)} && C & {}
	\arrow[""{name=0, anchor=center, inner sep=0}, "q"', bend right=20, end anchor=-140, from=1-1, to=1-3]
	\arrow[""{name=1, anchor=center, inner sep=0}, "q", bend left=20, end anchor=140, from=1-1, to=1-3]
	\arrow[shorten <=3pt, shorten >=3pt, Rightarrow, no head, from=1, to=0]
\end{tikzcd}\]
together with the following identiy which we shall refer to as the \emph{higher coherence condition}
\[\begin{tikzcd}[sep=small]
	{\mathscr{X}(2)} && {\mathscr{X}(1)} \\
	& {\mathscr{X}(1)} && {\mathscr{X}(0)} \\
	{\mathscr{X}(1)} \\
	& {\mathscr{X}(0)} && C
	\arrow["{\mathscr{X}(d_0)}"{description}, from=2-2, to=2-4]
	\arrow["q", from=2-4, to=4-4]
	\arrow["{\mathscr{X}(d_1)}"{description}, from=2-2, to=4-2]
	\arrow["q"', from=4-2, to=4-4]
	\arrow["{\xi \atop \simeq}"{description}, draw=none, from=2-2, to=4-4]
	\arrow["{\mathscr{X}(p_1)}", from=1-1, to=2-2]
	\arrow["{\mathscr{X}(p_0)}", from=1-1, to=1-3]
	\arrow["{\mathscr{X}(d_0)}", from=1-3, to=2-4]
	\arrow["{\mathscr{X}(p_2)}"', from=1-1, to=3-1]
	\arrow["{\mathscr{X}(d_1)}"', from=3-1, to=4-2]
	\arrow["{\mathscr{X}(\theta_{01}) \atop \simeq}"{description}, draw=none, from=2-2, to=1-3]
	\arrow["{\mathscr{X}(\theta_{12}) \atop \simeq}"{description}, draw=none, from=3-1, to=2-2]
\end{tikzcd}
=\begin{tikzcd}[sep=small]
	{\mathscr{X}(2)} && {\mathscr{X}(1)} \\
	&&& {\mathscr{X}(0)} \\
	{\mathscr{X}(1)} && {\mathscr{X}(0)} \\
	& {\mathscr{X}(0)} && C
	\arrow["q", from=2-4, to=4-4]
	\arrow["q"', from=4-2, to=4-4]
	\arrow["{\mathscr{X}(p_0)}", from=1-1, to=1-3]
	\arrow["{\mathscr{X}(d_0)}", from=1-3, to=2-4]
	\arrow["{\mathscr{X}(p_2)}"', from=1-1, to=3-1]
	\arrow["{\mathscr{X}(d_1)}"', from=3-1, to=4-2]
	\arrow["{\mathscr{X}(d_0)}"{description}, from=3-1, to=3-3]
	\arrow["{\mathscr{X}(d_1)}"{description}, from=1-3, to=3-3]
	\arrow["q"{description}, from=3-3, to=4-4]
	\arrow["{\mathscr{X}(\theta_{02}) \atop \simeq}"{description}, draw=none, from=1-1, to=3-3]
	\arrow["{\xi \atop \simeq}"{description}, draw=none, from=3-3, to=4-2]
	\arrow["{\xi \atop \simeq}"{description}, draw=none, from=3-3, to=2-4]
\end{tikzcd}\]


A \emph{bicoequalizer} of a codescent object $ \mathscr{X}$ is the data of a pair $(q_\mathscr{X} : \mathscr{X}(0) \rightarrow \bicoeq(\mathscr{X}), \xi_\mathscr{X})$ which is universal amongst the one pseudocoequalizing $\mathscr{X}$, in the sense that for any other pseudocoequalizing $ (C,q, \xi)$ there is a canonical 1-cell $ \langle q, \xi \rangle : \bicoeq(\mathscr{X}) \rightarrow C $, unique up to unique invertible 2-cell, together with a universal 2-cell  
\[\begin{tikzcd}[sep=large]
	& C \\
	{\mathscr{X}(0)} & {\bicoeq(\mathscr{X})}
	\arrow[""{name=0, anchor=center, inner sep=0}, "q", from=2-1, to=1-2]
	\arrow["{q_\mathscr{X}}"', from=2-1, to=2-2]
	\arrow["{ \langle q, \xi \rangle}"', from=2-2, to=1-2]
	\arrow["{\theta_{ \langle q, \xi \rangle} \atop \simeq}"{description, pos=0.6}, Rightarrow, draw=none, from=0, to=2-2]
\end{tikzcd}\]
satisfying the following identity - natural in $(C,q,\xi)$:
\[\begin{tikzcd}[sep=large]
	{\mathscr{X}(1)} & {\mathscr{X}(0)} \\
	{\mathscr{X}(0)} & {\bicoeq(\mathscr{X})} \\
	&& C
	\arrow[""{name=0, anchor=center, inner sep=0}, "q", curve={height=-24pt}, from=1-2, to=3-3]
	\arrow["{q_\mathscr{X}}"{description}, from=1-2, to=2-2]
	\arrow["{ \langle q, \xi \rangle}"{description}, from=2-2, to=3-3]
	\arrow["{q_\mathscr{X}}"{description}, from=2-1, to=2-2]
	\arrow[""{name=1, anchor=center, inner sep=0}, "q"', curve={height=20pt}, from=2-1, to=3-3]
	\arrow["{\mathscr{X}(d_1)}"', from=1-1, to=2-1]
	\arrow["{\mathscr{X}(d_0)}", from=1-1, to=1-2]
	\arrow["{\xi_{\mathscr{X}} \atop \simeq}"{description}, draw=none, from=1-1, to=2-2]
	\arrow["{\theta_{ \langle q, \xi \rangle} \atop \simeq}"{description, pos=0.66}, Rightarrow, draw=none, from=0, to=2-2]
	\arrow["{\theta_{ \langle q, \xi \rangle}^{-1} \atop \simeq}"{description}, Rightarrow, draw=none, from=2-2, to=1]
\end{tikzcd}
=\begin{tikzcd}[sep=large]
	{\mathscr{X}(1)} & {\mathscr{X}(0)} \\
	{\mathscr{X}(0)} & C
	\arrow["q", from=1-2, to=2-2]
	\arrow["q"', from=2-1, to=2-2]
	\arrow["{\mathscr{X}(d_1)}"', from=1-1, to=2-1]
	\arrow["{\mathscr{X}(d_0)}", from=1-1, to=1-2]
	\arrow["{\xi \atop \simeq}"{description}, draw=none, from=1-1, to=2-2]
\end{tikzcd}\]
\end{definition}

\begin{remark}\label{the weight}
The most formal definition of the bicoequalizer, one can find in \cite{le2002beck}[Definition 2.1], involves a convenient weight and describe the bicoequalizer as a weighted bicolimit, which is more suited to correctly state the universal property. Since we are going to make a brief reference to it later, we give a quick expression of it: it is the 2-functor $ \mathcal{J} : \mathbb{X}^{\op} \rightarrow \Cat $ defined as $ \mathcal{J}(0) = *$ (the point category), $\mathcal{J}(1) = \{ x \simeq y \}$ (the walking isomorphism) and $ \mathcal{J}(2) = \{ x \simeq y \simeq z \}$ (the walking triangle of isomorphisms); $\mathcal{J}(d_0),\mathcal{J}(d_1) : \mathcal{J}(0) \rightrightarrows \mathcal{J}(1)$ point respectively to $x$ and $y$, and $ \mathcal{J}(p_0), \mathcal{J}(p_1), \mathcal{J}(p_2)$ point respectively toward $ x \simeq y$, $x \simeq z$, $y\simeq z$.   
\end{remark}

\begin{remark}
The terminology in the theory of codescent is not totally unified. Here we follow \cite{le2002beck} terminology. Another terminology is also widely accepted: \cite{bourke2010codescent} (see for instance section 2.2) as well as \cite{lack2002codescent} speak of \emph{coherence data} for what we call a codescent object and say codescent object for our bicoequalizers. 
\end{remark}

\begin{remark}\label{functoriality of bicoequalize}[Pseudofunctoriality of bicoequalizers of codescent objects]
As any bicolimit construction, taking bicoequalizers defines a pseudofunctor $ \bicoeq(-) : [\mathbb{X}, \mathcal{C}] \rightarrow \mathcal{C}$. This means that any morphism of codescent objects $ x : \mathscr{X} \Rightarrow \mathscr{Y} $ defines uniquely a canonical invertible 2-cell 
\[\begin{tikzcd}
	{\mathscr{X}(0)} & {\bicoeq{(\mathscr{X})}} \\
	{\mathscr{Y}(0)} & {\bicoeq{(\mathscr{Y})}}
	\arrow["{q_{\mathscr{X}}}", from=1-1, to=1-2]
	\arrow["{\bicoeq{(x)}}", from=1-2, to=2-2]
	\arrow["{x_0}"', from=1-1, to=2-1]
	\arrow["{q_{\mathscr{Y}}}"', from=2-1, to=2-2]
	\arrow["{\overline{x} \atop \simeq}"{description}, draw=none, from=1-1, to=2-2]
\end{tikzcd}\]
satisfying the coherence condition 
\[\begin{tikzcd}[sep=large]
	{\mathscr{X}(1)} & {\mathscr{X}(0)} \\
	{\mathscr{Y}(1)} & {\mathscr{X}(0)} & {\bicoeq{(\mathscr{X})}} \\
	& {\mathscr{Y}(0)} & {\bicoeq{(\mathscr{Y})}}
	\arrow["{q_{\mathscr{X}}}"{description}, from=2-2, to=2-3]
	\arrow["{\bicoeq{(x)}}", from=2-3, to=3-3]
	\arrow["{x_0}"{description}, from=2-2, to=3-2]
	\arrow["{q_{\mathscr{Y}}}"', from=3-2, to=3-3]
	\arrow["{\overline{x} \atop \simeq}"{description}, draw=none, from=2-2, to=3-3]
	\arrow[""{name=0, anchor=center, inner sep=0}, "{\mathscr{X}(d_1)}"{description}, from=1-1, to=2-2]
	\arrow["{x_1}"', from=1-1, to=2-1]
	\arrow[""{name=1, anchor=center, inner sep=0}, "{\mathscr{X}(d_1)}"', from=2-1, to=3-2]
	\arrow["{\mathscr{X}(d_0)}", from=1-1, to=1-2]
	\arrow["{q_{\mathscr{X}}}", from=1-2, to=2-3]
	\arrow["{\xi \atop \simeq}"{description}, draw=none, from=1-2, to=2-2]
	\arrow["{x_{d_1} \atop \simeq}"{description}, Rightarrow, draw=none, from=0, to=1]
\end{tikzcd} =
\begin{tikzcd}[sep=large]
	{\mathscr{X}(1)} & {\mathscr{X}(0)} \\
	{\mathscr{Y}(1)} & {\mathscr{Y}(0)} & {\bicoeq{(\mathscr{X})}} \\
	& {\mathscr{Y}(0)} & {\bicoeq{(\mathscr{Y})}}
	\arrow["{\bicoeq{(x)}}", from=2-3, to=3-3]
	\arrow["{q_{\mathscr{Y}}}"', from=3-2, to=3-3]
	\arrow["{x_1}"', from=1-1, to=2-1]
	\arrow["{\mathscr{X}(d_1)}"', from=2-1, to=3-2]
	\arrow["{\mathscr{X}(d_0)}", from=1-1, to=1-2]
	\arrow["{q_{\mathscr{X}}}", from=1-2, to=2-3]
	\arrow["{\mathscr{Y}(d_0)}"{description}, from=2-1, to=2-2]
	\arrow["{x_0}"{description}, from=1-2, to=2-2]
	\arrow["{x_{d_0} \atop \simeq}"{description}, draw=none, from=1-1, to=2-2]
	\arrow["{q_\mathscr{Y}}"{description}, from=2-2, to=3-3]
	\arrow["{\overline{x}\atop \simeq}"{description}, draw=none, from=2-2, to=2-3]
	\arrow["{\zeta \atop \simeq}"{description}, draw=none, from=2-2, to=3-2]
\end{tikzcd}\]
\end{remark}

\begin{remark}\label{codescent object for the localization}
To any morphism $ f : A \rightarrow B$ in a 2-category $ \mathcal{C}$ with finite bilimits we can associate a codescent object, its nerve:
\[\begin{tikzcd}
	{f \simeq f} & A \\
	A & B
	\arrow["f"', from=2-1, to=2-2]
	\arrow["f", from=1-2, to=2-2]
	\arrow["d_1"', from=1-1, to=2-1]
	\arrow["d_0", from=1-1, to=1-2]
	\arrow["{\lambda_f \atop \simeq}"{description}, draw=none, from=1-1, to=2-2]
\end{tikzcd}\]

In particular the identity of $A$ induces trivially a common pseudosection of $d_0,d_1$
\[\begin{tikzcd}
	A \\
	& {f \simeq f} & A \\
	& A & B
	\arrow["d_1"{description},from=2-2, to=2-3]
	\arrow["d_0"{description} ,from=2-2, to=3-2]
	\arrow["f"', from=3-2, to=3-3]
	\arrow["f", from=2-3, to=3-3]
	\arrow["{\lambda_f \atop \simeq}"{description}, draw=none, from=2-2, to=3-3]
	\arrow[""{name=0, anchor=center, inner sep=0}, curve={height=-12pt}, Rightarrow, no head, from=1-1, to=2-3]
	\arrow[""{name=1, anchor=center, inner sep=0}, curve={height=12pt}, Rightarrow, no head, from=1-1, to=3-2]
	\arrow["{\iota}"{description}, from=1-1, to=2-2]
	\arrow["\simeq"{description}, Rightarrow, draw=none, from=0, to=2-2]
	\arrow["\simeq"{description}, Rightarrow, draw=none, from=2-2, to=1]
\end{tikzcd}\]

Those data define a simplicial object which can be shown to be codescent 
\[\begin{tikzcd}
	{f \simeq f \simeq f} & {f \simeq f} & A
	\arrow["{d_0}", shift left=2, from=1-2, to=1-3]
	\arrow["{p_0}", shift left=2, from=1-1, to=1-2]
	\arrow["{d_1}"', shift right=2, from=1-2, to=1-3]
	\arrow["\iota"{description}, from=1-3, to=1-2]
	\arrow["{p_1}"', shift right=2, from=1-1, to=1-2]
	\arrow["m"{description}, from=1-1, to=1-2]
\end{tikzcd}\]

In particular, in $\Cat$, given a localization $ q_\Sigma : C \rightarrow C[\Sigma^{-1}]$, one can consider the corresponding codescent object 
\[\begin{tikzcd}
	{q_F \simeq q_F \simeq q_F } & {q_F \simeq q_F} & {C}
	\arrow["{d_0}", shift left=2, from=1-2, to=1-3]
	\arrow["{d_1}"', shift right=2, from=1-2, to=1-3]
	\arrow["\iota"{description}, from=1-3, to=1-2]
	\arrow["{p_0}", shift left=2, from=1-1, to=1-2]
	\arrow["m"{description}, from=1-1, to=1-2]
	\arrow["{p_1}"', shift right=2, from=1-1, to=1-2]
\end{tikzcd}\]
Then $ q_F \simeq q_F$ can be seen as indexing arrows of $\Sigma$ (up to an equivalence) with $ d_0$, $d_1$ the restricted domain and codomain, while $ \iota$ represents the fact that all identities are in particular isomorphisms, hence are in particular inverted by $q_F$. The object $ q_F \simeq q_F \simeq q_F$ is equivalent to the object $ \Sigma \times_C \Sigma  $ of composable pairs and expresses that $\Sigma$ is to be closed under composition.
\end{remark}

\begin{remark}\label{codescent object for the pseudocolimit}
In particular, applying this construction in $\Cat$ to the localization of an oplax pseudocolimit to the corresponding pseudocolimit 
\[\begin{tikzcd}
	{\underset{i \in I}\oplaxpscolim \; F(i)} & {\underset{i \in I}{\pscolim} \; F(i) \simeq \underset{i \in I}\oplaxpscolim \; F(i)[\Cart_F^{-1}]}
	\arrow["{q_F}", from=1-1, to=1-2]
\end{tikzcd}\]
we get an oplax codescent object whose pseudocoequalizer coincides with the localization $q_F : C \rightarrow C[\Sigma^{-1}]$. Moreover, because of the coherence condition in the oplax 2-cell at any 2-cell $ \sigma : d \Rightarrow d' : i \rightarrow j$ in $I$ given by $ q_{d'} q_j*F(\sigma) = q_d $ we know that the $F(\sigma)$ are also inverted by the localization $q_F$ because the class of inverted maps generated from $\Cart_F$ satisfies the 2 out of 3 axiom and $ q_d, q_{d'}$ are inverted. \\

However here we used an already existing bicolimit to construct this codescent object: in this paper, we shall rather construct a bicolimit from oplax bicolimit and bicoequalizers of codescent objects. We want to show that this is a totally general and normal process. We propose here a quite natural categorification of the classical process of constructing colimits from coproducts and coequalizers. \\

This remark generalizes for the case of a $\sigma$-colimit in $\Cat$ for it is obtained also as a localization: for any diagram $F : I \rightarrow \Cat$ and $\Sigma$ a class of maps in $I$, one can compute as above a codescent object at the localization
\[\begin{tikzcd}
	{\underset{i \in I}\oplaxpscolim \; F(i)} & {\underset{i \in I}{\Sigma\pscolim} \; F(i) \simeq \underset{i \in I}\oplaxpscolim \; F(i)[\Sigma_{(F,\Sigma)}^{-1}]}
	\arrow["{q_F}", from=1-1, to=1-2]
\end{tikzcd}\]
where $\Sigma_{(F,\Sigma)}$ denotes the class of cartesian lifts of maps in $\Sigma$ as in \cref{sigmacolim in Cat}.
\end{remark}

\begin{division}[Codescent diagram associated to a marked 2-category]
It is well known that any 2-category induces a 2-truncated simplicial object called its nerve, which is in fact a codescent object. Here we describe the same construction yet with a restriction at the level of arrows to encode a choice of a class as in a $\sigma$-bicolimit. \\

Let be $ \mathcal{C}$ a 2-category, $ I$ a small 2-category, $\Sigma$ a distinguished class of arrows of $I$ containing identities and stable under composition. First, denote as $\Sigma \hookrightarrow I$ the (0,2)-full subcategory of $I$ consisting of all objects and only 1-cells in $\Sigma$ with all 2-cells between them. In the following we want a category correctly indexing those arrows that are in $\Sigma$ but seen as objects. To do so, consider first the arrow 2-category $ [2, I]_{\ps}$ of 2-functors $ 2 \rightarrow I$ from the walking arrow, together with pseudonatural transformations between them (so that morphisms in $[2,I]_{\ps}$ correspond to pseudosquares). Define now the following (1,2)-full sub-2-category $ \Sigma^2 \hookrightarrow [2,I]_{\ps}$ consisting only of those arrows that lay in $\Sigma$. Beware that its morphisms are all pseudosquares of the form 
\[\begin{tikzcd}
	{\dom \, d_0} & {\dom \, d_1} \\
	{\cod \, d_0} & {\cod \, d_1}
	\arrow["{d_0}"', from=1-1, to=2-1]
	\arrow["{\phi_0}", from=1-1, to=1-2]
	\arrow["{d_1}", from=1-2, to=2-2]
	\arrow["{\phi_1}"', from=2-1, to=2-2]
	\arrow["{\phi \atop \simeq}"{description}, draw=none, from=1-1, to=2-2]
\end{tikzcd}\]
where $ \phi_0, \phi_1$ are \emph{not} required to live in $\Sigma$ - though $ d_0, d_1$ are. Its two cells, which are modification of 2-functors, correspond to morphisms of pseudosquares.\\

This 2-category is equipped with a parallel pair $ \dom, \cod : \Sigma^2 \rightrightarrows I$. Moreover, because $ \Sigma$ is stipulated as containing identities, the identity morphism $ \id : I \rightarrow I^2$ admits a factorization $ I \rightarrow \Sigma^2$. This factorization provides a pseudosection of both $ \dom$ and $\cod$. \\

Now the higher data. Let $3$ denote the walking invertible triangular 2-cell
\[\begin{tikzcd}[sep=small]
	& {d_1} \\
	{d_0} && {d_2}
	\arrow["{p_0}", from=2-1, to=1-2]
	\arrow["{p_2}", from=1-2, to=2-3]
	\arrow[""{name=0, anchor=center, inner sep=0}, "{p_1}"', from=2-1, to=2-3]
	\arrow["\phi \atop \simeq"{description}, shorten >=3pt, draw=none, from=1-2, to=0]
\end{tikzcd}\]
We can consider the 2-functor 2-category $[3,I]_{\ps}$, and take again its (1,2)-full sub-2-category $ \Sigma^3$ consisting of all triangular 2-cells in $I$ whose edges are in $\Sigma$ - while morphisms are given by the pseudonatural transformations and 2-cells as modifications. Then we are provided with a canonical codescent object associated to the choice of $\Sigma$
\[\begin{tikzcd}
	{\Sigma^3} & \Sigma^2 & I
	\arrow["{p_0}", shift left=2, from=1-1, to=1-2]
	\arrow["{p_2}"', shift right=2, from=1-1, to=1-2]
	\arrow["{p_1}"{description}, from=1-1, to=1-2]
	\arrow["{\dom}", shift left=2, from=1-2, to=1-3]
	\arrow["{\cod}"', shift right=2, from=1-2, to=1-3]
	\arrow["\id"{description}, from=1-3, to=1-2]
\end{tikzcd}\]
Observe that this codescent object is moreover strict in the sense that the invertible 2-cells in its codescent data are actually equalities $ \dom  \,\id = 1_I = \cod \, \id$, $ \dom \, p_0=\dom \, p_1$, $ \cod\, p_0=\dom \, p_2$ and $ \cod \, p_1=\cod \, p_2$. We are now going to use this codescent diagram to index a canonical codescent diagram over any oplax bicolimit given a choice of marked maps.  
\end{division}

\begin{division}[Codescent data associated to a $\sigma$-bicolimit: 1-dimensional data]\label{Codescent data associated to a bicolimit: 1-dimensional data}

Let be $ \mathcal{C}$ a 2-category, $ I$ a small 2-category, $\Sigma$ a distinguished class of arrows of $I$ containing identities and stable under composition, and $ F : I \rightarrow \mathcal{C}$ a 2-functor. Suppose that oplax bicolimits exist in $\mathcal{C}$ and take the oplax bicolimit $ (q_i : F(i) \rightarrow \oplaxbicolim_I F(i))_{i \in I}$ with $ q_d : q_j F(d) \Rightarrow q_i$ the oplax transition 2-cell at an arrow $ i \in I$: the latter correspond to the cartesian arrows in the $\Cat$-valued case - and we want precisely to formally inverse them. To this purpose, we construct a codescent object formally indexing the oplax 2-cells as follows. \\

Consider the whiskering 
\[\begin{tikzcd}
	{\Sigma^2} && I & {\mathcal{C}}
	\arrow["F", from=1-3, to=1-4]
	\arrow[""{name=0, anchor=center, inner sep=0}, "\dom", curve={height=-12pt}, from=1-1, to=1-3]
	\arrow[""{name=1, anchor=center, inner sep=0}, "\cod"', curve={height=12pt}, from=1-1, to=1-3]
	\arrow[shorten <=3pt, shorten >=3pt, Rightarrow, from=0, to=1]
\end{tikzcd}\]
and take the oplax bicolimit of the composite $F \dom : \Sigma^2 \rightarrow \mathcal{C}$; then in particular for each $ d $ in $\Sigma^2$ we have both an inclusion at $ \dom(d)$ and another one in $\cod(d)$ related by the oplax 2-cell at $d$. In fact this provides us with two distinct oplax cocones over $ F \dom$ toward the oplax bicolimit $ \oplaxbicolim_I F$: \begin{itemize}
    \item the one provided by the data of all the $ (q_{\dom(d)} : F\dom(d) \rightarrow  \oplaxbicolim_I F)_{d \in \Sigma^2}$, where the transition 2-cells are given as follows: for a 1-cell $ \phi : d_0 \rightarrow d_1$ in $\Sigma^2$, that is, for a pseudosquare
\[\begin{tikzcd}
	{\dom ( d_0)} & {\cod ( d_0)} \\
	{\dom ( d_1)} & {\cod ( d_1)}
	\arrow["{\phi_0}"', from=1-1, to=2-1]
	\arrow[""{name=0, anchor=center, inner sep=0}, "{d_0}", from=1-1, to=1-2]
	\arrow["{\phi_1}", from=1-2, to=2-2]
	\arrow[""{name=1, anchor=center, inner sep=0}, "{d_1}"', from=2-1, to=2-2]
	\arrow["\phi \atop \simeq"{description}, shorten <=4pt, shorten >=4pt, draw=none, from=0, to=1]
\end{tikzcd}\]
just take the oplax 2-cell
\[\begin{tikzcd}
	{F(\dom ( d_0))} && {F(\dom ( d_1))} \\
	& {\underset{I}\oplaxbicolim \; F}
	\arrow["{F(\phi_0)}", from=1-1, to=1-3]
	\arrow[""{name=0, anchor=center, inner sep=0}, "{q_{\dom ( d_0)}}"', from=1-1, to=2-2]
	\arrow[""{name=1, anchor=center, inner sep=0}, "{q_{\dom ( d_1)}}", from=1-3, to=2-2]
	\arrow["{q_{\phi_0}}"', shorten <=13pt, shorten >=13pt, Rightarrow, from=1, to=0]
\end{tikzcd}\]
and for a 2-cell in $\Sigma^2$, that is, a natural modification $ \lambda : \phi \Rrightarrow \psi$ corresponding to an equality of 2-cells 
\[\begin{tikzcd}
	{\dom ( d_0)} & {\cod ( d_0)} \\
	{\dom ( d_1)} & {\cod ( d_1)}
	\arrow[""{name=0, anchor=center, inner sep=0}, "{\psi_0}", from=1-1, to=2-1]
	\arrow[""{name=1, anchor=center, inner sep=0}, "{d_1}"', from=2-1, to=2-2]
	\arrow[""{name=2, anchor=center, inner sep=0}, "{d_0}", from=1-1, to=1-2]
	\arrow["{\phi_1}", from=1-2, to=2-2]
	\arrow[""{name=3, anchor=center, inner sep=0}, "{\phi_0}"', curve={height=18pt}, from=1-1, to=2-1]
	\arrow["\psi \atop \simeq"{description}, shorten <=4pt, shorten >=4pt, draw=none, from=2, to=1]
	\arrow["{\lambda_0}", shorten <=4pt, shorten >=4pt, Rightarrow, from=3, to=0]
\end{tikzcd} = \begin{tikzcd}
	{\dom ( d_0)} & {\cod ( d_0)} \\
	{\dom ( d_1)} & {\cod ( d_1)}
	\arrow[""{name=0, anchor=center, inner sep=0}, "{d_1}"', from=2-1, to=2-2]
	\arrow[""{name=1, anchor=center, inner sep=0}, "{d_0}", from=1-1, to=1-2]
	\arrow[""{name=2, anchor=center, inner sep=0}, "{\phi_1}"', from=1-2, to=2-2]
	\arrow["{\phi_0}"', from=1-1, to=2-1]
	\arrow[""{name=3, anchor=center, inner sep=0}, "{\psi_1}", curve={height=-18pt}, from=1-2, to=2-2]
	\arrow["\phi \atop \simeq"{description}, shorten <=4pt, shorten >=4pt, draw=none, from=1, to=0]
	\arrow["{\lambda_1}", shorten <=4pt, shorten >=4pt, Rightarrow, from=2, to=3]
\end{tikzcd}\]
just take the equality between oplax 2-cells provided by the coherence condition of the oplax bicolimiting cocone
\[\begin{tikzcd}[column sep=0.05cm]
	{F(\dom ( d_0))} && {F(\dom ( d_1))} \\
	& {\underset{I}\oplaxbicolim \; F}
	\arrow[""{name=0, anchor=center, inner sep=0}, "{F(\phi_0)}", curve={height=-22pt}, from=1-1, to=1-3]
	\arrow[""{name=1, anchor=center, inner sep=0}, "{F(\psi_0)}"{description}, from=1-1, to=1-3]
	\arrow[""{name=2, anchor=center, inner sep=0}, "{q_{\dom ( d_0)}}"', from=1-1, to=2-2]
	\arrow[""{name=3, anchor=center, inner sep=0}, "{q_{\dom ( d_1)}}", from=1-3, to=2-2]
	\arrow["{F(\lambda_0)}", shorten <=2pt, shorten >=2pt, Rightarrow, from=0, to=1]
	\arrow["{q_{\psi_0}}"', shorten <=13pt, shorten >=13pt, Rightarrow, from=3, to=2]
\end{tikzcd} = \begin{tikzcd}[column sep=0.05cm]
	{F(\dom ( d_0))} && {F(\dom ( d_1))} \\
	& {\underset{I}\oplaxbicolim \; F}
	\arrow["{F(\phi_0)}", from=1-1, to=1-3]
	\arrow[""{name=0, anchor=center, inner sep=0}, "{q_{\dom ( d_0)}}"', from=1-1, to=2-2]
	\arrow[""{name=1, anchor=center, inner sep=0}, "{q_{\dom ( d_1)}}", from=1-3, to=2-2]
	\arrow["{q_{\phi_0}}"', shorten <=13pt, shorten >=13pt, Rightarrow, from=1, to=0]
\end{tikzcd}\]
    \item we also have an oplax cocone provided by the data of the composite $ (q_{\cod(d)} F(d) : F(\dom(d) \rightarrow  \oplaxbicolim_I F)_{d \in \Sigma^2}$ with the transition 2-cells given as follows: at a pseudosquare $ \phi$ take the pasting $ q_{\phi_1} * F d_0 q_{\cod \, d_1 \phi}$ as depicted below 
\[\begin{tikzcd}
	{F\dom \, d_0} && {F\dom \, d_1} \\
	{F\cod \, d_0} && {F\cod \, d_1} \\
	& {{\underset{I}\oplaxbicolim \; F}}
	\arrow["{Fd_0}"', from=1-1, to=2-1]
	\arrow[""{name=0, anchor=center, inner sep=0}, "{F\phi_0}", from=1-1, to=1-3]
	\arrow["{Fd_1}", from=1-3, to=2-3]
	\arrow[""{name=1, anchor=center, inner sep=0}, "{F\phi_1}"{description}, from=2-1, to=2-3]
	\arrow[""{name=2, anchor=center, inner sep=0}, "{q_{\cod \, d_1}}", from=2-3, to=3-2]
	\arrow[""{name=3, anchor=center, inner sep=0}, "{q_{\cod \, d_0}}"', from=2-1, to=3-2]
	\arrow["{q_{\phi_1}}"{description}, shorten <=13pt, shorten >=13pt, Rightarrow, from=2, to=3]
	\arrow["{\phi \atop \simeq}"{description}, draw=none, from=0, to=1]
\end{tikzcd}\]
and for a morphism of pseudosquares consider again the equalities provided by the coherence conditions.


\end{itemize}
Then, one can also compute the oplax bicolim 
\[( \begin{tikzcd}
	{F\dom(d)} & {\underset{\Sigma^2}\oplaxbicolim \; F \dom}
	\arrow["{q^{\dom}_{d}}", from=1-1, to=1-2]
\end{tikzcd} )_{d \in \Sigma^2} \]
 with transition 2-cells given in the similar manner as above - they are just restriction of the oplax cocone along the $\dom$ functor. Then the oplax cocones above provides us with a parallel pair 
\[\begin{tikzcd}[sep=huge]
	{\underset{\Sigma^2}\oplaxbicolim \; F \dom } & {\underset{I}\oplaxbicolim \; F}
	\arrow["{\langle q_{\dom(d)} \rangle_{d \in \Sigma^2}}", shift left=2, from=1-1, to=1-2]
	\arrow["{q_{\cod(d)\langle F(d)} \rangle_{d \in \Sigma^2}}"', shift right=2, from=1-1, to=1-2]
\end{tikzcd}\]

Moreover, from the factorization $ I \rightarrow \Sigma^2$ which provides us at each object $ i$ with an arrow $ \id_i$ and this induces also an oplax cocone $ (q_{\dom\;  \id_i} : F(i) \rightarrow \oplaxbicolim_{\Sigma^2} F \dom)_{i \in I}$. Altogether those data induce two parallel arrows and a common section: those are the desired lower data of a codescent object. 
\end{division}

\begin{division}[Codescent data associated to a $\sigma$-bicolimit: higher data]
We can complete the construction above to incorporate higher data attached to the oplax bicolimit. Recall we defined the object 3 as the walking triangular 2-cell, and used it to describe the object of triangular 2-cells in $\Sigma$: its object are of the form 
\[\begin{tikzcd}[sep=small]
	& {\cod \, p_0\phi = \dom \, p_2\phi} \\
	{\dom \, p_0\phi = \dom \, p_1\phi} && {\cod \, p_2\phi = \cod \, p_1\phi}
	\arrow["{p_0\phi}", from=2-1, to=1-2]
	\arrow["{p_2\phi}", from=1-2, to=2-3]
	\arrow[""{name=0, anchor=center, inner sep=0}, "{p_1\phi}"', from=2-1, to=2-3]
	\arrow["\phi \atop \simeq"', shorten >=3pt, draw =none, from=1-2, to=0]
\end{tikzcd}\]
We can consider the composite 
$\begin{tikzcd}
	{\Sigma^3} & \Sigma^2 & I & {\mathcal{C}}
	\arrow["{p_0}", from=1-1, to=1-2]
	\arrow["{d_0}", from=1-2, to=1-3]
	\arrow["F", from=1-3, to=1-4]
\end{tikzcd}$ and compute the oplax bicolimit $ \oplaxbicolim_{\Sigma^3} \, F \dom \, p_0$ of this functor. This will be the higher object in the codescent diagram. The three parallel 1-cells are obtained from the following oplax cocones:\begin{itemize}
    \item for each $\phi$ in $\Sigma^3$ we have an oplax cocone obtained by precomposing the oplax cocone over $F \dom$ along $p_0$
    \[( \begin{tikzcd}
	{F\dom \, p_0\phi} && {\underset{\Sigma^2}\oplaxbicolim \; F \dom}
	\arrow["{q^{\dom}_{p_0\phi}}", from=1-1, to=1-3]
\end{tikzcd} )_{\sigma \in \Sigma^3}\]
    \item we have the oplax cocone provided by the composite 
\[(\begin{tikzcd}
	{F\dom \, p_0\phi} & {F\cod \, p_0\phi = F\dom \, p_2\phi} & {\underset{\Sigma^2}\oplaxbicolim \; F \dom}
	\arrow["{q^{\dom}_{p_2\phi}}", from=1-2, to=1-3]
	\arrow["{F(p_0\phi)}", from=1-1, to=1-2]
\end{tikzcd})_{\sigma \in \Sigma^3}\]
    \item and we have a last oplax cocone:
\[(\begin{tikzcd}
	{F\dom \, p_0\phi = F\dom \, p_1\phi} & {\underset{\Sigma^2}\oplaxbicolim \; F \dom}
	\arrow["{q^{\dom}_{p_1\phi}}", from=1-1, to=1-2]
\end{tikzcd})_{\sigma \in \Sigma^3}\]
\end{itemize}

Those three oplax cocones over $ F \dom \, p_0$ provide us with the desired three parallel 1-cells of the higher data. 
\end{division}

\begin{lemma}\label{codescent object at the oplax colimit}
The following data define a codescent object $ \mathscr{X}_{F,\Sigma}: \mathbb{X} \rightarrow \mathcal{C}$
\[\begin{tikzcd}[sep=1.2cm]
	{ \underset{\Sigma^3}\oplaxbicolim \; F\dom \, p_0} &&&[-12pt] { \underset{\Sigma^2}\oplaxbicolim\;  F\dom} && { \underset{I}{\oplaxbicolim}\;  F}
	\arrow["{\langle q_{\dom(d)} \rangle_{d \in \Sigma^2}}", shift left=5, from=1-4, to=1-6]
	\arrow["{\langle q_{\cod(d)} F(d) \rangle_{d \in \Sigma^2}}"', shift right=5, from=1-4, to=1-6]
	\arrow["{\langle q_{\dom \, \id_i} \rangle_{i \in I}}"{description}, from=1-6, to=1-4]
	\arrow["{\langle q^{\dom}_{p_0\phi} \rangle_{\phi \in \Sigma^3}}", shift left=5, from=1-1, to=1-4]
	\arrow["{\langle  q^{\dom}_{p_2\phi}F(p_0\phi)  \rangle_{\phi \in \Sigma^3}  }"', shift right=5, from=1-1, to=1-4]
	\arrow["{ \langle  q^{\dom}_{p_1\phi}   \rangle_{\phi \in \Sigma^3}}"{description}, from=1-1, to=1-4]
\end{tikzcd}\]
\end{lemma}

\begin{proof}

The pseudosection 2-cells corresponding to the $n_0$ and $n_1$ of \cref{codescent diagram} are induced from the universal property of the oplax bicolimit from the existence in each $i$ of the following decomposition of the inclisions - using that $ i= \dom \, \id_i = \cod \, \id_i$
\[\begin{tikzcd}[column sep=small]
	{F(i)=F\dom \, (\id_i)} & {\underset{\Sigma^2}\oplaxbicolim\;  F\dom} \\
	& {\underset{I}{\oplaxbicolim}\;  F}
	\arrow["{q^{\dom}_{\id_i}}", from=1-1, to=1-2]
	\arrow["{\langle q_{\dom(d)} \rangle_{d \in \Sigma^2}}", from=1-2, to=2-2]
	\arrow[""{name=0, anchor=center, inner sep=0}, "{q_{\dom \, \id_i}}"', shift right=2, shorten >=10pt, from=1-1, to=2-2]
	\arrow["{\theta_{\id_i} \atop \simeq}"{description, pos=0.1}, Rightarrow, draw=none, from=1-2, to=0]
\end{tikzcd}
\hskip0.5cm
\begin{tikzcd}[column sep=small]
	{F(i)=F\dom \, (\id_i)} & {\underset{\Sigma^2}\oplaxbicolim\;  F\dom} \\
	{F(i)=F\cod \, (\id_i)} & {\underset{I}{\oplaxbicolim}\;  F}
	\arrow[""{name=0, anchor=center, inner sep=0}, "{q^{\dom}_{\id_i}}", from=1-1, to=1-2]
	\arrow["{\langle q_{\dom(d)} \rangle_{d \in \Sigma^2}}", from=1-2, to=2-2]
	\arrow[""{name=1, anchor=center, inner sep=0}, "{q_{\cod \, \id_i}}"', from=2-1, to=2-2]
	\arrow["{F(i)=\id_{F(i)}}"', Rightarrow, no head, from=1-1, to=2-1]
	\arrow["{\theta'_{\id_i} \atop \simeq}"{description}, Rightarrow, draw=none, from=0, to=1]
\end{tikzcd}\]
For the higher condition: we have at each $\phi$ of $\Sigma^3$ an invertible 2-cell
\[\begin{tikzcd}[sep=small]
	& {F\dom \,p_0\phi= F\dom \,p_1\phi} \\
	{\underset{\Sigma^2}\oplaxbicolim\;  F\dom} && {\underset{\Sigma^2}\oplaxbicolim\;  F\dom} \\
	& {\underset{I}{\oplaxbicolim}\;  F}
	\arrow["{q^{\dom}_{p_0\phi}}"', from=1-2, to=2-1]
	\arrow["{\langle q_{\dom(d)} \rangle_{d \in \Sigma^2}}"', from=2-1, to=3-2]
	\arrow["{q^{\dom}_{p_1\phi}}", from=1-2, to=2-3]
	\arrow["{\langle q_{\dom(d)} \rangle_{d \in \Sigma^2}}", from=2-3, to=3-2]
	\arrow[""{name=0, anchor=center, inner sep=0}, "{q_{\dom \,p_0\phi } =q_{\dom \,p_1\phi }}"{description}, from=1-2, to=3-2]
	\arrow["{\theta'^{-1}_{p_1\phi} \atop \simeq}"{description, pos=0.7}, Rightarrow, draw=none, from=0, to=2-3]
	\arrow["{\theta'_{p_0\phi} \atop \simeq}"{description, pos=0.3}, Rightarrow, draw=none, from=2-1, to=0]
\end{tikzcd}\]
which induces, by the universal property of the oplaxbicolimit, the desired invertible 2-cell 
\[\begin{tikzcd}[sep=small]
	& {\underset{\Sigma^3}\oplaxbicolim\;  F\dom \, p_0} \\
	{\underset{\Sigma^2}\oplaxbicolim\;  F\dom} && {\underset{\Sigma^2}\oplaxbicolim\;  F\dom} \\
	& {\underset{I}{\oplaxbicolim}\;  F}
	\arrow["{\langle q^{\dom}_{p_0\phi} \rangle_{\phi \in \Sigma^3}}"', from=1-2, to=2-1]
	\arrow["{\langle q_{\dom(d)} \rangle_{d \in \Sigma^2}}"', from=2-1, to=3-2]
	\arrow["{\langle q^{\dom}_{p_1\phi} \rangle_{\phi \in \Sigma^3}}", from=1-2, to=2-3]
	\arrow["{\langle q_{\dom(d)} \rangle_{d \in \Sigma^2}}", from=2-3, to=3-2]
	\arrow["{\theta_{01} \atop \simeq}"{description}, draw=none, from=1-2, to=3-2]
\end{tikzcd}\]

Similarly the middle 2-cell $ \theta_{02}$ is induced from the pasting 
\[\begin{tikzcd}
	& {F\dom\, p_0\phi} & {F\cod\,p_0\phi = F \dom \, p_2\phi} \\
	{\underset{\Sigma^2}\oplaxbicolim\;  F\dom} &&& {\underset{\Sigma^2}\oplaxbicolim\;  F\dom} \\
	&& {\underset{I}\oplaxbicolim\;  F}
	\arrow["{Fp_0\phi}", from=1-2, to=1-3]
	\arrow["{q^{\dom}_{p_0\phi}}"', from=1-2, to=2-1]
	\arrow["{q^{\dom}_{p_2\phi}}", from=1-3, to=2-4]
	\arrow["{\langle q_{\dom(d)} \rangle_{d \in \Sigma^2}}", from=2-4, to=3-3]
	\arrow["{\langle q_{\cod(d)F(d)} \rangle_{d \in \Sigma^2}}"', from=2-1, to=3-3]
	\arrow[""{name=0, anchor=center, inner sep=0}, "{q_{\dom\, p_1}}"{description}, from=1-3, to=3-3]
	\arrow["{\theta_{p_2\phi}^{-1} \atop \simeq}"{description}, Rightarrow, draw=none, from=0, to=2-4]
	\arrow["{\theta'_{p_0\phi} \atop \simeq}"{description}, Rightarrow, draw=none, from=2-1, to=0]
\end{tikzcd}\]


Finally the last invertible 2-cell $ \theta_{12}$ is induced from the pasting 
\footnotesize{\[\begin{tikzcd}[column sep=tiny]
	& {F\dom\, p_0\phi = F\dom \, p_1\phi} && {F\cod\,p_0\phi = F \dom \, p_2\phi} \\
	{\underset{\Sigma^2}\oplaxbicolim\;  F\dom} && {F\cod \, p_1\phi = F\cod \, p_2\phi} && {\underset{\Sigma^2}\oplaxbicolim\;  F\dom} \\
	&& {\underset{I}\oplaxbicolim\;  F}
	\arrow[""{name=0, anchor=center, inner sep=0}, "{Fp_0\phi}", from=1-2, to=1-4]
	\arrow["{q^{\dom}_{p_1\phi}}"', from=1-2, to=2-1]
	\arrow["{q^{\dom}_{p_2\phi}}", from=1-4, to=2-5]
	\arrow[""{name=1, anchor=center, inner sep=0}, "{\langle q_{\cod(d)F(d)} \rangle_{d \in \Sigma^2}}"', from=2-1, to=3-3]
	\arrow[""{name=2, anchor=center, inner sep=0}, "{\langle q_{\cod(d)F(d)} \rangle_{d \in \Sigma^2}}", from=2-5, to=3-3]
	\arrow["{Fp_1\phi}"{description}, from=1-2, to=2-3]
	\arrow["{Fp_2\phi}"{description}, from=1-4, to=2-3]
	\arrow["{q_{\cod \, p_1\phi }= q_{\cod \, p_2\phi}}"{description}, from=2-3, to=3-3]
	\arrow["{F\phi \atop \simeq}"{description}, Rightarrow, draw=none, from=0, to=2-3]
	\arrow["{\theta'_{p_0\phi} \atop \simeq}"{description}, Rightarrow, draw=none, from=1-2, to=1]
	\arrow["{\theta'^{-1}_{p_2\phi} \atop \simeq}"', draw=none, from=1-4, to=2]
\end{tikzcd}\]}
\end{proof}

\begin{remark}
In the following we are going to compare the bicolimit of the diagram $ F$ and the bicoequalizer of the codescent object $ \mathscr{X}_{F,\Sigma}$. Before embarking in this task, we should compare our construction with the canonical codescent object constructed at the localization $ q_{F,\Sigma} : \oplaxcolim_I F \rightarrow \Sigma \pscolim_I F$ in the case of a diagram valuated in $\Cat$ in order to check this construction subsumes the canonical one. Take $ F : I \rightarrow \Cat$ and $\Sigma$ in $I^2$. Let us compare the data of $\mathscr{X}_{F,\Sigma}$ with the canonical codescent object at $q_{F,\Sigma}$ defined at \cref{codescent object for the pseudocolimit}. \\

The object $ q_{F,\Sigma} \simeq q_{F,\Sigma}$, being the pseudopullback of $q_{F,\Sigma}$ at itself, can be described concretely as follows. Its objects are data $((i,a),(j,b),(d,\alpha))$ with $ a $ an object of $F(i)$, $b$ an object of $F(j)$, $s : i \rightarrow j$ a morphism in $\Sigma$ and $ \alpha : F(d)(a) \simeq b$ an isomorphism in $F(j)$, so that the pair $ (d, \alpha)$ becomes an isomorphism in $\Sigma\colim_I F$ as depicted below 
\[\begin{tikzcd}[row sep=small]
	& {F(i)} \\
	{*} && {\Sigma\underset{I}\colim F} \\
	& {F(j)}
	\arrow["a", from=2-1, to=1-2]
	\arrow["{q_i}", from=1-2, to=2-3]
	\arrow["b"', from=2-1, to=3-2]
	\arrow["{q_j}"', from=3-2, to=2-3]
	\arrow[""{name=0, anchor=center, inner sep=0}, "{F(d)}"{description}, from=1-2, to=3-2]
	\arrow["{\alpha \atop \simeq}"{description}, draw=none, from=2-1, to=0]
	\arrow["{q_d \atop \simeq}"{description}, draw=none, from=0, to=2-3]
\end{tikzcd}\]
A morphism $((i_0,a_0),(j_0,b_0, (s_0,\alpha_0)) \rightarrow ((i_1,a_1),(j_1,b_1, (s_1,\alpha_1))$ consists of a pair of morphisms $(\phi, u) : (i_0, a_0) \rightarrow (i_1, a_1)$, $(\psi, v) : (j_0, b_0) \rightarrow (j_1, b_1)$ in $\oplaxcolim_I F$ ensuring the following commutation in the $\sigma$-colimit (where, beware, $ \phi$ and $\psi$ are arbitrary arrows of $I$):
\[\begin{tikzcd}
	{(i_0,a_0)} & {(j_0,b_0)} \\
	{(i_1,a_1)} & {(j_1,b_1)}
	\arrow["{(d_0,\alpha_0) \atop \simeq}", from=1-1, to=1-2]
	\arrow["{(\phi,u)}"', from=1-1, to=2-1]
	\arrow["{(\psi,v)}", from=1-2, to=2-2]
	\arrow["{(d_1,\alpha_1) \atop \simeq}"', from=2-1, to=2-2]
\end{tikzcd}\]

Hence we have a functor $ \oplaxcolim_{\Sigma^2} F\dom \rightarrow q_{F,\Sigma} \simeq q_{F,\Sigma} $ sending $(d, a)$ (with $d$ in $\Sigma$) to the data $(\dom \, d, a),(\cod\, d, F(d)(a)), (d, 1_{F(d)(a)})) $ and a morphism $ (\phi, u) :(d_0,a_0) \rightarrow (d_1,a_1)$ (with $ \phi$ a pseudosquare and $ u : a_1 \rightarrow F(\phi_0)(a_0)$) to the pair $ (\phi_0, u),(\phi_1, \phi F(d_1)(u))$. Conversely, any data $ ((i,a), (j,b), (d,\alpha))$, for $ d : i \rightarrow j$ has to be in $\Sigma$, can be sent to the pair $(d,a)$ in $\oplaxcolim_{\Sigma^2} F\dom$, while a morphism $(\phi,u),(\psi,v)$ to the pair consisting of the strict underlying pseudosquare $ d_1 \phi = \psi d_0 $ together with the arrow $u$ in $F(\dom \, d_1)$. \\

Those two functors can be shown to define an equivalence: indeed one can trivially recover any $(d,a)$ in $\oplaxcolim_{\Sigma^2} F\dom$, from $(\dom \, d, a),(\cod\, d, F(d)(a)), (d, 1_{F(d)(a)})) $ up to equality, while on the other side, one has a canonical isomorphism $((1_i, 1_a), (1_j, \alpha)): ((i,a),(j,b),(d, \alpha)) \simeq ((i,a), (j, F(a)), (d, 1_{F(a)}))$ in $q_{F,\Sigma} \simeq q_{F,\Sigma}$. This ensures the existence of a retract equivalence
\[  \oplaxcolim_{\Sigma^2} F\dom \simeq (q_{F,\Sigma} \simeq q_{F,\Sigma}) \]

Similarly, one can check that the higher data are also equivalent, that is, $ \oplaxcolim_{\Sigma^3} F\dom \, p_0 $ is equivalent to $q_{F,\Sigma}  \simeq q_{F,\Sigma} \simeq q_{F,\Sigma}  $. The induced morphisms between the higher data are also expected to coincide, which may be shown using the universal properties of the different objects involved here. Hence, being equivalent, the two codescent objects considered above will have the same pseudocoequalizers up to canonical equivalences, which comfort us in the intuition that the pseudocoequalizer of the codescent diagram $ \mathscr{X}_{F,\Sigma}$ has to be the $\sigma$-colimit for it will coincide with the convenient localization in $\Cat$. We are going to prove directly why it is so in an arbitrary 2-category. 
\end{remark}

\begin{division}
We now get back to the general setting, with $ F : I \rightarrow \mathcal{C} $ a diagram and $\Sigma$ a class of maps in $I$. Whenever it exists, the bicoequalizer is equipped with an oplax cocone over $F$ given by the composite 
\[\begin{tikzcd}
	{F(i) } & {\underset{I}\oplaxbicolim \; F } & {\bicoeq(\mathscr{X}_{F,\Sigma})}
	\arrow["{q_i}", from=1-1, to=1-2]
	\arrow["{q_{F,\Sigma}}", from=1-2, to=1-3]
\end{tikzcd}\]
together with the whiskering of oplax 2-cells $ q_{F,\Sigma} * q_d$ at each $d$ in $I^2$. Moreover the bicoequalizer inserts a canonical invertible 2-cell satisfying the coherence conditions of \cref{bicoequalization of codescent diagram}. 
\[\begin{tikzcd}
	{\underset{\Sigma}\oplaxbicolim \; F \dom} && {\underset{I}\oplaxbicolim \; F } \\
	{\underset{I}\oplaxbicolim \; F } && {\bicoeq(\mathscr{X}_{F,\Sigma})}
	\arrow["{\langle q_{\dom(d)} \rangle_{d \in I^2}}", from=1-1, to=1-3]
	\arrow["{q_{F,\Sigma}}", from=1-3, to=2-3]
	\arrow["{\langle q_{\cod(d)} F(d) \rangle_{d \in I^2}}"', from=1-1, to=2-1]
	\arrow["{q_{F,\Sigma}}"', from=2-1, to=2-3]
	\arrow["{\xi_{F,\Sigma} \atop \simeq}"{description}, draw=none, from=1-1, to=2-3]
\end{tikzcd}\]

\end{division}

\begin{proposition}\label{bicolim is a bicoeq}
Whenever it exists, the bicoequalizer of $ \mathscr{X}_{F,\Sigma}$ is a $\sigma$-bicolimit of ${F}$:
\[  \bicoeq(\mathscr{X}_{F,\Sigma}) \simeq \Sigma\underset{I}\bicolim \; F  \]
\end{proposition}

\begin{proof}
Recall that the bicoequalizer of a codescent object was defined as being universal amongst those that pseudocoequalizes the codescent object as in \cref{bicoequalization of codescent diagram}. We prove here that, for a given object $C$ of $\mathcal{C}$, the category of pairs $(q :\oplaxbicolim_I F \rightarrow C, \xi)$ pseudocoequalizig the codescent object $ \mathscr{X}_{F,\Sigma}$ and the category of $\Sigma$-cocones over $ F$ with tip $C$ are pseudonaturally equivalent. \\

Suppose we have a 1-cell $ q : \oplaxbicolim_I F \rightarrow C$ together with an invertible 2-cell
\[\begin{tikzcd}
	{\underset{\Sigma}\oplaxbicolim \; F \dom} && {\underset{I}\oplaxbicolim \; F } \\
	{\underset{I}\oplaxbicolim \; F } && C
	\arrow["{\langle q_{\dom(d)} \rangle_{d \in I^2}}", from=1-1, to=1-3]
	\arrow["q", from=1-3, to=2-3]
	\arrow["{\langle q_{\cod(d)} F(d) \rangle_{d \in I^2}}"', from=1-1, to=2-1]
	\arrow["q"', from=2-1, to=2-3]
	\arrow["{\xi \atop \simeq}"{description, pos=0.3}, draw=none, from=1-1, to=2-3]
\end{tikzcd}\]
satisfying moreover the required coherence conditions. Then it induces an oplax cocone $ F \Rightarrow C$ with component $ qq_i : F(i) \rightarrow C$ at $i$ and the whiskering $ q*q_d$ at $d$ in $I^2$. But now, in the case where $d$ is in $\Sigma$, the whiskering above happens to coincide with the following pasting 
\[\begin{tikzcd}[sep=large]
	{F\dom(d)} \\
	{F\cod(d)} & {\underset{\Sigma^2}\oplaxbicolim \; F \dom} & {\underset{I}\oplaxbicolim \; F} \\
	& {\underset{I}\oplaxbicolim \; F} & C
	\arrow["{F(d)}"', from=1-1, to=2-1]
	\arrow[""{name=0, anchor=center, inner sep=0}, "{q^{\dom}_d}"{description}, from=1-1, to=2-2]
	\arrow[""{name=1, anchor=center, inner sep=0}, "{q_{\cod(d)}}"', from=2-1, to=3-2]
	\arrow["{\langle q_{\cod(d)} F(d) \rangle_{d \in \Sigma^2}}"{description, pos=0.4}, from=2-2, to=3-2]
	\arrow["{\langle q_{\dom(d)} \rangle_{d \in I^2}}"', from=2-2, to=2-3]
	\arrow[""{name=2, anchor=center, inner sep=0}, "{q_{\dom(d)}}", curve={height=-12pt}, from=1-1, to=2-3]
	\arrow["q", from=2-3, to=3-3]
	\arrow["q"', from=3-2, to=3-3]
	\arrow["{\xi \atop \simeq}"{description}, draw=none, from=2-2, to=3-3]
	\arrow["{\theta'_d \atop \simeq}"{description, pos=0.6}, Rightarrow, draw=none, from=0, to=1]
	\arrow["{\theta_d \atop \simeq}"{description}, Rightarrow, draw=none, from=2-2, to=2]
\end{tikzcd}\]
which is an invertible 2-cell: this forces the oplax transition 2-cells of the induced oplax cocone to be invertible, hence this cocone to be actually a $\sigma$-cocone relative to $\Sigma$.\\

Conversely, if one has a $\Sigma$-cocone $ (q'_i : F(i) \rightarrow C)_{i \in I} $; it is in particular a oplax cocone, for which it admits a universal factorization through $ \oplaxbicolim_I F$
\[\begin{tikzcd}
	{F(i)} && C \\
	& {\underset{I}\oplaxbicolim \; F }
	\arrow[""{name=0, anchor=center, inner sep=0}, "{q'_i}", from=1-1, to=1-3]
	\arrow["{q_i}"', from=1-1, to=2-2]
	\arrow["{\langle q_i' \rangle_{i \in I}}"', from=2-2, to=1-3]
	\arrow["{\tau_i \atop \simeq}"{description}, Rightarrow, draw=none, from=0, to=2-2]
\end{tikzcd}\]
Now for each distinguished arrow $ d$ of $\Sigma$ we have a decomposition of the oplax 2-cell $q'_d$: 
\[\begin{tikzcd}
	{F({\dom(d)})} \\
	& {\underset{I}\oplaxbicolim \; F } && C \\
	{F({\cod(d)})}
	\arrow[""{name=0, anchor=center, inner sep=0}, "{q'_{\dom(d)}}", curve={height=-16pt}, from=1-1, to=2-4]
	\arrow[""{name=1, anchor=center, inner sep=0}, end anchor=166, "{q_{\dom(d)}}"{description}, from=1-1, to=2-2]
	\arrow["{\langle q_i' \rangle_{i \in I}}"{description}, from=2-2, to=2-4]
	\arrow["{F(d)}"', from=1-1, to=3-1]
	\arrow[""{name=2, anchor=center, inner sep=0}, "{q_{\cod(d)}}"{description}, from=3-1, to=2-2]
	\arrow[""{name=3, anchor=center, inner sep=0}, "{q'_{\cod(d)}}"', curve={height=16pt}, from=3-1, to=2-4]
	\arrow["{\tau_{\dom(d)} \atop \simeq}"{description}, Rightarrow, draw=none, from=0, to=2-2]
	\arrow["{\tau_{\cod(d)} \atop \simeq}"{description, pos=0.44}, Rightarrow, draw=none, from=2-2, to=3]
	\arrow["{q_d}", shorten <=4pt, shorten >=4pt, Rightarrow, from=2, to=1]
\end{tikzcd}\]
\[=\begin{tikzcd}
	& {F({\dom(d)})} \\
	& {\underset{\Sigma}\oplaxbicolim \; F \dom } \\
	{F({\dom(d)})} && {\underset{I}\oplaxbicolim \; F } && C \\
	& {\underset{\Sigma}\oplaxbicolim \; F \dom } \\
	& {F({\cod(d)})}
	\arrow["{\langle q_i' \rangle_{i \in I}}"{description}, from=3-3, to=3-5]
	\arrow["{F(d)}"', curve={height=12pt}, from=3-1, to=5-2]
	\arrow["{q_{\cod(d)}}"{description}, curve={height=12pt}, from=5-2, to=3-3]
	\arrow[""{name=0, anchor=center, inner sep=0}, "{q'_{\cod(d)}}"', curve={height=24pt}, from=5-2, to=3-5]
	\arrow["{\langle q_{\dom(d)} \rangle_{d \in \Sigma^2}}"{description}, from=2-2, to=3-3]
	\arrow["{q^{\dom}_{d}}"{description}, from=3-1, to=2-2]
	\arrow["{q^{\dom}_{d}}"{description}, from=3-1, to=4-2]
	\arrow["{\langle q_{\cod(d)} F(d) \rangle_{d \in \Sigma^2}}"{description}, from=4-2, to=3-3]
	\arrow[""{name=1, anchor=center, inner sep=0}, "{q'_{\dom(d)}}", shift left=2, curve={height=-30pt}, from=1-2, to=3-5]
	\arrow[curve={height=12pt}, Rightarrow, no head, from=1-2, to=3-1]
	\arrow["{q_{\dom(d)}}"{description}, curve={height=-12pt}, from=1-2, to=3-3]
	\arrow["{\theta_d q_d \theta'^{-1}_d}"{description}, shorten <=9pt, shorten >=9pt, Rightarrow, from=4-2, to=2-2]
	\arrow["{\theta'_d \atop \simeq}"{description}, draw=none, from=5-2, to=4-2]
	\arrow["{\theta_d^{-1} \atop \simeq}"{description}, draw=none, from=2-2, to=1-2]
	\arrow["{\tau_{\cod(d)} \atop \simeq}"{description}, Rightarrow, draw=none, from=3-3, to=0]
	\arrow["{\tau_{\dom(d)} \atop \simeq}"{description}, Rightarrow, draw=none, from=1, to=3-3]
\end{tikzcd}\]
But for $q'_d$ is invertible as $d$ is in $\Sigma$, the decomposition above forces by cancellation of invertible 2-cells the whiskering $ \langle q'_i \rangle_{i \in I} * \theta_d q_d \theta'^{-1}_d $ to be itself invertible. But then by universal property of the oplax bicolimit, this invertible 2-cell induces uniquely an invertible 2-cell $\langle q_{\cod(d)} F(d) \rangle_{d \in \Sigma^2}$ with the property that for each $d$ in $\Sigma$ we have an equality of whiskerings:
\[\begin{tikzcd}[row sep=small]
	& {\underset{\Sigma^2}\oplaxbicolim \; F \dom } \\
	{F({\dom(d)})} && {\underset{I}\oplaxbicolim \; F } && C \\
	& {\underset{\Sigma^2}\oplaxbicolim \; F \dom }
	\arrow["{\langle q_i' \rangle_{i \in I}}", from=2-3, to=2-5]
	\arrow["{\langle q_{\dom(d)} \rangle_{d \in \Sigma^2}}", from=1-2, to=2-3]
	\arrow["{q^{\dom}_{d}}", from=2-1, to=1-2]
	\arrow["{q^{\dom}_{d}}"', from=2-1, to=3-2]
	\arrow["{\langle q_{\cod(d)} F(d) \rangle_{d \in \Sigma^2}}"', from=3-2, to=2-3]
	\arrow["{\theta_d q_d \theta'^{-1}_d}", shorten <=9pt, shorten >=9pt, Rightarrow, from=3-2, to=1-2]
\end{tikzcd}\]
\[=\begin{tikzcd}[row sep=small]
	&& {\underset{I}\oplaxbicolim \; F } \\
	{F({\dom(d)})} & {\underset{\Sigma^2}\oplaxbicolim \; F \dom } &&& C \\
	&& {\underset{I}\oplaxbicolim \; F }
	\arrow["{\langle q_i' \rangle_{i \in I}}", from=1-3, to=2-5]
	\arrow["{\langle q_{\dom(d)} \rangle_{d \in \Sigma^2}}", from=2-2, to=1-3]
	\arrow["{q^{\dom}_{d}}", from=2-1, to=2-2]
	\arrow["{\langle q_i' \rangle_{i \in I}}"', from=3-3, to=2-5]
	\arrow["{\langle \theta_d q_d \theta'^{-1}_d \rangle_{d \in \Sigma^2} \atop \simeq}"{description}, Rightarrow, draw=none, from=3-3, to=1-3]
	\arrow["{\langle q_{\cod(d)} F(d) \rangle_{d \in \Sigma^2}}"', from=2-2, to=3-3]
\end{tikzcd}\]

We must show this inserted 2-cell satisfies the coherence conditions. For the lower condition, we use pseudofunctoriality of the oplax bicolimit construction. Observe that the identity of $ \langle q'_i \rangle_{i \in I}$ is jointly induced by the identities $ 1_{  \langle q'_i \rangle_{i \in I}} * q_i  : \langle q'_i \rangle_{i \in I} q_i  =\langle q'_i \rangle_{i \in I} q_i  $. But if we consider the whiskering of the following diagram
\[\begin{tikzcd}[sep=large]
	{\underset{I}\oplaxbicolim \; F } \\
	& {\underset{\Sigma^2}\oplaxbicolim \; F \dom} & {\underset{I}\oplaxbicolim \; F } \\
	& {\underset{I}\oplaxbicolim \; F } & C
	\arrow["{\langle q_{\dom \, \id_i} \rangle_{i \in I}}"{description}, from=1-1, to=2-2]
	\arrow[""{name=0, anchor=center, inner sep=0}, curve={height=-18pt}, Rightarrow, no head, from=1-1, to=2-3]
	\arrow[""{name=1, anchor=center, inner sep=0}, curve={height=26pt}, Rightarrow, no head, from=1-1, to=3-2]
	\arrow["{\langle q_{\cod(d)} F(d) \rangle_{d \in \Sigma^2}}"{description}, from=2-2, to=3-2]
	\arrow["{\langle q_{\dom(d)} \rangle_{d \in \Sigma^2}}"', from=2-2, to=2-3]
	\arrow["{\langle q_i' \rangle_{i \in I}}"', from=3-2, to=3-3]
	\arrow["{\langle q_i' \rangle_{i \in I}}", from=2-3, to=3-3]
	\arrow["{\langle \theta_d q_d \theta'^{-1}_d \rangle_{d \in \Sigma^2} \atop \simeq}"{description}, draw=none, from=2-2, to=3-3]
	\arrow["{\langle\theta^\dom_{\id_i} \rangle_{i \in I}}"{description}, Rightarrow, draw=none, from=2-2, to=0]
	\arrow["{\langle\theta^\cod_{\id_i} \rangle_{i \in I}}"{description, pos=0.2}, Rightarrow, draw=none, from=2-2, to=1]
\end{tikzcd}\]
with each of the projections $ q_i$ we get the pasting below:
\[\begin{tikzcd}[sep=large]
	{F(i)} \\
	& {\underset{\Sigma^2}\oplaxbicolim \; F \dom} & {\underset{I}\oplaxbicolim \; F } \\
	& {\underset{I}\oplaxbicolim \; F } & C
	\arrow["{q_{\dom \, \id_i} }"{description}, from=1-1, to=2-2]
	\arrow[""{name=0, anchor=center, inner sep=0}, "{q_i}", curve={height=-18pt}, from=1-1, to=2-3]
	\arrow[""{name=1, anchor=center, inner sep=0}, "{q_i}"', curve={height=24pt}, from=1-1, to=3-2]
	\arrow["{\langle q_{\cod(d)} F(d) \rangle_{d \in \Sigma^2}}"{description}, from=2-2, to=3-2]
	\arrow["{\langle q_{\dom(d)} \rangle_{d \in \Sigma^2}}"', from=2-2, to=2-3]
	\arrow["{\langle q_i' \rangle_{i \in I}}"', from=3-2, to=3-3]
	\arrow["{\langle q_i' \rangle_{i \in I}}", from=2-3, to=3-3]
	\arrow["{\langle \theta_d q_d \theta'^{-1}_d \rangle_{d \in \Sigma^2} \atop \simeq}"{description}, draw=none, from=2-2, to=3-3]
	\arrow["{\theta^\dom_{\id_i} }"{description}, Rightarrow, draw=none, from=2-2, to=0]
	\arrow["{\theta^\cod_{\id_i} }"{description, pos=0.2}, Rightarrow, draw=none, from=2-2, to=1]
\end{tikzcd}\]
But from the way we constructed the inserted 2-cell this pasting recomposes as the following
\[\begin{tikzcd}[row sep=large]
	{F(i)} \\
	&& {\underset{\Sigma^2}\oplaxbicolim \; F \dom} & {\underset{I}\oplaxbicolim \; F } \\
	& {\underset{\Sigma^2}\oplaxbicolim \; F \dom} & {\underset{I}\oplaxbicolim \; F } \\
	& {\underset{I}\oplaxbicolim \; F } && C
	\arrow["{\langle q_i' \rangle_{i \in I}}", from=2-4, to=4-4]
	\arrow["{\langle q_i' \rangle_{i \in I}}"', from=4-2, to=4-4]
	\arrow["{\langle q_{\cod(d)} F(d) \rangle_{d \in \Sigma^2}}"{description}, from=3-2, to=4-2]
	\arrow["{=}"{description}, draw=none, from=3-2, to=4-4]
	\arrow[""{name=0, anchor=center, inner sep=0}, "{q^{\dom}_{\id_i}}"{description}, from=1-1, to=3-2]
	\arrow[""{name=1, anchor=center, inner sep=0}, "{q_i}"', curve={height=20pt}, end anchor=160, from=1-1, to=4-2]
	\arrow[""{name=2, anchor=center, inner sep=0}, "{q_i}", curve={height=-18pt}, from=1-1, to=2-4]
	\arrow[""{name=3, anchor=center, inner sep=0}, "{q^{\dom}_{\id_i}}"{description}, from=1-1, to=2-3]
	\arrow["{\langle q_{\dom(d)} \rangle_{d \in \Sigma^2}}"', from=2-3, to=2-4]
	\arrow["{\langle q_{\cod(d)} F(d) \rangle_{d \in \Sigma^2}}", from=3-2, to=3-3]
	\arrow["{\langle q_{\dom(d)} \rangle_{d \in \Sigma^2}}"{description}, from=2-3, to=3-3]
	\arrow[Rightarrow, no head, from=3-3, to=4-2]
	\arrow[Rightarrow, no head, from=3-3, to=2-4]
	\arrow["{ \theta_{\id_i}q_{\id_i}\theta'^{-1}_{\id_i}  \atop \simeq}"{pos=0.5}, draw=none, from=3-2, to=2-3]
	\arrow["{\theta'_{\id_i} \atop \simeq}"{description}, Rightarrow, draw=none, from=1, to=0]
	\arrow["{\theta^{-1}_{\id_i} \atop \simeq}"{description, pos=0.6}, Rightarrow, shift right=2, draw=none, from=3, to=2]
\end{tikzcd}\]
But now pasting the isomorphisms $ \theta_{\id_i}$ and $ \theta_{\id_i}'$ with their respective inverses return identity 2-cells, while the oplax transition 2-cells $ q_{\id_i} $ always are identities: hence the pasting above really reduces on the identity, and induces the identity $ \langle q'_i \rangle_{i \in I} $ by passing through the bicolimit. \\

Now we must prove that the inserted 2-cell $ \langle \theta_d q_d \theta'^{-1}_d \rangle$ satisfies the higher coherence condition. Consider the following pasting of the inserted 2-cell along the outer higher data:
\[\begin{tikzcd}
	&& {\underset{\Sigma^2}\oplaxbicolim\;  F\dom} && {\underset{I}\oplaxbicolim\;  F} \\
	{\underset{\Sigma^3}\oplaxbicolim\;  F\dom \, p_0} && {\underset{\Sigma^2}\oplaxbicolim\;  F\dom} &&&& C \\
	&& {\underset{\Sigma^2}\oplaxbicolim\;  F\dom} && {\underset{I}\oplaxbicolim\;  F}
	\arrow["{\langle  q^{\dom}_{p_1\phi}  \rangle_{\phi \in \Sigma^3} }", from=2-1, to=2-3]
	\arrow["{\langle q_{\cod(d)F(d)} \rangle_{d \in \Sigma^2}}", from=1-3, to=1-5]
	\arrow["{\langle  q^{\dom}_{p_0\phi}  \rangle_{\phi \in \Sigma^3} }"', from=2-1, to=3-3]
	\arrow["{\langle q'_i \rangle_{i \in I}}", from=1-5, to=2-7]
	\arrow["{\langle q'_i \rangle_{i \in I}}"', from=3-5, to=2-7]
	\arrow["{\langle \theta_d q_d \theta'^{-1}_d \rangle_{d \in \Sigma^2} \atop \simeq}"{description}, draw=none, from=3-5, to=1-5]
	\arrow["{\langle q_{\dom(d)} \rangle_{d \in \Sigma^2}}"', from=3-3, to=3-5]
	\arrow["{\langle q_{\dom(d)} \rangle_{d \in \Sigma^2}}"{description}, from=2-3, to=3-5]
	\arrow["{\langle q_{\cod(d)F(d)} \rangle_{d \in \Sigma^2}}"{description}, from=2-3, to=1-5]
	\arrow["{\langle  q^{\dom}_{p_2\phi}F(p_0\phi)  \rangle_{\phi \in \Sigma^3} }", from=2-1, to=1-3]
	\arrow["{\theta_{12} \atop \simeq}"', draw=none, from=1-3, to=2-3]
	\arrow["{\theta_{01} \atop \simeq}"', draw=none, from=2-3, to=3-3]
\end{tikzcd}\]

We are again going to prove that this pasting is locally equal to the other one: we saw at \cref{codescent object at the oplax colimit} how the higher coherence data were induced from a family of invertible 2-cells at each $ \phi$ of $\Sigma^3$; then it suffices to replace them in the diagram above and reduce it to a pasting of $\langle q_{\cod(d)F(d)} \rangle_{d \in \Sigma^2}$ along the 2-cell at $\phi$ from the family from which we induced the intermediate coherence data. Consider the following pasting:  

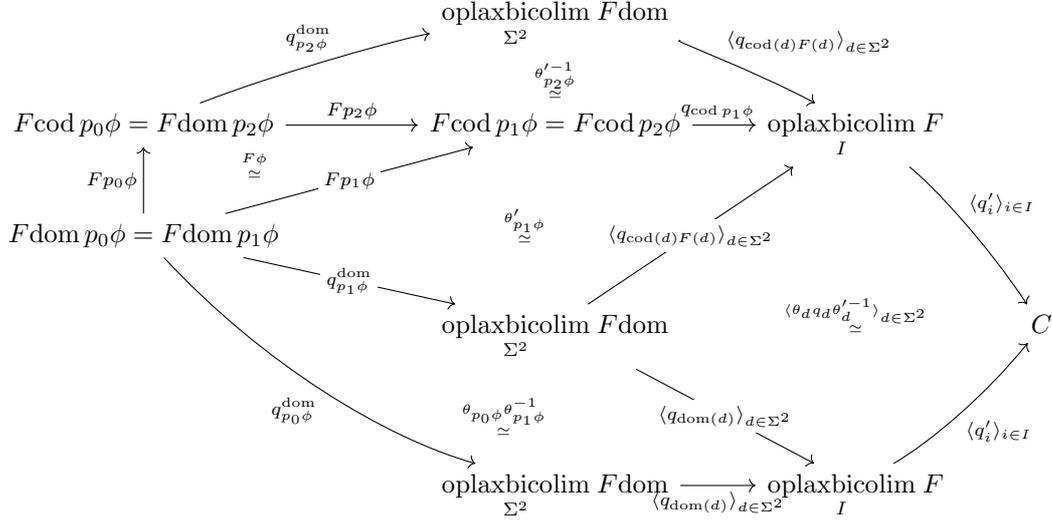
\[\begin{tikzcd}
	&& {\underset{\Sigma^2}\oplaxbicolim\;  F\dom} \\
	{F\cod\,p_0\phi = F \dom \, p_2\phi} && {F\cod \, p_1\phi = F\cod \, p_2\phi} & {\underset{I}\oplaxbicolim\;  F} \\
	{F\dom\, p_0\phi = F\dom \, p_1\phi} \\
	&& {\underset{\Sigma^2}\oplaxbicolim\;  F\dom} && C \\
	\\
	&& {\underset{\Sigma^2}\oplaxbicolim\;  F\dom} & {\underset{I}\oplaxbicolim\;  F}
	\arrow[""{name=1, anchor=center, inner sep=0}, "{Fp_0\phi}", from=3-1, to=2-1]
	\arrow["{q^{\dom}_{p_1\phi}}"{description}, from=3-1, to=4-3]
	\arrow["{q^{\dom}_{p_2\phi}}", curve={height=-6pt}, from=2-1, to=1-3]
	\arrow["{\langle q_{\cod(d)F(d)} \rangle_{d \in \Sigma^2}}", curve={height=-6pt}, from=1-3, to=2-4]
	\arrow[""{name=0, anchor=center, inner sep=0}, "{Fp_1\phi}"{description}, from=3-1, to=2-3]
	\arrow["{q_{\cod \, p_1\phi }}"{pos=0.4}, from=2-3, to=2-4]
	\arrow["{q^{\dom}_{p_0\phi}}"', curve={height=18pt}, from=3-1, to=6-3]
	\arrow["{\langle q'_i \rangle_{i \in I}}", curve={height=-6pt}, from=2-4, to=4-5]
	\arrow["{\langle q'_i \rangle_{i \in I}}"', curve={height=6pt}, from=6-4, to=4-5]
	\arrow["{\langle \theta_d q_d \theta'^{-1}_d \rangle_{d \in \Sigma^2} \atop \simeq}"{description}, draw=none, from=6-4, to=2-4]
	\arrow["{\langle q_{\dom(d)} \rangle_{d \in \Sigma^2}}"', from=6-3, to=6-4]
	\arrow["{\langle q_{\dom(d)} \rangle_{d \in \Sigma^2}}"{description}, from=4-3, to=6-4]
	\arrow["{\langle q_{\cod(d)F(d)} \rangle_{d \in \Sigma^2}}"{description}, from=4-3, to=2-4]
	\arrow["{\theta_{p_0\phi} \theta^{-1}_{p_1\phi} \atop \simeq}", draw=none, from=6-3, to=4-3]
	\arrow["{Fp_2\phi}", from=2-1, to=2-3]
	\arrow["{\theta'_{p_1\phi} \atop \simeq}"', Rightarrow, draw=none, from=2-3, to=4-3]
	\arrow["{\theta'^{-1}_{p_2\phi} \atop \simeq}"{description}, draw=none, from=2-3, to=1-3]
	\arrow["{F\phi \atop \simeq}"{description, pos=0.6}, Rightarrow, draw=none, from=2-3, to=1]
\end{tikzcd}\]

Then we know that the inserted 2-cell, from the very universal property it was induced from in the first part of this proof, satisfies the whiskering equality
\[ \langle q_{\cod(d)F(d)} \rangle_{d \in \Sigma^2} * q^{\dom}_{p_1\phi} = \langle q'_i \rangle_{i \in I} * \theta_{p_1\phi}q_{p_1\phi}\theta'^{-1}_{p_1\phi} \]
(where, beware, the oplax transition 2-cell $ q_{p_1\phi}$ is not invertible) so the diagram above decomposes as the following pasting where $\langle q_{\cod(d)F(d)} \rangle_{d \in \Sigma^2}$ has been whiskered along $q^{\dom}_{p_1\phi}$:

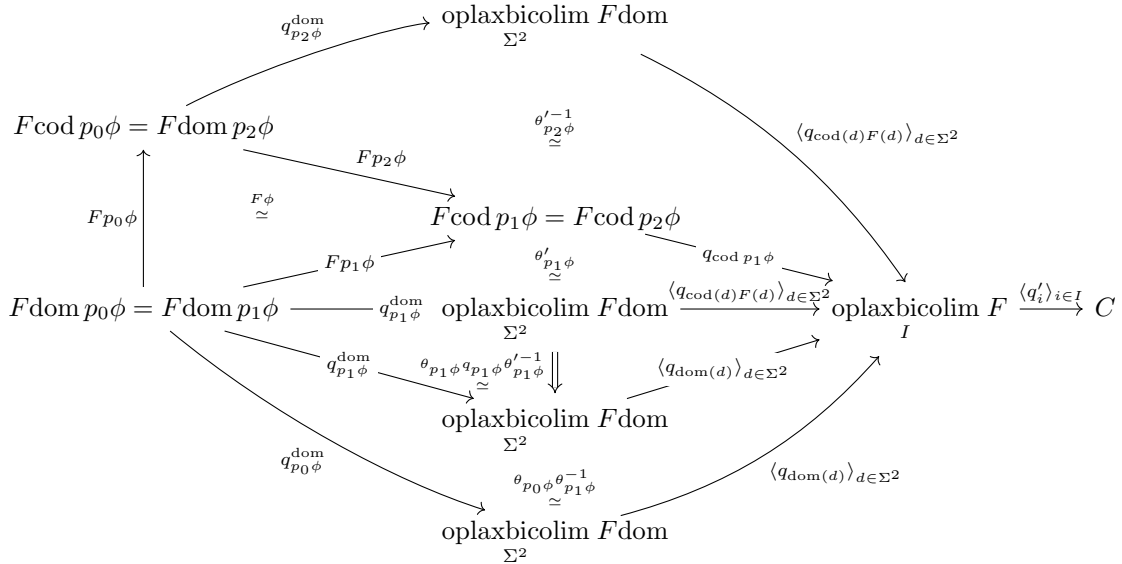
\[\begin{tikzcd}
	&& {\underset{\Sigma^2}\oplaxbicolim\;  F\dom} \\
	{F\cod\,p_0\phi = F \dom \, p_2\phi} \\
	&& {F\cod \, p_1\phi = F\cod \, p_2\phi} \\
	{F\dom\, p_0\phi = F\dom \, p_1\phi} && {\underset{\Sigma^2}\oplaxbicolim\;  F\dom} && {\underset{I}\oplaxbicolim\;  F} & C \\
	&& {\underset{\Sigma^2}\oplaxbicolim\;  F\dom} \\
	&& {\underset{\Sigma^2}\oplaxbicolim\;  F\dom}
	\arrow["{Fp_0\phi}", from=4-1, to=2-1]
	\arrow["{q^{\dom}_{p_1\phi}}"{description}, from=4-1, to=5-3]
	\arrow["{Fp_2\phi}", from=2-1, to=3-3]
	\arrow["{q_{\cod \, p_1\phi }}"{pos=0.5, description}, shift left=1, from=3-3, to=4-5]
	\arrow["{q^{\dom}_{p_0\phi}}"', curve={height=12pt}, from=4-1, to=6-3]
	\arrow["{\langle q'_i \rangle_{i \in I}}", from=4-5, to=4-6]
	\arrow["{q^{\dom}_{p_1\phi} }"{description, pos=0.8}, from=4-1, to=4-3]
	\arrow["{\langle q_{\cod(d)F(d)} \rangle_{d \in \Sigma^2}}", from=4-3, to=4-5]
	\arrow["{\langle q_{\dom(d)} \rangle_{d \in \Sigma^2}}"{description}, from=5-3, to=4-5]
	\arrow["{\langle q_{\dom(d)} \rangle_{d \in \Sigma^2}}"', shift right=1, curve={height=12pt}, from=6-3, to=4-5]
	\arrow["{q^{\dom}_{p_2\phi}}", curve={height=-12pt}, from=2-1, to=1-3]
	\arrow["{\langle q_{\cod(d)F(d)} \rangle_{d \in \Sigma^2}}", shift left=1, curve={height=-12pt}, from=1-3, to=4-5]
	\arrow["{\theta_{p_1\phi}q_{p_1\phi}\theta'^{-1}_{p_1\phi} \atop \simeq}"'{pos=0.4}, shorten >=2pt, Rightarrow, from=4-3, to=5-3]
	\arrow["{\theta_{p_0\phi} \theta^{-1}_{p_1\phi} \atop \simeq}"{description}, draw=none, from=6-3, to=5-3]
	\arrow["{\theta'^{-1}_{p_2\phi} \atop \simeq}"{description}, draw=none, from=3-3, to=1-3]
	\arrow[""{name=0, anchor=center, inner sep=0}, "{Fp_1\phi}"{description}, from=4-1, to=3-3]
	\arrow["{\theta'_{p_1\phi} \atop \simeq}"{description}, draw=none, from=3-3, to=4-3]
	\arrow["{F\phi \atop \simeq}"{description}, Rightarrow, draw=none, from=2-1, to=0]
\end{tikzcd}\]

Now we can compose in this diagram the universal invertible 2-cell $ \theta'_{p_1\phi}$ with its inverse which was put right on its side by inserting the composite $\theta_{p_1\phi}q_{p_1\phi}\theta'^{-1}_{p_1\phi}$, so we can isolate the oplax transition 2-cell $ q_{p_1\phi}$ as below:

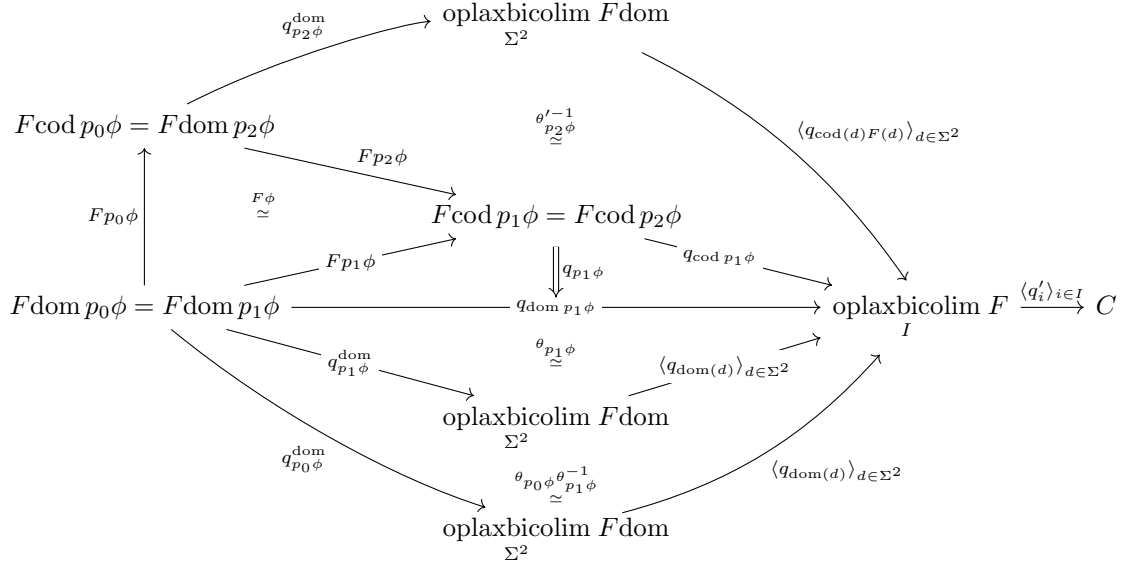
\[\begin{tikzcd}
	&& {\underset{\Sigma^2}\oplaxbicolim\;  F\dom} \\
	{F\cod\,p_0\phi = F \dom \, p_2\phi} \\
	&& {F\cod \, p_1\phi = F\cod \, p_2\phi} \\
	{F\dom\, p_0\phi = F\dom \, p_1\phi} &&&& {\underset{I}\oplaxbicolim\;  F} & C \\
	&& {\underset{\Sigma^2}\oplaxbicolim\;  F\dom} \\
	&& {\underset{\Sigma^2}\oplaxbicolim\;  F\dom}
	\arrow["{Fp_0\phi}", from=4-1, to=2-1]
	\arrow["{q^{\dom}_{p_1\phi}}"{description}, from=4-1, to=5-3]
	\arrow["{Fp_2\phi}", from=2-1, to=3-3]
	\arrow["{q_{\cod \, p_1\phi }}"{pos=0.4, description}, from=3-3, to=4-5]
	\arrow["{q^{\dom}_{p_0\phi}}"', curve={height=12pt}, from=4-1, to=6-3]
	\arrow["{\langle q'_i \rangle_{i \in I}}", from=4-5, to=4-6]
	\arrow["{\langle q_{\dom(d)} \rangle_{d \in \Sigma^2}}"{description}, from=5-3, to=4-5]
	\arrow["{\langle q_{\dom(d)} \rangle_{d \in \Sigma^2}}"', shift right=1, curve={height=12pt}, from=6-3, to=4-5]
	\arrow["{q^{\dom}_{p_2\phi}}", curve={height=-12pt}, from=2-1, to=1-3]
	\arrow["{\langle q_{\cod(d)F(d)} \rangle_{d \in \Sigma^2}}", shift left=1, curve={height=-12pt}, from=1-3, to=4-5]
	\arrow["{\theta_{p_0\phi} \theta^{-1}_{p_1\phi} \atop \simeq}"{description}, draw=none, from=6-3, to=5-3]
	\arrow["{\theta'^{-1}_{p_2\phi} \atop \simeq}"{description}, draw=none, from=3-3, to=1-3]
	\arrow[""{name=0, anchor=center, inner sep=0}, "{Fp_1\phi}"{description}, from=4-1, to=3-3]
	\arrow[""{name=1, anchor=center, inner sep=0}, "{q_{\dom \,p_1\phi}}"{description}, from=4-1, to=4-5]
	\arrow["{F\phi \atop \simeq}"{description}, Rightarrow, draw=none, from=2-1, to=0]
	\arrow["{q_{p_1\phi}}", shorten <=3pt, shorten >=3pt, Rightarrow, from=3-3, to=1]
	\arrow["{\theta_{p_1\phi} \atop \simeq}"{description}, Rightarrow, draw=none, from=1, to=5-3]
\end{tikzcd}\]

Now we use the compatibility condition of the oplax transition 2-cells to decomposes $ q_{Fp_1\phi}$ as $ q_{Fp_0\phi} q_{Fp_1}\phi$ in the following diagram:

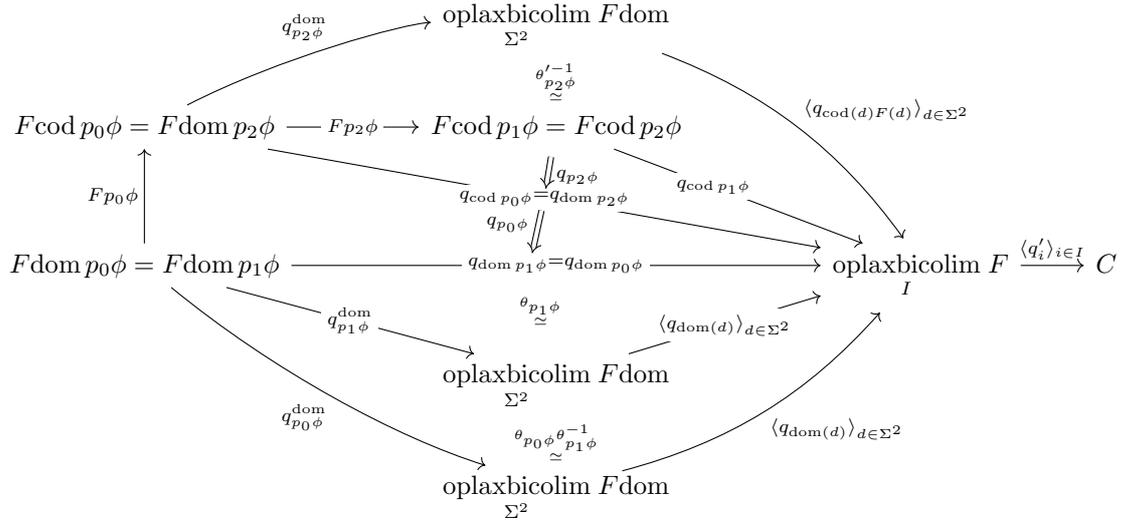
\[\begin{tikzcd}
	&& {\underset{\Sigma^2}\oplaxbicolim\;  F\dom} \\
	{F\cod\,p_0\phi = F \dom \, p_2\phi} && {F\cod \, p_1\phi = F\cod \, p_2\phi} \\
	\\
	{F\dom\, p_0\phi = F\dom \, p_1\phi} &&&& {\underset{I}\oplaxbicolim\;  F} & C \\
	&& {\underset{\Sigma^2}\oplaxbicolim\;  F\dom} \\
	&& {\underset{\Sigma^2}\oplaxbicolim\;  F\dom}
	\arrow["{Fp_0\phi}", from=4-1, to=2-1]
	\arrow["{q^{\dom}_{p_1\phi}}"{description}, from=4-1, to=5-3]
	\arrow["{Fp_2\phi}"{description}, from=2-1, to=2-3]
	\arrow["{q_{\cod \, p_1\phi }}"{description, pos=0.4}, from=2-3, to=4-5]
	\arrow["{q^{\dom}_{p_0\phi}}"', curve={height=12pt}, from=4-1, to=6-3]
	\arrow["{\langle q'_i \rangle_{i \in I}}", from=4-5, to=4-6]
	\arrow["{\langle q_{\dom(d)} \rangle_{d \in \Sigma^2}}"{description}, from=5-3, to=4-5]
	\arrow["{\langle q_{\dom(d)} \rangle_{d \in \Sigma^2}}"', shift right=1, curve={height=12pt}, from=6-3, to=4-5]
	\arrow["{q^{\dom}_{p_2\phi}}", curve={height=-12pt}, from=2-1, to=1-3]
	\arrow["{\langle q_{\cod(d)F(d)} \rangle_{d \in \Sigma^2}}", shift left=1, curve={height=-12pt}, from=1-3, to=4-5]
	\arrow["{\theta_{p_0\phi} \theta^{-1}_{p_1\phi} \atop \simeq}"{description}, draw=none, from=6-3, to=5-3]
	\arrow["{\theta'^{-1}_{p_2\phi} \atop \simeq}"{description}, draw=none, from=2-3, to=1-3]
	\arrow[""{name=0, anchor=center, pos=0.45, inner sep=0}, "{q_{\dom \,p_1\phi}= q_{\dom \,p_0\phi}}"{description}, from=4-1, to=4-5]
	\arrow[""{name=1, anchor=center, inner sep=0}, "{q_{\cod \, p_0\phi} = q_{\dom \, p_2\phi}}"{description}, from=2-1, to=4-5]
	\arrow["{\theta_{p_1\phi} \atop \simeq}"{description}, Rightarrow, draw=none, from=0, to=5-3]
	\arrow["{q_{p_2\phi}}"{pos=0.3}, shorten <=3pt, shorten >=3pt, Rightarrow, from=2-3, to=1]
	\arrow["{q_{p_0\phi}}"'{pos=0.6}, shorten <=4pt, shorten >=4pt, Rightarrow, from=1, to=0]
\end{tikzcd}\]

Now we insert the identity of the projection $ q_{\dom \, p_2\phi}$ and decompose it as the pasting of $ \theta_{p_2\phi}$ together with its inverse, our aim being to make appear both the data $ \theta_{p_2\phi}q_{p_2\phi}\theta'^{-1}_{p_2\phi}$ and $\theta_{p_0\phi}q_{p_0\phi}\theta'^{-1}_{p_0\phi}$ respectively in the upper and lower part of the diagram and then whisker them with the induced map $\langle q'_i \rangle_{i \in I}$ in order to replace the outer coherence 2-cells we started with two copies of the coequalizing 2-cell. We then obtain the following diagram: 

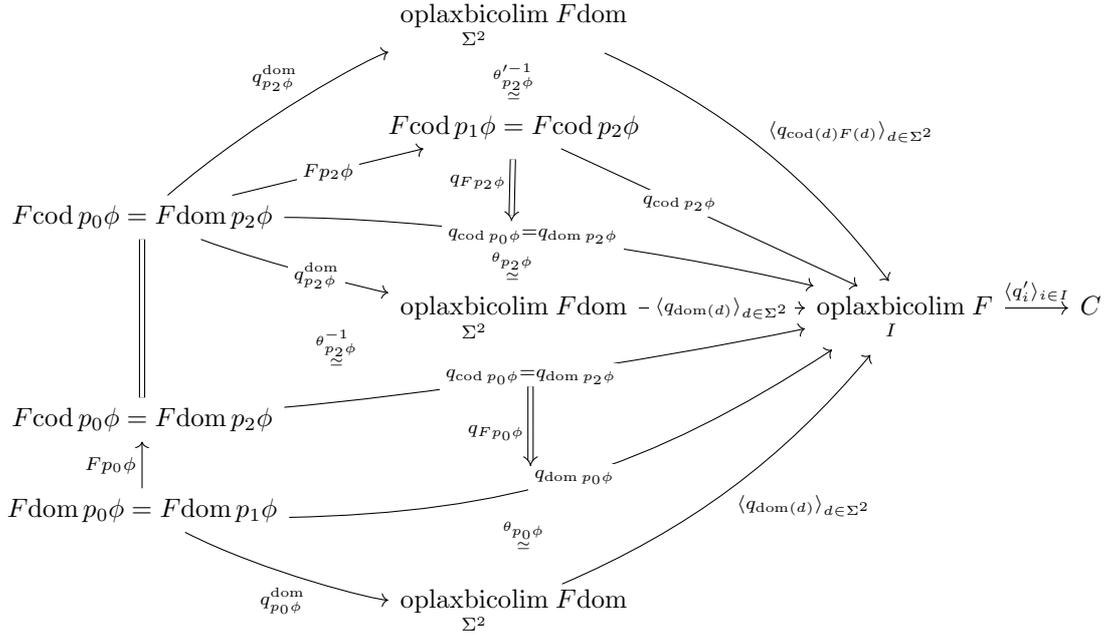
\[\begin{tikzcd}
	&[-15pt] & {\underset{\Sigma^2}\oplaxbicolim\;  F\dom} \\
	&& {F\cod \, p_1\phi = F\cod \, p_2\phi} \\
	{F\cod\,p_0\phi = F \dom \, p_2\phi} \\
	&& {\underset{\Sigma^2}\oplaxbicolim\;  F\dom} &[10pt]& {\underset{I}\oplaxbicolim\;  F} & C \\
	{F\cod\,p_0\phi = F \dom \, p_2\phi} \\
	{F\dom\, p_0\phi = F\dom \, p_1\phi} \\
	&& {\underset{\Sigma^2}\oplaxbicolim\;  F\dom}
	\arrow["{Fp_2\phi}"{description}, from=3-1, to=2-3]
	\arrow["{q_{\cod \, p_2\phi}}"{description, pos=0.4}, from=2-3, to=4-5]
	\arrow["{q^{\dom}_{p_0\phi}}"', curve={height=12pt}, from=6-1, to=7-3]
	\arrow["{\langle q'_i \rangle_{i \in I}}", from=4-5, to=4-6]
	\arrow["{\langle q_{\dom(d)} \rangle_{d \in \Sigma^2}}"', shift right=1, curve={height=12pt}, from=7-3, to=4-5]
	\arrow["{q^{\dom}_{p_2\phi}}", curve={height=-12pt}, from=3-1, to=1-3]
	\arrow["{\langle q_{\cod(d)F(d)} \rangle_{d \in \Sigma^2}}", shift left=1, curve={height=-12pt}, from=1-3, to=4-5]
	\arrow["{\theta'^{-1}_{p_2\phi} \atop \simeq}"{description}, draw=none, from=2-3, to=1-3]
	\arrow[""{name=0, anchor=center, inner sep=0, pos=0.415}, "{q_{\dom \,p_0\phi}}"{description}, bend right=16,  end anchor=-150, from=6-1, to=4-5]
	\arrow[""{name=1, anchor=center, inner sep=0, pos=0.462}, "{q_{\cod \, p_0\phi} = q_{\dom \, p_2\phi}}"{description}, curve={height=-12pt}, from=3-1, to=4-5]
	\arrow["{q^{\dom}_{p_2\phi}}"{description}, curve={height=6pt}, from=3-1, to=4-3]
	\arrow["{\langle q_{\dom(d)} \rangle_{d \in \Sigma^2}}"{description}, from=4-3, to=4-5]
	\arrow["{Fp_0\phi}", from=6-1, to=5-1]
	\arrow[""{name=2, anchor=center, inner sep=0}, Rightarrow, no head, from=3-1, to=5-1]
	\arrow[""{name=3, anchor=center, inner sep=0}, "{q_{\cod \, p_0\phi} = q_{\dom \, p_2\phi}}"{description}, curve={height=6pt}, from=5-1, to=4-5]
	\arrow["{\theta_{p_0\phi} \atop \simeq}"{description}, Rightarrow, draw=none, from=0, to=7-3]
	\arrow["{q_{Fp_2\phi}}"'{pos=0.4}, shorten <=4pt, shorten >=4pt, Rightarrow, from=2-3, to=1]
	\arrow["{\theta_{p_2\phi} \atop \simeq}"{description, pos=0.6}, Rightarrow, draw=none, from=1, to=4-3]
	\arrow["{\theta_{p_2\phi}^{-1} \atop \simeq}"{description}, Rightarrow, draw=none, from=3, to=2]
	\arrow["{q_{Fp_0\phi}}"'{pos=0.5}, shorten <=3pt, shorten >=8pt, Rightarrow, from=3, to=0]
\end{tikzcd}\]

We can now replace the whiskering $\langle q'_i \rangle_{i \in I} * \theta_{p_2\phi}q_{p_2\phi}\theta'^{-1}_{p_2\phi} $ as the whiskering $\langle \theta_d q_d \theta'^{-1}_d \rangle_{d \in \Sigma^2} * q^{\dom}_{p_2\phi} $ as below:

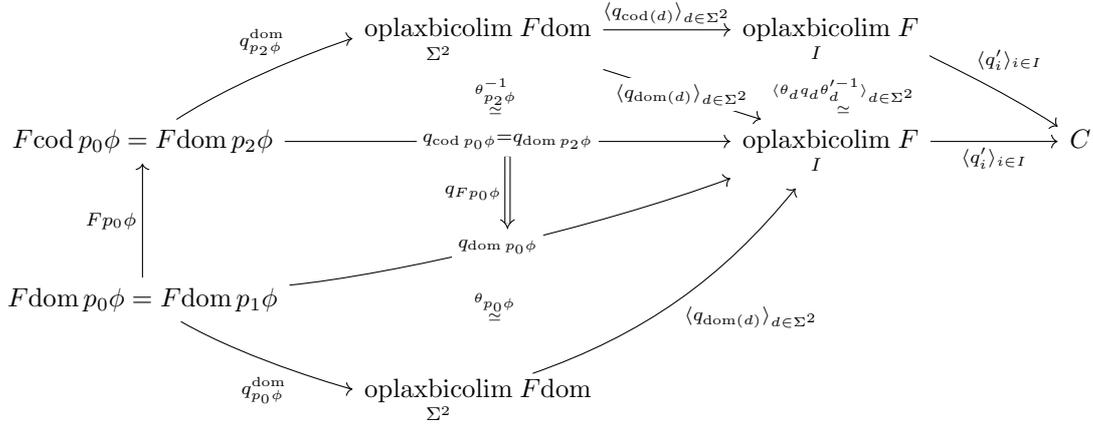
\[\begin{tikzcd}
	& {\underset{\Sigma^2}\oplaxbicolim\;  F\dom} && {\underset{I}\oplaxbicolim\;  F} \\
	{F\cod\,p_0\phi = F \dom \, p_2\phi} &&& {\underset{I}\oplaxbicolim\;  F} && C \\
	\\
	{F\dom\, p_0\phi = F\dom \, p_1\phi} \\
	& {\underset{\Sigma^2}\oplaxbicolim\;  F\dom}
	\arrow["{q^{\dom}_{p_0\phi}}"', curve={height=12pt}, from=4-1, to=5-2]
	\arrow["{\langle q'_i \rangle_{i \in I}}"', from=2-4, to=2-6]
	\arrow["{\langle q_{\dom(d)} \rangle_{d \in \Sigma^2}}"', shift right=1, curve={height=12pt}, from=5-2, to=2-4]
	\arrow[""{name=0, anchor=center, inner sep=0, pos=0.52}, "{q_{\dom \,p_0\phi}}"{description}, curve={height=12pt}, from=4-1, to=2-4]
	\arrow["{\langle q_{\dom(d)} \rangle_{d \in \Sigma^2}}"{description}, from=1-2, to=2-4]
	\arrow["{Fp_0\phi}", from=4-1, to=2-1]
	\arrow[""{name=1, anchor=center, inner sep=0}, "{q_{\cod \, p_0\phi} = q_{\dom \, p_2\phi}}"{description}, from=2-1, to=2-4]
	\arrow["{\langle q'_i \rangle_{i \in I}}", curve={height=-6pt}, from=1-4, to=2-6]
	\arrow["{\langle q_{\cod(d)} \rangle_{d \in \Sigma^2}}", from=1-2, to=1-4]
	\arrow["{\langle \theta_d q_d \theta'^{-1}_d \rangle_{d \in \Sigma^2} \atop \simeq}"{description}, shift left=2, Rightarrow, draw=none, from=1-4, to=2-4]
	\arrow["{q^{\dom}_{p_2\phi}}", curve={height=-12pt}, from=2-1, to=1-2]
	\arrow["{\theta_{p_0\phi} \atop \simeq}"{description}, Rightarrow, draw=none, from=0, to=5-2]
	\arrow["{q_{Fp_0\phi}}"', shorten <=5pt, shorten >=5pt, Rightarrow, from=1, to=0]
	\arrow["{\theta^{-1}_{p_2\phi} \atop \simeq}"{description, pos=0.35}, Rightarrow, draw=none, from=1-2, to=1]
\end{tikzcd}\]

The last step is obtained by inserting the identity of $ p_0\phi$ and splitting it as the pasting of $ \theta'_{p_0\phi}$ with its own inverse, which we can compose with the oplax transition 2-cell at $ q_{p_0\phi}$:

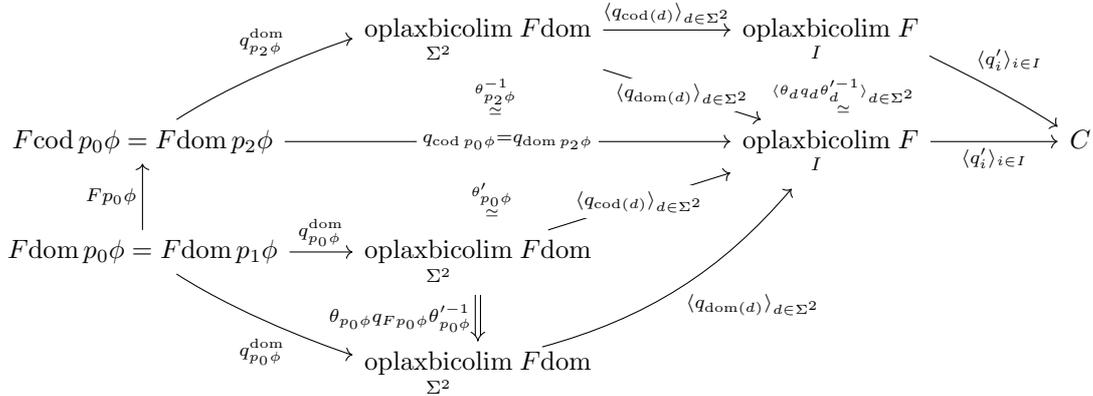
\[\begin{tikzcd}
	& {\underset{\Sigma^2}\oplaxbicolim\;  F\dom} && {\underset{I}\oplaxbicolim\;  F} \\
	{F\cod\,p_0\phi = F \dom \, p_2\phi} &&& {\underset{I}\oplaxbicolim\;  F} && C \\
	{F\dom\, p_0\phi = F\dom \, p_1\phi} & {\underset{\Sigma^2}\oplaxbicolim\;  F\dom} \\
	& {\underset{\Sigma^2}\oplaxbicolim\;  F\dom}
	\arrow["{q^{\dom}_{p_0\phi}}"', curve={height=12pt}, from=3-1, to=4-2]
	\arrow["{\langle q'_i \rangle_{i \in I}}"', from=2-4, to=2-6]
	\arrow["{\langle q_{\dom(d)} \rangle_{d \in \Sigma^2}}"{description}, from=1-2, to=2-4]
	\arrow[""{name=0, anchor=center, inner sep=0}, "{q_{\cod \, p_0\phi} = q_{\dom \, p_2\phi}}"{description}, from=2-1, to=2-4]
	\arrow["{\langle q'_i \rangle_{i \in I}}", curve={height=-6pt}, from=1-4, to=2-6]
	\arrow["{\langle q_{\cod(d)} \rangle_{d \in \Sigma^2}}", from=1-2, to=1-4]
	\arrow["{\langle \theta_d q_d \theta'^{-1}_d \rangle_{d \in \Sigma^2} \atop \simeq}"{description}, shift left=2, Rightarrow, draw=none, from=1-4, to=2-4]
	\arrow["{q^{\dom}_{p_2\phi}}", curve={height=-12pt}, from=2-1, to=1-2]
	\arrow["{\langle q_{\dom(d)} \rangle_{d \in \Sigma^2}}"', shift right=1, curve={height=12pt}, from=4-2, to=2-4]
	\arrow["{Fp_0\phi}", from=3-1, to=2-1]
	\arrow["{q^{\dom}_{p_0\phi}}", from=3-1, to=3-2]
	\arrow["{\langle q_{\cod(d)} \rangle_{d \in \Sigma^2}}"{description}, from=3-2, to=2-4]
	\arrow["{\theta_{p_0\phi} q_{Fp_0\phi} \theta'^{-1}_{p_0\phi}}"', Rightarrow, from=3-2, to=4-2]
	\arrow["{\theta_{p_2\phi}^{-1} \atop \simeq}"{description, pos=0.35}, Rightarrow, draw=none, from=1-2, to=0]
	\arrow["{\theta'_{p_0\phi} \atop \simeq}"{description, pos=0.35}, Rightarrow, draw=none, from=3-2, to=0]
\end{tikzcd}\]

Applying again the whiskering identity yields the following desired pasting with the intermediate coherence 2-cell (as it was defined at \cref{codescent object at the oplax colimit}):

\[\begin{tikzcd}
	& {\underset{\Sigma^2}\oplaxbicolim\;  F\dom} & {\underset{I}\oplaxbicolim\;  F} \\
	{F\cod\,p_0\phi = F \dom \, p_2\phi} \\
	{F\dom\, p_0\phi} && {\underset{I}\oplaxbicolim\;  F} && C \\
	\\
	& {\underset{\Sigma^2}\oplaxbicolim\;  F\dom} & {\underset{I}\oplaxbicolim\;  F}
	\arrow["{Fp_0\phi}", from=3-1, to=2-1]
	\arrow["{q^{\dom}_{p_0\phi}}"', curve={height=12pt}, from=3-1, to=5-2]
	\arrow["{q^{\dom}_{p_2\phi}}", curve={height=-12pt}, from=2-1, to=1-2]
	\arrow["{\langle q_{\cod(d)F(d)} \rangle_{d \in \Sigma^2}}", from=1-2, to=1-3]
	\arrow["{\langle q'_i \rangle_{i \in I}}", curve={height=-12pt}, from=1-3, to=3-5]
	\arrow["{\langle q'_i \rangle_{i \in I}}"', from=3-3, to=3-5]
	\arrow["{\langle q_{\dom(d)} \rangle_{d \in \Sigma^2}}"{description}, from=1-2, to=3-3]
	\arrow[""{name=0, anchor=center, inner sep=0}, "{q_{\dom\, p_1}}"{description}, from=2-1, to=3-3]
	\arrow["{\langle q_{\cod(d)F(d)} \rangle_{d \in \Sigma^2}}"{description}, from=5-2, to=3-3]
	\arrow["{\langle q_{\dom(d)} \rangle_{d \in \Sigma^2}}"', from=5-2, to=5-3]
	\arrow["{\langle q'_i \rangle_{i \in I}}"', curve={height=12pt}, from=5-3, to=3-5]
	\arrow["{\langle \theta_d q_d \theta'^{-1}_d \rangle_{d \in \Sigma^2} \atop \simeq}"{description, pos=0.3}, Rightarrow, draw=none, from=1-3, to=3-3]
	\arrow["{\langle \theta_d q_d \theta'^{-1}_d \rangle_{d \in \Sigma^2} \atop \simeq}"{description}, Rightarrow, draw=none, from=3-3, to=5-3]
	\arrow["{\theta_{p_2\phi} \atop \simeq}"{description}, Rightarrow, draw=none, from=0, to=1-2]
	\arrow["{\theta'_{p_0\phi} \atop \simeq}"{description}, Rightarrow, draw=none, from=5-2, to=0]
\end{tikzcd}\]

But the intermediate 2-cell is the one from which we induced the intermediate coherence data in the codescent diagram through the universal property of the oplax bicolimit, which provides the equality with the desired pasting
\[\begin{tikzcd}
	& {\underset{\Sigma^2}\oplaxbicolim\;  F\dom} & {\underset{I}\oplaxbicolim\;  F} \\
	\\
	{\underset{\Sigma^3}\oplaxbicolim\;  F\dom \, p_0} && {\underset{I}\oplaxbicolim\;  F} && C \\
	\\
	& {\underset{\Sigma^2}\oplaxbicolim\;  F\dom} & {\underset{I}\oplaxbicolim\;  F}
	\arrow["{\langle q_{\cod(d)F(d)} \rangle_{d \in \Sigma^2}}", from=1-2, to=1-3]
	\arrow["{\langle q'_i \rangle_{i \in I}}", curve={height=-12pt}, from=1-3, to=3-5]
	\arrow["{\langle q'_i \rangle_{i \in I}}"', from=3-3, to=3-5]
	\arrow["{\langle q_{\dom(d)} \rangle_{d \in \Sigma^2}}"{description, pos=0.4}, shift right=2, shorten <=-5, shorten >=7, from=1-2, to=3-3]
	\arrow["{\langle q_{\cod(d)F(d)} \rangle_{d \in \Sigma^2}}"{description}, from=5-2, to=3-3]
	\arrow["{\langle q_{\dom(d)} \rangle_{d \in \Sigma^2}}"', from=5-2, to=5-3]
	\arrow["{\langle q'_i \rangle_{i \in I}}"', curve={height=12pt}, from=5-3, to=3-5]
	\arrow["{\langle \theta_d q_d \theta'^{-1}_d \rangle_{d \in \Sigma^2} \atop \simeq}"{description}, Rightarrow, draw=none, from=1-3, to=3-3]
	\arrow["{\langle \theta_d q_d \theta'^{-1}_d \rangle_{d \in \Sigma^2} \atop \simeq}"{description}, Rightarrow, draw=none, from=3-3, to=5-3]
	\arrow["{\langle  q^{\dom}_{p_2\phi}F(p_0\phi)  \rangle_{\phi \in \Sigma^3} }", curve={height=-12pt}, from=3-1, to=1-2]
	\arrow["{\langle  q^{\dom}_{p_0\phi}  \rangle_{\phi \in \Sigma^3} }"', curve={height=12pt}, from=3-1, to=5-2]
	\arrow["{\theta_{02} \atop \simeq}"{description}, draw=none, from=3-1, to=3-3]
\end{tikzcd}\]

Hence the induced 1-cell $\langle q'_i \rangle_{i \in I}$ together with its universal 2-cell $\langle \theta_d q_d \theta'^{-1}_d \rangle_{d \in \Sigma^2 }$ satisfies the higher coherence data, which achieves to prove it pseudocoequalizes the codescent diagram $ \mathscr{X}_{F,\Sigma}$.\\

To sum up, we proved that any pair pseudocoequalizing the codescent diagram $ \mathscr{X}_{F,\Sigma}$ defines uniquely a $\sigma$-cocone over $ F$ relative to $\Sigma$, and that conversely any such $\sigma$-cocone defines a pseudocoequalizing pair. It is clear from our process that those constructions are uniquely defined and induce, for any object $C$, equivalences between $ \sigma$-cocones over $F$ for $\Sigma$ with tip $C$ and pairs $ (q: \oplaxbicolim_{I} F \rightarrow C, \xi)$ pseudocoequalizing $\mathscr{X}_{F,\Sigma}$. Hence the 2-functors sending $ C$ on those categories are naturally equivalent, and hence, if one is representable, so is the other one and by the same objects. This achieves to prove that a $\sigma$-bicolimit of $F$ for $\Sigma$ is the same as a bicoequalizer of $\mathscr{X}_{F,\Sigma}$. 
\end{proof}

\begin{theorem}\label{Oplax bicolimit and bicoeq of codescent generates bicolimits}
Let be $\mathcal{C}$ a 2-category. Then $ \mathcal{C}$ is bicocomplete if and only if it has oplax bicolimits and bicoequalizers of codescent objects.
\end{theorem}

\begin{proof}
We saw at \cref{deweighting} that any weighted bicolimit can be obtained as a conical $\sigma$-bicolimit. But from \cref{bicolim is a bicoeq}, any conical $\sigma$-bicolimit can be obtained as the bicoequalizer of a codescent diagram constructed from oplax bicolimits. 
\end{proof}

This result will be now used in the third and fourth sections of this paper where we shall reduce a bicocompleteness result to existence of bicoequalizers of codescent objects.

\section{Pseudomonads, pseudo-algebras and codescent}

Here we recall some elements of pseudomonad theory. In particular we recall the 2-dimensional bar construction of \cite{le2002beck}. An important source for pseudomonad is also \cite{nunes2017pseudomonads}. 


\begin{definition}
A \emph{pseudomonad}\index{pseudomonad} on a 2-category $ \mathcal{C}$ is a 2-functor $ T : \mathcal{C} \rightarrow \mathcal{C}$ equipped with two pseudonatural transformations \emph{unit} $ \eta : 1 \Rightarrow T$ and a \emph{multiplication} $ \mu : TT \Rightarrow T$ together with canonical invertible 2-cells $ (\xi, \zeta)$ and $ \rho $ defined as follows:
\[\begin{tikzcd}[sep=large]
	T & TT & T \\
	& T
	\arrow[""{name=0, anchor=center, inner sep=0}, Rightarrow, no head, from=1-1, to=2-2]
	\arrow[""{name=1, anchor=center, inner sep=0}, Rightarrow, no head, from=1-3, to=2-2]
	\arrow["{\eta_T}", from=1-1, to=1-2]
	\arrow["T\eta"', from=1-3, to=1-2]
	\arrow["\mu"{description}, from=1-2, to=2-2]
	\arrow["{\xi \atop \simeq}"{description}, Rightarrow, draw=none, from=1-2, to=0]
	\arrow["{\zeta \atop \simeq}"{description}, Rightarrow, draw=none, from=1-2, to=1]
\end{tikzcd} \hskip1cm \begin{tikzcd}[sep=large]
	TTT & TT \\
	TT & T
	\arrow["{\mu_T}"', from=1-1, to=2-1]
	\arrow["T\mu", from=1-1, to=1-2]
	\arrow["\mu", from=1-2, to=2-2]
	\arrow["{\mu}"', from=2-1, to=2-2]
	\arrow["{\rho \atop \simeq}"{description}, draw=none, from=1-1, to=2-2]
\end{tikzcd} \]
\end{definition}

\begin{definition}
A \emph{pseudo-algebra}\index{pseudo-algebra} of a pseudomonad $(T, \eta, \mu, (\xi, \zeta, \rho))$ is a pair $ (A,a, (\alpha^t, \alpha^s))$ with $ A$ an object of $\mathcal{C}$ and $ a : A \rightarrow TA$ a 1-cell in $\mathcal{C}$ and $(\alpha^t, \alpha^s)$ is a pair to 2-cells  as below 

\[
\begin{tikzcd}[sep=large]
	A & TA \\
	& A
	\arrow["{\eta_A}", from=1-1, to=1-2]
	\arrow["a", from=1-2, to=2-2]
	\arrow[""{name=0, anchor=center, inner sep=0}, Rightarrow, no head, from=1-1, to=2-2]
	\arrow["{\alpha^t \atop \simeq}"{description}, Rightarrow, draw=none, from=1-2, to=0]
\end{tikzcd}  \hskip1cm 
\begin{tikzcd}[sep=large]
	TTA & TA \\
	TA & A
	\arrow["{\mu_A}"', from=1-1, to=2-1]
	\arrow["Ta", from=1-1, to=1-2]
	\arrow["a", from=1-2, to=2-2]
	\arrow["a"', from=2-1, to=2-2]
	\arrow["{\alpha^s \atop \simeq}"{description}, draw=none, from=1-1, to=2-2]
\end{tikzcd} 
\]
subject to the following coherence conditions:
\[\begin{tikzcd}
	&& TA \\
	TA & TTA && A \\
	&& TA
	\arrow["{T\eta_A}", from=2-1, to=2-2]
	\arrow["{\mu_A}"{description}, from=2-2, to=3-3]
	\arrow["Ta"{description}, from=2-2, to=1-3]
	\arrow["a", from=1-3, to=2-4]
	\arrow["a"', from=3-3, to=2-4]
	\arrow["{\alpha^s \atop \simeq}"{description}, draw=none, from=2-2, to=2-4]
\end{tikzcd}  = 
\begin{tikzcd}
	& TTA \\
	TA && TA & A \\
	& TTA
	\arrow["{T\eta_A}"{description}, from=2-1, to=1-2]
	\arrow["Ta"{description}, from=1-2, to=2-3]
	\arrow["a", from=2-3, to=2-4]
	\arrow["{T\eta_A}"{description}, from=2-1, to=3-2]
	\arrow["{\mu_A}"{description}, from=3-2, to=2-3]
	\arrow[""{name=0, anchor=center, inner sep=0}, Rightarrow, no head, from=2-1, to=2-3]
	\arrow["{T\alpha^t \atop \simeq}"{description}, draw=none, from=1-2, to=0]
	\arrow["{\zeta \atop \simeq}"{description}, draw=none, from=0, to=3-2]
\end{tikzcd}\]
\[\begin{tikzcd}
	TTTA & TTA \\
	TTA & TTA & TA \\
	& TA & A
	\arrow["{\mu_A}"{description}, from=2-2, to=3-2]
	\arrow["Ta"{description}, from=2-2, to=2-3]
	\arrow["a", from=2-3, to=3-3]
	\arrow["a"', from=3-2, to=3-3]
	\arrow["{\alpha^s \atop \simeq}"{description}, draw=none, from=2-2, to=3-3]
	\arrow["{T\mu_{A}}"{description}, from=1-1, to=2-2]
	\arrow["TTa", from=1-1, to=1-2]
	\arrow["Ta", from=1-2, to=2-3]
	\arrow["{\mu_{TA}}"', from=1-1, to=2-1]
	\arrow["{\mu_A}"', from=2-1, to=3-2]
	\arrow["{\rho_A \atop \simeq}"{description}, draw=none, from=2-1, to=2-2]
	\arrow["{T\alpha^s \atop \simeq}"{description}, draw=none, from=1-2, to=2-2]
\end{tikzcd} = 
\begin{tikzcd}
	TTTA & TTA \\
	TTA & TTA & TA \\
	& TA & A
	\arrow["a", from=2-3, to=3-3]
	\arrow["a"', from=3-2, to=3-3]
	\arrow["TTa", from=1-1, to=1-2]
	\arrow["Ta", from=1-2, to=2-3]
	\arrow["{\mu_{TA}}"', from=1-1, to=2-1]
	\arrow["{\mu_A}"', from=2-1, to=3-2]
	\arrow["{\mu_A}"{description}, from=1-2, to=2-2]
	\arrow["Ta"{description}, from=2-1, to=2-2]
	\arrow["a"{description}, from=2-2, to=3-3]
	\arrow["{\mu_a \atop \simeq}"{description}, draw=none, from=1-1, to=2-2]
	\arrow["{\alpha^s \atop \simeq}"{description}, draw=none, from=2-2, to=2-3]
	\arrow["{\alpha^s \atop \simeq}"{description}, draw=none, from=2-2, to=3-2]
\end{tikzcd}\]

\end{definition}

\begin{definition}
Let $ (T, \eta, \mu, (\xi, \zeta, \rho))$ be a pseudomonad and $ (A,a, (\alpha^t, \alpha^s))$, $(B,b, (\beta^t, \beta^s))$ two pseudo-algebras: then a \emph{pseudomorphism}\index{pseudomorphism} $ (A,a, (\alpha^t, \alpha^s)) \rightarrow (B,b, (\beta^t, \beta^s))$ is a pair $(f,\phi)$ with an arrow $ f : A \rightarrow B$ in $\mathcal{C}$ and $ \phi$ an invertible 2-cell as below
\[ 
\begin{tikzcd}
TA \arrow[r, "Tf"] \arrow[d, "a"'] \arrow[rd, "\phi\atop \simeq", phantom] & TB \arrow[d, "b"] \\
A \arrow[r, "f"']                                                            & B                
\end{tikzcd} \]
subject to the following coherence conditions:
\[
\begin{tikzcd}
	TTA & TTB \\
	TA & TB \\
	A & B
	\arrow["{\mu_A}"', from=1-1, to=2-1]
	\arrow["a"', from=2-1, to=3-1]
	\arrow["Tf"{description}, from=2-1, to=2-2]
	\arrow["b", from=2-2, to=3-2]
	\arrow["TTf", from=1-1, to=1-2]
	\arrow["{\mu_B}", from=1-2, to=2-2]
	\arrow["f"', from=3-1, to=3-2]
	\arrow["\mu_f \atop \simeq"{description}, draw=none, from=1-1, to=2-2]
	\arrow["{\phi \atop \simeq}"{description}, draw=none, from=2-1, to=3-2]
\end{tikzcd}
=
\begin{tikzcd}
	TTA & TTB \\
	TA & TB \\
	A & B
	\arrow["Ta"', from=1-1, to=2-1]
	\arrow["a"', from=2-1, to=3-1]
	\arrow["Tf"{description}, from=2-1, to=2-2]
	\arrow["b", from=2-2, to=3-2]
	\arrow["TTf", from=1-1, to=1-2]
	\arrow["Tb", from=1-2, to=2-2]
	\arrow["f"', from=3-1, to=3-2]
	\arrow["{T\phi \atop \simeq}"{description}, draw=none, from=1-1, to=2-2]
	\arrow["{\phi \atop \simeq}"{description}, draw=none, from=2-1, to=3-2]
\end{tikzcd} \hskip1cm
\begin{tikzcd}
	A &&& B \\
	& TA & TB \\
	A &&& B
	\arrow[""{name=0, anchor=center, inner sep=0}, Rightarrow, no head, from=1-1, to=3-1]
	\arrow[""{name=1, anchor=center, inner sep=0}, "f", from=1-1, to=1-4]
	\arrow[""{name=2, anchor=center, inner sep=0}, Rightarrow, no head, from=1-4, to=3-4]
	\arrow[""{name=3, anchor=center, inner sep=0}, "f"', from=3-1, to=3-4]
	\arrow["{\eta_A}"{description}, from=1-1, to=2-2]
	\arrow["a"{description}, from=2-2, to=3-1]
	\arrow["{\eta_B}"{description}, from=1-4, to=2-3]
	\arrow["b"{description}, from=2-3, to=3-4]
	\arrow[""{name=4, anchor=center, inner sep=0}, "Tf"{description}, from=2-2, to=2-3]
	\arrow["{\eta_f \atop \simeq}"{description}, Rightarrow, draw=none, from=1, to=4]
	\arrow["{\phi \atop \simeq}"{description}, Rightarrow, draw=none, from=3, to=4]
	\arrow["{\alpha^t \atop \simeq}"{description}, Rightarrow, draw=none, from=0, to=2-2]
	\arrow["{\beta^t \atop \simeq}"{description}, Rightarrow, draw=none, from=2-3, to=2]
\end{tikzcd} = {1_f}
 \] 
and also compatibility conditions for the triangle parts and square parts
\[\begin{tikzcd}[sep=large]
	A & TA & TB \\
	& A & B
	\arrow["{\eta_A}", from=1-1, to=1-2]
	\arrow["a"{description}, from=1-2, to=2-2]
	\arrow["Tf", from=1-2, to=1-3]
	\arrow["b", from=1-3, to=2-3]
	\arrow["f"{description}, from=2-2, to=2-3]
	\arrow[""{name=0, anchor=center, inner sep=0}, Rightarrow, no head, from=1-1, to=2-2]
	\arrow["{\phi \atop \simeq}"{description}, draw=none, from=1-2, to=2-3]
	\arrow["{\alpha^t \atop \simeq}"{description}, Rightarrow, draw=none, from=1-2, to=0]
\end{tikzcd} = \begin{tikzcd}[sep=large]
	TA & TB \\
	A & B & B
	\arrow["{\eta_A}", from=2-1, to=1-1]
	\arrow["Tf", from=1-1, to=1-2]
	\arrow["{\eta_B}"{description}, from=2-2, to=1-2]
	\arrow["f"{description}, from=2-1, to=2-2]
	\arrow["\eta_f \atop \simeq"{description}, draw=none, from=1-1, to=2-2]
	\arrow[""{name=0, anchor=center, inner sep=0}, "b", from=1-2, to=2-3]
	\arrow[Rightarrow, no head, from=2-2, to=2-3]
	\arrow["{\beta^t \atop \simeq}"{description}, Rightarrow, draw=none, from=2-2, to=0]
\end{tikzcd}\]
\[\begin{tikzcd}[sep=large]
	TTA & TA & TB \\
	TA & A & B
	\arrow["Ta", from=1-1, to=1-2]
	\arrow["a"{description}, from=1-2, to=2-2]
	\arrow["{\mu_A}"', from=1-1, to=2-1]
	\arrow["a"', from=2-1, to=2-2]
	\arrow["Tf", from=1-2, to=1-3]
	\arrow["b", from=1-3, to=2-3]
	\arrow["f"', from=2-2, to=2-3]
	\arrow["{\alpha^s \atop \simeq}"{description}, draw=none, from=1-1, to=2-2]
	\arrow["{\phi \atop \simeq}"{description}, draw=none, from=1-2, to=2-3]
\end{tikzcd} = \begin{tikzcd}[sep=large]
	TTA & TTB & TB \\
	TA & TB & B
	\arrow["TTf", from=1-1, to=1-2]
	\arrow["{\mu_B}"{description}, from=1-2, to=2-2]
	\arrow["{\mu_A}"', from=1-1, to=2-1]
	\arrow["Tf"', from=2-1, to=2-2]
	\arrow["Tb", from=1-2, to=1-3]
	\arrow["b", from=1-3, to=2-3]
	\arrow["b"', from=2-2, to=2-3]
	\arrow["\mu_f \atop\simeq"{description}, draw=none, from=1-1, to=2-2]
	\arrow["{\beta^s \atop \simeq}"{description}, draw=none, from=1-2, to=2-3]
\end{tikzcd}\]
\end{definition}

\begin{definition}
Let $ (f, \phi), (g,\gamma) : (A,a, (\alpha^t, \alpha^s)) \rightrightarrows (B,b, (\beta^t, \beta^s)) $ be two pseudomorphisms of pseudo-algebras with same domain and codomain; then a transformation between them is a 2-cell $ \alpha : f \rightarrow g$ in $\mathcal{C}$ such that 
\[\begin{tikzcd}[sep=large]
	TA & TB \\
	A & B
	\arrow["a"', from=1-1, to=2-1]
	\arrow["Tf", from=1-1, to=1-2]
	\arrow["b", from=1-2, to=2-2]
	\arrow[""{name=0, anchor=center, inner sep=0}, "f", from=2-1, to=2-2]
	\arrow[""{name=1, anchor=center, inner sep=0}, "g"', curve={height=12pt}, from=2-1, to=2-2]
	\arrow["{\phi \atop \simeq}"{description}, draw=none, from=1-1, to=2-2]
	\arrow["\sigma", shorten <=2pt, shorten >=2pt, Rightarrow, from=0, to=1]
\end{tikzcd} = \begin{tikzcd}[sep=large]
	TA & TB \\
	A & B
	\arrow["a"', from=1-1, to=2-1]
	\arrow[""{name=0, anchor=center, inner sep=0}, "Tf", curve={height=-12pt}, from=1-1, to=1-2]
	\arrow["b", from=1-2, to=2-2]
	\arrow["g"', from=2-1, to=2-2]
	\arrow["{\gamma \atop \simeq}"{description}, draw=none, from=1-1, to=2-2]
	\arrow[""{name=1, anchor=center, inner sep=0}, from=1-1, to=1-2]
	\arrow["T\sigma", shorten <=2pt, shorten >=2pt, Rightarrow, from=0, to=1]
\end{tikzcd}\]
This defines a 2-category $ T\hy\psAlg$,  whose 0-cells are pseudo-algebras, 1-cells are pseudomorphisms of T-pseudoa-algebras, and 2-cells are transformations between them. 
\end{definition}

\begin{division}
For a pseudomonad $(T, \eta, \mu, (\xi, \zeta, \rho))$, we have then a forgetful functor 
\[\begin{tikzcd}
	{T\hy\psAlg} & {\mathcal{C}}
	\arrow["{U_T}", from=1-1, to=1-2]
\end{tikzcd}\]
sending an algebra $ (A, a, (\alpha^t, \alpha^s))$ on the underlying $ A$ and $ (f,\phi)$ on $f$. This functor is right pseudo-adjoint to the associated free functor sending $ A$ to the pseudo-algebra $ (TA, \mu_A, (\xi_A, \rho_A))$ and $ f$ to $ (f, \mu_f)$ with $ \mu_f$ the naturality square of $ \mu$ at $f$: it is standard calculation so see that we have a pseudo-adjunction
\[\begin{tikzcd}
	{T\hy\psAlg} & \perp & {\mathcal{C}}
	\arrow[ "{U_T}"',  start anchor=350, bend right=25, from=1-1, to=1-3]
	\arrow[ "{F_T}"', end anchor=10, bend right=25, from=1-3, to=1-1]
\end{tikzcd}\]
\end{division}

In \cite{osmond:tel-03609605} we proved a pseudo-version of \cite{blackwell1989two}[Theorem 2.6]:

\begin{proposition}\label{pseudoaglebras inherit bilimits}[\cite{osmond:tel-03609605} proposition 6.3.1.3]
Let be a pseudomonad on a 2-category with bilimits. Then the forgetful functor creates bilimits. 
\end{proposition}

However, in this work we are interested in {bicolimits} of pseudoalgebras, which will be the topic of the next two sections. It is known that not all 2-category of pseudo-algebras of a pseudomonad - neither their stricter or lower dimensional analogs - is bicocomplete in general: some additional conditions must be enforced. Several kinds of conditions are known to enable computation of bicolimits of pseudo-algebras, for instance preservation of bicolimits of a certain shape by the 2-functor $T$: 

\begin{lemma}\label{psalg have bicolimit that T preserves}
Let $(T, \eta, \mu, (\xi, \zeta, \rho))$ be a pseudomonad on a 2-category $\mathcal{C}$. Suppose that $\mathcal{C}$ has $I$-indexed conical bicolimits of $I$ a 2-category and that $ T$ preserves them. Then $ T\hy\psAlg$ has $I$-indexed conical bicolimits and $ U_T$ creates them. 
\end{lemma}

However, in the following, we are interested in establishing the existence of arbitrary bicolimits, without assumption about their preservation by $T$. This will rely on auxiliary results on codescent objects, as well as the existence of colimits of algebras was related to reflexive coequalizers in 1-dimension.\\

It is well known that, for an ordinary monad, any algebra is the reflexive coequalizer of a diagram of free algebras in the category of algebra, called the \emph{bar construction}. Here we recall the corresponding statement for pseudo-algebras, which is established in \cite{bourke2010codescent}[Remark 6.7], and also in \cite{le2002beck} by a combination of three results (\cite{le2002beck}[lemma 2.3, proposition 3.2 and corollary 3.3]). Observe that, while \cite{le2002beck} speaks of \emph{pseudocoequalizer} of codescent objects, the universal property they use is the same as our bicoequalizers as they only require an equivalence of categories in their definition 2.1.

\begin{proposition}[Bar construction at a pseudo-algbera]\label{codescent object at an algebra}
Let $(A,a,(\alpha^s, \alpha^t))$ be a pseudo-algebra. Then the following diagram, which we will refer as \emph{the bar construction at $(A,a,(\alpha^s, \alpha^t))$}, is a codescent object in $T\hy\psAlg$, which we will denote $ \mathscr{X}_{(A,a,(\alpha^s, \alpha^t))}$:
\[\begin{tikzcd}
	{(TTTA, \mu_{TTA},(\xi_{TTA},\rho_{TTA}))} \\
	\\
	{(TTA, \mu_{TA},(\xi_{TA},\rho_{TA}))} \\
	\\
	{(TA, \mu_A,(\xi_A,\rho_A))}
	\arrow["{(Ta,\mu_a)}"{pos=0.7}, shift left=12, from=3-1, to=5-1]
	\arrow["{(\mu_A,\rho_A)}"'{pos=0.3}, shift right=12, from=3-1, to=5-1]
	\arrow["{(T\eta_A,\mu_{\eta_A})}"{description}, from=5-1, to=3-1]
	\arrow["{(TTa, \mu_{Ta})}"{pos=0.7}, shift left=15, from=1-1, to=3-1]
	\arrow["{(T\mu_A,\mu_{\mu_A})}"'{pos=0.3}, shift right=15, from=1-1, to=3-1]
	\arrow["{(\mu_{TA},\rho_{TA})}"{description}, from=1-1, to=3-1]
\end{tikzcd}\]
\end{proposition}

\begin{proposition}\label{an algebra coequalize its codescent object}
Let $(A,a,(\alpha^s, \alpha^t))$ be a pseudo-algebra: then the pseudomorphism 
\[\begin{tikzcd}
	{(TA, \mu_A,(\xi_A,\rho_A))} & {(A,a,(\alpha^t,\alpha^s))}
	\arrow["{(a, \alpha^s)}", from=1-1, to=1-2]
\end{tikzcd}\]
exhibits $ (A,a,(\alpha^t,\alpha^s))$ as the bicoequalizer of the codescent object $ \mathscr{X}_{(A,a,(\alpha^s, \alpha^t))}$ in $T\hy\psAlg$. Moreover this bicoequalizer is preserved by the forgetful functor $U_T$. 
\end{proposition}

\begin{remark}
The bar construction provides an instance of a bicoequalizer which always exists at the level of the 2-category of pseudo-algebras, without assumption about their existence for arbitrary codescent diagrams;l: those latter might indeed not exist in general - and we are going to see what happens when they do in the next section. 
\end{remark}

\begin{division}
We are also going to use the 2-functoriality of both the bar construction and the construction of bicoequalizers. For any morphism of pseudo-algebras $ (f,\phi) : (A,a,(\alpha^t,\alpha^s)) \rightarrow (B,b,(\beta^t,\beta^s))$, we have a morphism of codescent object whose pseudonaturality data are given as the following transformations of pseudomorphisms: \begin{itemize}
    \item at the level of the two lower projections, take the following two squares 
\[\begin{tikzcd}[column sep=large]
	{(TTA, \mu_{TA},(\xi_{TA},\rho_{TA}))} & {(TTB, \mu_{TB},(\xi_{TB},\rho_{TB}))} \\
	{(TA, \mu_A,(\xi_A,\rho_A))} & {(TB, \mu_B,(\xi_B,\rho_B))}
	\arrow[""{name=0, anchor=center, inner sep=0}, "{(Tf, \mu_f)}"', from=2-1, to=2-2]
	\arrow["{(Ta, \mu_a)}"', from=1-1, to=2-1]
	\arrow[""{name=1, anchor=center, inner sep=0}, "{(TTf, \mu_{Tf})}", from=1-1, to=1-2]
	\arrow["{(Tb, \mu_b)}", from=1-2, to=2-2]
	\arrow["{T\phi \atop \simeq}"{description}, Rightarrow, draw=none, from=1, to=0]
\end{tikzcd} \]
\[\begin{tikzcd}[column sep=large]
	{(TTA, \mu_{TA},(\xi_{TA},\rho_{TA}))} & {(TTB, \mu_{TB},(\xi_{TB},\rho_{TB}))} \\
	{(TA, \mu_A,(\xi_A,\rho_A))} & {(TB, \mu_B,(\xi_B,\rho_B))}
	\arrow[""{name=0, anchor=center, inner sep=0}, "{(Tf, \mu_f)}"', from=2-1, to=2-2]
	\arrow["{(\mu_A, \rho_A)}"', from=1-1, to=2-1]
	\arrow[""{name=1, anchor=center, inner sep=0}, "{(TTf, \mu_{Tf})}", from=1-1, to=1-2]
	\arrow["{(\mu_B, \rho_B)}", from=1-2, to=2-2]
	\arrow["{\mu_f \atop \simeq}"{description}, Rightarrow, draw=none, from=1, to=0]
\end{tikzcd}\]
\item at the level of their common pseudosection take
\[\begin{tikzcd}[column sep=large]
	{(TTA, \mu_{TA},(\xi_{TA},\rho_{TA}))} & {(TTB, \mu_{TB},(\xi_{TB},\rho_{TB}))} \\
	{(TA, \mu_A,(\xi_A,\rho_A))} & {(TB, \mu_B,(\xi_B,\rho_B))}
	\arrow[""{name=0, anchor=center, inner sep=0}, "{(Tf, \mu_f)}"', from=2-1, to=2-2]
	\arrow["{(T\eta_A, \mu_{\eta_A})}", from=2-1, to=1-1]
	\arrow[""{name=1, anchor=center, inner sep=0}, "{(TTf, \mu_{Tf})}", from=1-1, to=1-2]
	\arrow["{(T\eta_B, \mu_{\eta_B})}"', from=2-2, to=1-2]
	\arrow["{T\eta_f \atop \simeq}"{description}, Rightarrow, draw=none, from=1, to=0]
\end{tikzcd}\]
\item at the level of the higher codescent data take
\[\begin{tikzcd}[column sep=large]
	{(TTTA, \mu_{TTA},(\xi_{TTA},\rho_{TTA}))} & {(TTTB, \mu_{TTB},(\xi_{TTB},\rho_{TTB}))} \\
	{(TTA, \mu_{TA},(\xi_{TA},\rho_{TA}))} & {(TTB, \mu_{TB},(\xi_{TB},\rho_{TB}))}
	\arrow[""{name=0, anchor=center, inner sep=0}, "{(TTf, \mu_Tf)}"', from=2-1, to=2-2]
	\arrow["{(TTa, \mu_{Ta})}"', from=1-1, to=2-1]
	\arrow[""{name=1, anchor=center, inner sep=0}, "{(TTTf, \mu_{TTf})}", from=1-1, to=1-2]
	\arrow["{(TTb, \mu_{Tb})}", from=1-2, to=2-2]
	\arrow["{TT\phi \atop \simeq}"{description}, Rightarrow, draw=none, from=1, to=0]
\end{tikzcd}\]
\[\begin{tikzcd}[column sep=large]
	{(TTTA, \mu_{TTA},(\xi_{TTA},\rho_{TTA}))} & {(TTTB, \mu_{TTB},(\xi_{TTB},\rho_{TTB}))} \\
	{(TTA, \mu_{TA},(\xi_{TA},\rho_{TA}))} & {(TTB, \mu_{TB},(\xi_{TB},\rho_{TB}))}
	\arrow[""{name=0, anchor=center, inner sep=0}, "{(TTf, \mu_Tf)}"', from=2-1, to=2-2]
	\arrow["{(T\mu_A, \mu_{\mu_A})}"', from=1-1, to=2-1]
	\arrow[""{name=1, anchor=center, inner sep=0}, "{(TTTf, \mu_{TTf})}", from=1-1, to=1-2]
	\arrow["{(T\mu_B, \mu_{\mu_B})}", from=1-2, to=2-2]
	\arrow["{T\mu_f \atop \simeq}"{description}, Rightarrow, draw=none, from=1, to=0]
\end{tikzcd}\]

\[\begin{tikzcd}[column sep=large]
	{(TTTA, \mu_{TTA},(\xi_{TTA},\rho_{TTA}))} & {(TTTB, \mu_{TTB},(\xi_{TTB},\rho_{TTB}))} \\
	{(TTA, \mu_{TA},(\xi_{TA},\rho_{TA}))} & {(TTB, \mu_{TB},(\xi_{TB},\rho_{TB}))}
	\arrow[""{name=0, anchor=center, inner sep=0}, "{(TTf, \mu_Tf)}"', from=2-1, to=2-2]
	\arrow["{(\mu_{TA}, \rho_{TA})}"', from=1-1, to=2-1]
	\arrow[""{name=1, anchor=center, inner sep=0}, "{(TTTf, \mu_{TTf})}", from=1-1, to=1-2]
	\arrow["{(\mu_{TB}, \rho_{TB})}", from=1-2, to=2-2]
	\arrow["{\mu_{Tf} \atop \simeq}"{description}, Rightarrow, draw=none, from=1, to=0]
\end{tikzcd}\]
\end{itemize}  

Those data define altogether a pseudonatural transformation 
\[\begin{tikzcd}
	{\mathscr{X}_{(A,a,(\alpha^t,\alpha^s))}} & {\mathscr{X}_{(B,b,(\beta^t,\beta^s))}}
	\arrow["{\overline{(f,\phi)}}", from=1-1, to=1-2]
\end{tikzcd}\]

\end{division}

\begin{lemma}\label{recovering pseudomorphisms through bicoeq of Barr}
A pseudomorphism $(f,\phi)$ can be recovered as the morphism induced between the bicoequalizers of the bar construction, that is one can take $(f,\phi) \simeq \bicoeq({ \overline{(f,\phi)} }) $.
\end{lemma}
\begin{proof}
This induces uniquely, up to a unique invertible 2-cell, a pseudomorphism between the bicoequalizer of the corresponding bar constructions: but we know those latter are exactly the underlying pseudo-algebras, and the existence of the following transformation of pseudomorphisms
\[\begin{tikzcd}
	{(TA, \mu_A,(\xi_A,\rho_A))} & {(TB, \mu_B,(\xi_B,\rho_B))} \\
	{(A,a,(\alpha^t,\alpha^s))} & {(B,b,(\beta^t,\beta^s))}
	\arrow["{(a, \alpha^s)}"', from=1-1, to=2-1]
	\arrow[""{name=0, anchor=center, inner sep=0}, "{(Tf,\mu_f)}", from=1-1, to=1-2]
	\arrow["{(b, \beta^s)}", from=1-2, to=2-2]
	\arrow[""{name=1, anchor=center, inner sep=0}, "{(f, \phi)}"', from=2-1, to=2-2]
	\arrow["{\phi \atop \simeq}"{description}, Rightarrow, draw=none, from=0, to=1]
\end{tikzcd}\]
forces the existence, by the universal property of bicoequalizers, of an invertible 2-cell
\[\begin{tikzcd}
	{\bicoeq(\mathscr{X}_{ (A,a,(\alpha^t,\alpha^s))})} & {\bicoeq(\mathscr{X}_{ (B,b,(\beta^t,\beta^s))})} \\
	{(A,a,(\alpha^t,\alpha^s))} & {(B,b,(\beta^t,\beta^s))}
	\arrow[""{name=0, anchor=center, inner sep=0}, "{(f, \phi)}"', from=2-1, to=2-2]
	\arrow["\simeq"', from=1-1, to=2-1]
	\arrow[""{name=1, anchor=center, inner sep=0}, "{\bicoeq(\overline{(f,\phi)})}", from=1-1, to=1-2]
	\arrow["\simeq", from=1-2, to=2-2]
	\arrow["\simeq"{description}, Rightarrow, draw=none, from=1, to=0]
\end{tikzcd}\]
This ensures that $ (f,\phi)$ can be chosen as the morphism induced between the bicoequalizers by pseudofunctoriality of those latter. 
\end{proof}


\section{Oplax bicolimits of pseudo-algebras}

In this section, we establish the analog of a well known theorem of monad theory, stating that existence of colimits of algebras amounts to existence of coequalizers of algebras. This result first constructs coproducts as coequalizers of free algebras in the category of algebras: then, for arbitrary colimits can be constructed as coequalizers of parallel pairs between coproducts, existence of coequalizers appears sufficient to generate arbitrary colimits. Here we process in a very similar manner, observing that the coequalizers are replaced everywhere by bicoequalizers of codescent objects. We first construct the oplax bicolimit of a diagram of pseudo-algebras as the bicoequalizer of a codescent object made of free pseudo-algebras. Then, invoking the observation of the first section ensuring that any bicolimit can be constructed from oplax bicolimit and bicoequalizer of codescent objects, we reduce bicocompleteness of the 2-category of pseudo-algebras to existence of bicoequalizer of codescent objects. The overall strategy of this section is close to \cite{borceux1994handbook}[4.3], although one must not only handle coherence conditions attached to the pseudo-algebraic structure but also produce further higher codescent data with intricate coherence conditions that have no equivalent in the 1-dimensional case. \\

In this section we fix a pseudomonad $ (T, \eta, \mu, (\xi, \zeta, \rho)) $ on a bicocomplete 2-category $ \mathcal{C}$. 

\begin{division}[Oplax colimit inclusions and structural maps]\label{Oplax colimit inclusions and structural maps}
    
 Let be $ \mathbb{A} : I \rightarrow T\hy\psAlg $ with $ \mathbb{A}(i) = (A_i, a_i, (\alpha_i^t,\alpha_i^s))$ and $ \mathbb{A}(d) = (f_d, \phi_d)$. First compute the oplax bicolimit in $\mathcal{C}$ of the underlying diagram $U_T\mathbb{A} : I \rightarrow \mathcal{C} $:
\[(\begin{tikzcd}
	{A_i} & {\underset{I}\oplaxbicolim \: A_i}
	\arrow["{q_i}", from=1-1, to=1-2]
\end{tikzcd})_{i \in I}\] 
with $ q_d : q_j f_d \Rightarrow q_i$ the oplax inclusion 2-cell at $ d : i \rightarrow j$. Similarly compute the oplax bicolimit in $\mathcal{C}$ $(q^{T}_i : TA_i \rightarrow \oplaxbicolim_I TA_i)_{i \in I}  $ with $ q^{T}_d : q_j f_d \Rightarrow q_i$ the oplax inclusion 2-cell at $ d : i \rightarrow j$. Then we have an oplax cocone $ (q_i a_i : TA_i \rightarrow \oplaxbicolim_I TA_i$ with the transition 2-cell at $ d : i \rightarrow j$ in $I$ given as the pasting 

\[\begin{tikzcd}
	{TA_i} & {A_i} \\
	&& {\underset{I}\oplaxbicolim \; A_i} \\
	{TA_j} & {A_j}
	\arrow[""{name=0, anchor=center, inner sep=0}, "{Tf_d}"', from=1-1, to=3-1]
	\arrow["a_i", from=1-1, to=1-2]
	\arrow[""{name=1, anchor=center, inner sep=0}, "{f_d}"{description}, from=1-2, to=3-2]
	\arrow["a_j", from=3-1, to=3-2]
	\arrow[""{name=2, anchor=center, inner sep=0}, "{q_i}", from=1-2, end anchor=163, to=2-3]
	\arrow[""{name=3, anchor=center, inner sep=0}, "{q_j}"', from=3-2, to=2-3]
	\arrow["{\phi_d \atop \simeq}"{description}, Rightarrow, draw=none, from=0, to=1]
	\arrow["{q_d}", shorten <=5pt, shorten >=5pt, Rightarrow, from=3, to=2]
\end{tikzcd}\]
This oplax cocone induces hence a unique morphism 
\[\begin{tikzcd}
	{\underset{I}\oplaxbicolim \; TA_i} && {\underset{I}\oplaxbicolim \; A_i}
	\arrow["{\underset{I}{\oplaxbicolim \; a_i}}", from=1-1, to=1-3]
\end{tikzcd}\]
together with a family of invertible 2-cells $ (\theta_i)_{i \in I}$ as below
\[\begin{tikzcd}
	{TA_i} && {A_i} \\
	{\underset{I}\oplaxbicolim \; TA_i} && {\underset{I}\oplaxbicolim \; A_i}
	\arrow["{a_i}", from=1-1, to=1-3]
	\arrow["{q_i}", from=1-3, to=2-3]
	\arrow["{q^{T}_i}"', from=1-1, to=2-1]
	\arrow["{\underset{I}{\oplaxbicolim} \; a_i}"', from=2-1, to=2-3]
	\arrow["{\theta_i \atop \simeq}"{description}, draw=none, from=1-1, to=2-3]
\end{tikzcd}\]
\end{division}

\begin{division}[Comparison map]\label{comparison map}

Now on the other hand, we get another oplax cocone over the composite $ T\mathbb{A} : I \rightarrow \mathcal{C} $ as the data of $(Tq_i : TA_i \rightarrow T\oplaxbicolim_I A_i)_{i \in I}$ together with the $Tq_d : Tq_j Tf_d \Rightarrow Tq_i$ as transition 2-cells at $ d : i \rightarrow j$. This defines a universal morphism (which we could see as the comparison map measuring how far $T$ is from preserving oplax bicolimits): 
\[\begin{tikzcd}
	{\underset{I}\oplaxbicolim \; TA_i} & {T\underset{I}\oplaxbicolim \; A_i}
	\arrow["s", from=1-1, to=1-2]
\end{tikzcd}\]
together, for each $i$, with an invertible 2-cell 
\[\begin{tikzcd}
	{TA_i} & {T\underset{I}\oplaxbicolim \; A_i} \\
	{\underset{I}\oplaxbicolim \; TA_i}
	\arrow[""{name=0, anchor=center, inner sep=0}, "s"', from=2-1, to=1-2]
	\arrow["{Tq_i}", from=1-1, to=1-2]
	\arrow["{q^{T}_i}"', from=1-1, to=2-1]
	\arrow["{\sigma_i \atop \simeq}"{description}, Rightarrow, draw=none, from=1-1, to=0]
\end{tikzcd}\]
\end{division}

\begin{division}[Codescent object at the oplax bicolimit: lower data]\label{Codescent object at oplax bicolimit: lower data}
Now, on one hand, we get two parallel 1-cells; first the composite of the multiplication at the oplax bicolimit with the image of the comparison map
\[\begin{tikzcd}
	{T\underset{I}\oplaxbicolim \; TA_i} & {TT\underset{I}\oplaxbicolim \; A_i} && {T\underset{I}\oplaxbicolim \; A_i}
	\arrow["{Ts}", from=1-1, to=1-2]
	\arrow["{\mu_{\underset{I}\oplaxbicolim \; A_i}}", from=1-2, to=1-4]
\end{tikzcd}\]
and, one the other hand, the image of the oplax bicolimit of the structure maps
\[\begin{tikzcd}
	{T\underset{I}\oplaxbicolim \; TA_i} && {T\underset{I}\oplaxbicolim \; A_i}
	\arrow["{T\underset{I}\oplaxbicolim \; a_i}", from=1-1, to=1-3]
\end{tikzcd}\]

We now construct a common pseudosection of those two parallel 1-cells. As well as the structure maps of the algebras provided us with a canonical 1-cell, their units altogether provide an oplax cocone $(q^{T}_i \eta_{A_i} : A_i \rightarrow \oplaxbicolim_I TA_i)_{i \in I}$ with transition 2-cell at $d : i \rightarrow j$ provided by the pasting with the pseudonaturality component of $\eta$ at the underlying map $ f_d$:
\[\begin{tikzcd}
	{A_i} & {TA_i} \\
	&& {\underset{I}{\oplaxbicolim} \; TA_i} \\
	{A_j} & {TA_j}
	\arrow[""{name=0, anchor=center, inner sep=0}, "{f_d}"', from=1-1, to=3-1]
	\arrow[""{name=1, anchor=center, inner sep=0}, "{Tf_d}"{description}, from=1-2, to=3-2]
	\arrow["{\eta_{A_j}}"', from=3-1, to=3-2]
	\arrow["{\eta_{A_i}}", from=1-1, to=1-2]
	\arrow[""{name=2, anchor=center, inner sep=0}, "{q^{T}_j}"', from=3-2, to=2-3]
	\arrow[""{name=3, anchor=center, inner sep=0}, end anchor=164.5, "{q^{T}_i}", from=1-2, to=2-3]
	\arrow["{\eta_{f_d} \atop \simeq}"{description}, Rightarrow, draw=none, from=0, to=1]
	\arrow["{q^{T}_d}", shorten <=5pt, shorten >=5pt, Rightarrow, from=2, to=3]
\end{tikzcd}\]
This oplax cocone provides again a map, in the reverse direction 
\[\begin{tikzcd}
	{\underset{I}{\oplaxbicolim} \; A_i} && {\underset{I}{\oplaxbicolim} \; TA_i}
	\arrow["{\underset{I}{\oplaxbicolim} \; \eta_{A_i}}", from=1-1, to=1-3]
\end{tikzcd}\]
Then, the pseudonaturality of $\eta$ gives an invertible 2-cell 
\[\begin{tikzcd}[column sep=large]
	{\underset{I}{\oplaxbicolim} \; A_i} && {\underset{I}{\oplaxbicolim} \; TA_i} \\
	{T\underset{I}{\oplaxbicolim} \; A_i} && {T\underset{I}{\oplaxbicolim} \; TA_i}
	\arrow["{\underset{I}{\oplaxbicolim} \; \eta_{A_i}}", from=1-1, to=1-3]
	\arrow["{\eta_{\underset{I}{\oplaxbicolim} \; {A_i}}}"', from=1-1, to=2-1]
	\arrow["{\eta_{\underset{I}{\oplaxbicolim} \; T{A_i}}}", from=1-3, to=2-3]
	\arrow["{T\underset{I}{\oplaxbicolim} \; \eta_{A_i}}"', from=2-1, to=2-3]
	\arrow["{\eta_{\underset{I}{\oplaxbicolim} \; \eta_{A_i}} \atop \simeq}"{description}, draw=none, from=1-1, to=2-3]
\end{tikzcd}\]
\end{division}

\begin{lemma}\label{coherence data oplax bicolim : lower data}
The morphism $ T\oplaxbicolim_I \eta_{A_i}$ is a common section of $T\oplaxbicolim_I a_i$ and $ \mu_{\oplaxbicolim_I A_i} Ts$.
\end{lemma}

\begin{proof}
In each $i$ of $I$, the structure map $a_i$ defines a pseudo-retraction of the unit $ \eta_{A_i}$ thanks to the invertible 2-cell $ \alpha^t_i$. Moreover the coherence condition at the transition morphisms $ (f_d, \phi_d)$ 
\[\begin{tikzcd}
	{A_i} &&& {A_j} \\
	& {TA_i} & {TA_j} \\
	{A_i} &&& {A_j}
	\arrow[""{name=0, anchor=center, inner sep=0}, Rightarrow, no head, from=1-1, to=3-1]
	\arrow[""{name=1, anchor=center, inner sep=0}, "{f_d}", from=1-1, to=1-4]
	\arrow[""{name=2, anchor=center, inner sep=0}, Rightarrow, no head, from=1-4, to=3-4]
	\arrow[""{name=3, anchor=center, inner sep=0}, "{f_d}"', from=3-1, to=3-4]
	\arrow["{\eta_{A_i}}"{description}, from=1-1, to=2-2]
	\arrow["{a_i}"{description}, from=2-2, to=3-1]
	\arrow["{\eta_{A_j}}"{description}, from=1-4, to=2-3]
	\arrow["{a_j}"{description}, from=2-3, to=3-4]
	\arrow[""{name=4, anchor=center, inner sep=0}, "{Tf_d}"{description}, from=2-2, to=2-3]
	\arrow["{\alpha_i^t \atop \simeq}"{description}, Rightarrow, draw=none, from=0, to=2-2]
	\arrow["{\alpha_j^t \atop \simeq}"{description}, Rightarrow, draw=none, from=2-3, to=2]
	\arrow["{\eta_{f_d} \atop \simeq}"{description}, Rightarrow, draw=none, from=1, to=4]
	\arrow["{\phi_d \atop \simeq}"{description}, Rightarrow, draw=none, from=3, to=4]
\end{tikzcd} = {1_{f_d}} \]
ensures that those pseudo-retractions are pseudo-natural: this defines a pseudo-retraction in the 2-functors category and pseudonatural squares $[I, \mathcal{C}]_\ps$: 
\[\begin{tikzcd}
	{U_T\mathbb{A}} & {TU_T\mathbb{A}} \\
	& {U_T\mathbb{A}}
	\arrow["{\eta_\mathbb{A}}", from=1-1, to=1-2]
	\arrow["{\overline{a}}", from=1-2, to=2-2]
	\arrow[""{name=0, anchor=center, inner sep=0}, Rightarrow, no head, from=1-1, to=2-2]
	\arrow["{\overline{\alpha} \atop \simeq}"{description, pos=0.35}, Rightarrow, draw=none, from=1-2, to=0]
\end{tikzcd}\]
whith $ \overline{a} : TU_T\mathbb{A} \rightarrow U_T\mathbb{A}$ defined by the pseudonaturality squares given by the data of the $a_i $ and the $(f_d, \phi_d)$ and $ \overline{\alpha}$ provided by the structures 2-cells $ \alpha_i$. This pseudoretraction is sent by the 2-functor $T \oplaxbicolim_I : [I, \mathcal{C}]_\ps \rightarrow \mathcal{C}$ to a pseudoretraction 
\[\begin{tikzcd}
	{T\underset{I}\oplaxbicolim \; A_i} & {T\underset{I}\oplaxbicolim \; TA_i} \\
	& {T\underset{I}\oplaxbicolim \; A_i}
	\arrow["{T\underset{I}\oplaxbicolim \; \eta_{A_i}}", from=1-1, to=1-2]
	\arrow["{T\underset{I}\oplaxbicolim \; a_i}", from=1-2, to=2-2]
	\arrow[""{name=0, anchor=center, inner sep=0}, Rightarrow, no head, from=1-1, to=2-2]
	\arrow["{T \overline{\alpha} \atop \simeq}"{description, pos=0.29}, Rightarrow, draw=none, from=1-2, to=0]
\end{tikzcd}\]

For the second retraction: first observe that the pseudomonad data $\zeta$ gives at $\oplaxbicolim_I A_i$ an invertible 2-cell 
\[\begin{tikzcd}[sep=large]
	{T\underset{I}\oplaxbicolim \; A_i} & {TT\underset{I}\oplaxbicolim \; A_i} \\
	& {T\underset{I}\oplaxbicolim \; A_i}
	\arrow["{T\eta_{\underset{I}\oplaxbicolim \; A_i}}", from=1-1, to=1-2]
	\arrow["{\mu_{\underset{I}\oplaxbicolim \; A_i}}", from=1-2, to=2-2]
	\arrow[""{name=0, anchor=center, inner sep=0}, Rightarrow, shorten <=-2, shorten >=10, no head, shift right=2, from=1-1, to=2-2]
	\arrow["{\zeta_{\underset{I}\oplaxbicolim \; A_i} \atop \simeq}"{description, pos=0.2}, Rightarrow, draw=none, from=1-2, to=0]
\end{tikzcd}\]
Let us now produce a decomposition of $\eta_{\oplaxbicolim_I \; A_i} $ through $ \oplaxbicolim_I \eta_{\; A_i}$. The naturality of the oplax bicolimit at the arrows $ \eta_{A_i}$ provides us with invertible 2-cells 
\[\begin{tikzcd}
	{A_i} & {\underset{I}\oplaxbicolim \; A_i} \\
	{TA_i} & {\underset{I}\oplaxbicolim \; TA_i}
	\arrow["{q_i}", from=1-1, to=1-2]
	\arrow["{\eta_{A_i}}"', from=1-1, to=2-1]
	\arrow["{\widetilde{\theta}_i \atop \simeq}"{description}, draw=none, from=1-1, to=2-2]
	\arrow["{q^{T}_i}"', from=2-1, to=2-2]
	\arrow["{\underset{I}\oplaxbicolim \; \eta_{A_i}}", from=1-2, to=2-2]
\end{tikzcd}\]
which are oplax bicolimit inclusion in $ [2, \mathcal{C}]_{\ps}$: they will be used to produce a pseudosquare form data $ \eta_{A_i} \rightarrow 1_{T\oplaxbicolim_I A_i}$ in $ [2, \mathcal{C}]_{\ps}$. Let us construct those data: recall that we exhibited at \cref{comparison map} an invertible 2-cell $\sigma_i$ at each $i$; now paste them together with the naturality invertible 2-cells $ \eta_{q_i}$ of $ \eta$ at the oplax bicolimit inclusion $q_i$: then those data, which are moreover pseudonatural, induce uniquely an invertible 1-cell $\sigma = \langle \eta_{q_i} \sigma_i \rangle_{i \in I}$:

\[\begin{tikzcd}[sep=large]
	{A_i} & {\underset{I}\oplaxbicolim \; A_i} & {T\underset{I}\oplaxbicolim \; A_i} \\
	{TA_i} & {\underset{I}\oplaxbicolim \; TA_i} & {T\underset{I}\oplaxbicolim \; A_i}
	\arrow["{q_i}", from=1-1, to=1-2]
	\arrow["{\eta_{\underset{I}\oplaxbicolim \; A_i}}", from=1-2, to=1-3]
	\arrow["{\eta_{A_i}}"', from=1-1, to=2-1]
	\arrow["{q^{T}_i}"', from=2-1, to=2-2]
	\arrow[""{name=0, anchor=center, inner sep=0}, "{Tq_i}"{description}, from=2-1, to=1-3]
	\arrow["s"', from=2-2, to=2-3]
	\arrow[""{name=1, anchor=center, inner sep=0}, Rightarrow, no head, from=1-3, to=2-3]
	\arrow["{\eta_{q_i} \atop \simeq}"{description}, Rightarrow, draw=none, from=1-1, to=0]
	\arrow["{\sigma_i \atop \simeq}"{description, pos=0.6}, Rightarrow, draw=none, from=0, to=1]
\end{tikzcd}\]

\[=\begin{tikzcd}[sep=large]
	{A_i} & {\underset{I}\oplaxbicolim \; A_i} & {T\underset{I}\oplaxbicolim \; A_i} \\
	{TA_i} & {\underset{I}\oplaxbicolim \; TA_i} & {T\underset{I}\oplaxbicolim \; A_i}
	\arrow["{q_i}", from=1-1, to=1-2]
	\arrow["{\eta_{A_i}}"', from=1-1, to=2-1]
	\arrow["{\widetilde{\theta}_i \atop \simeq}"{description}, draw=none, from=1-1, to=2-2]
	\arrow["{q^{T}_i}"', from=2-1, to=2-2]
	\arrow["{\underset{I}\oplaxbicolim \; \eta_{A_i}}"{description}, from=1-2, to=2-2]
	\arrow["{\eta_{\underset{I}\oplaxbicolim \; A_i}}", from=1-2, to=1-3]
	\arrow["{\sigma \atop \simeq}"{description}, Rightarrow, draw=none, from=1-2, to=2-3]
	\arrow["s"', from=2-2, to=2-3]
	\arrow[Rightarrow, no head, from=1-3, to=2-3]
\end{tikzcd}\]
This provides us with a decomposition as below, which it suffices to paste with $\zeta_{\oplaxbicolim_I A_i}$ to yield the desired pseudo-retraction:
\[\begin{tikzcd}
	{T\underset{I}\oplaxbicolim \; A_i} && {TT\underset{I}\oplaxbicolim \; A_i} \\
	& {T\underset{I}\oplaxbicolim \; TA_i}
	\arrow["{T\underset{I}\oplaxbicolim \; \eta_{A_i}}"', from=1-1, to=2-2]
	\arrow["{T\eta_{\underset{I}\oplaxbicolim \; A_i}}", from=1-1, to=1-3]
	\arrow[""{name=0, anchor=center, inner sep=0}, "{Ts}"', from=2-2, to=1-3]
	\arrow["{T\sigma \atop \simeq}"{description}, Rightarrow, draw=none, from=1-1, to=0]
\end{tikzcd}\]
\end{proof}

\begin{division}[Iterated comparison map]\label{iterated comparison map}
To construct the higher data we need to first construct a provision of auxiliary 2-cells we are going to use to compare the different combinations of higher and lower coherence maps. First compute the oplax bicolimit of the iterated free construction over $\mathbb{A}$
\[(\begin{tikzcd}
	{TTA_i} & {\underset{I}\oplaxbicolim \; TTA_i}
	\arrow["{q^{TT}_i}", from=1-1, to=1-2]
\end{tikzcd})_{i \in I}\]

Then we obtain again a second comparison map relating it to the oplax bicolimit of the underlying object of free pseudoalgebras, together with an invertible 2-cell in each $i$
\[\begin{tikzcd}
	{TTA_i} & {\underset{I}\oplaxbicolim \; TTA_i} \\
	& {T\underset{I}\oplaxbicolim \; TA_i}
	\arrow["{q^{TT}_i}", from=1-1, to=1-2]
	\arrow["t", from=1-2, to=2-2]
	\arrow[""{name=0, anchor=center, inner sep=0}, "{Tq^T_i}"', from=1-1, to=2-2]
	\arrow["{\tau_i \atop \simeq}"{description, pos=0.03}, Rightarrow, draw=none, from=1-2, to=0]
\end{tikzcd}\]

Moreover we can relate the comparison maps $t$ and $s$ together with the structure maps $a_i$ as follows. The $\tau_i$, together with the 2-cells $ \sigma_i$ of \cref{comparison map} and the underlying 2-cells of the free transformation over the $ \theta_i$ gives us a pasting as below: 

\[\begin{tikzcd}
	& {\underset{I}\oplaxbicolim \; TTA_i} \\
	{TTA_i} && {T\underset{I}\oplaxbicolim \; TA_i} \\
	{TA_i} && {T\underset{I}\oplaxbicolim \; A_i} \\
	& {\underset{I}\oplaxbicolim \; TA_i}
	\arrow["{Ta_i}"', from=2-1, to=3-1]
	\arrow["{q^T_i}"', from=3-1, to=4-2]
	\arrow["s"', from=4-2, to=3-3]
	\arrow["{T\underset{I}\oplaxbicolim \; a_i}", from=2-3, to=3-3]
	\arrow[""{name=0, anchor=center, inner sep=0}, "{Tq^T_i}"{description}, from=2-1, to=2-3]
	\arrow["{q^{TT}_i}", from=2-1, to=1-2]
	\arrow["t", from=1-2, to=2-3]
	\arrow[""{name=1, anchor=center, inner sep=0}, "{Tq_i}"{description}, from=3-1, to=3-3]
	\arrow["{\tau_i \atop \simeq}"{description}, Rightarrow, draw=none, from=0, to=1-2]
	\arrow["{\sigma_i \atop \simeq}"{description}, Rightarrow, draw=none, from=1, to=4-2]
	\arrow["{T\theta_i \atop \simeq}"{description}, Rightarrow, draw=none, from=0, to=1]
\end{tikzcd}\]
But then one can use the oplax bicolimit propery of $ \oplaxbicolim_I Ta_i$ in $\ps[2,\mathcal{C}]$ to infer the existence of a unique invertible 2-cell $\theta =  \langle \tau_i T\theta_i\sigma_i^{-1} \rangle_ {i \in I}$ decomposing the 2-cell above as below
\[ \begin{tikzcd}
	& {\underset{I}\oplaxbicolim \; TTA_i} \\
	{TTA_i} && {T\underset{I}\oplaxbicolim \; TA_i} \\
	{TA_i} && {T\underset{I}\oplaxbicolim \; A_i} \\
	& {\underset{I}\oplaxbicolim \; TA_i}
	\arrow[""{name=0, anchor=center, inner sep=0}, "{Ta_i}"', from=2-1, to=3-1]
	\arrow["{q^T_i}"', from=3-1, to=4-2]
	\arrow["s"', from=4-2, to=3-3]
	\arrow[""{name=1, anchor=center, inner sep=0}, "{T\underset{I}\oplaxbicolim \; a_i}", from=2-3, to=3-3]
	\arrow["{q^{TT}_i}", from=2-1, to=1-2]
	\arrow["t", from=1-2, to=2-3]
	\arrow[""{name=2, anchor=center, inner sep=0}, "{\underset{I}\oplaxbicolim \; Ta_i}"{description, pos=0.6}, from=1-2, to=4-2]
	\arrow["{\theta'_i \atop \simeq}", Rightarrow, draw=none, from=0, to=2]
	\arrow["{\theta \atop \simeq}"{pos=0.4}, Rightarrow, draw=none, from=2, to=1]
\end{tikzcd}\]

We are also going to make use of a certain canonical 2-cell induced from the oplax bicolimit of the multiplications maps in $[2, \mathcal{C}]_{\ps}$:
\[\begin{tikzcd}
	{TTA_i} & {\underset{I}\oplaxbicolim \; TTA_i} \\
	{TA_i} & {\underset{I}\oplaxbicolim \; TA_i}
	\arrow[""{name=0, anchor=center, inner sep=0}, "{q^{TT}_i}", from=1-1, to=1-2]
	\arrow["{\underset{I}\oplaxbicolim \; \mu_{A_i}}", from=1-2, to=2-2]
	\arrow["{\mu_{A_i}}"', from=1-1, to=2-1]
	\arrow[""{name=1, anchor=center, inner sep=0}, "{q^T_i}"', from=2-1, to=2-2]
	\arrow["{\mu_i \atop \simeq}", Rightarrow, draw=none, from=0, to=1]
\end{tikzcd}\]
\end{division}

\begin{division}[Codescent object at the oplax bicolimit: higher coherence data]\label{higer coherence data oplax}
The higher codescent data must encode composition-like operation. The higher object is the free on the oplax bicolimit over the iterated power of $T$: one must choose $ T\oplaxbicolim_I TTA_i$. Then we claim that the three parallel higher maps are to be chosen as the following:\begin{itemize}
    \item the first map is the image of the oplax colimits of the images of the structure map
\[\begin{tikzcd}
	{T\underset{I}\oplaxbicolim \; TTA_i} && {T\underset{I}\oplaxbicolim \; TA_i}
	\arrow["{T\underset{I}\oplaxbicolim \; Ta_i}", from=1-1, to=1-3]
\end{tikzcd}\]
\item the second map is the image of the multiplication 
\[\begin{tikzcd}
	{T\underset{I}\oplaxbicolim \; TTA_i} && {T\underset{I}\oplaxbicolim \; TA_i}
	\arrow["{T\underset{I}\oplaxbicolim \; \mu_{A_i}}", from=1-1, to=1-3]
\end{tikzcd}\]
\item the last one is the composite of the multiplication at the oplax colimit of the images along the image of the second comparison map
\[\begin{tikzcd}
	{T\underset{I}\oplaxbicolim \; TTA_i} & {TT\underset{I}\oplaxbicolim \; TA_i} && {T\underset{I}\oplaxbicolim \; TA_i}
	\arrow["{Tt}", from=1-1, to=1-2]
	\arrow["{\mu_{\underset{I}\oplaxbicolim \; TA_i}}", from=1-2, to=1-4]
\end{tikzcd}\]
\end{itemize}
\end{division}

\begin{lemma}\label{underlying codescent object}
The following diagram is a codescent object in $\mathcal{C}$:
\[\begin{tikzcd}[column sep=large]
	{T\underset{I}\oplaxcolim \; TTA_i} &[-10pt]&& {T\underset{I}\oplaxcolim \; TA_i} && {T\underset{I}\oplaxcolim \; A_i}
	\arrow["{\small{T\underset{I}\oplaxcolim \; \mu_{A_i}}}"{description}, from=1-1, to=1-4]
	\arrow["{T\underset{I}\oplaxcolim \; Ta_i}", shift left=5, from=1-1, to=1-4]
	\arrow["{\mu_{\underset{I}\oplaxcolim \; TA_i} Tt}"', shift right=5, from=1-1, to=1-4]
	\arrow["{T\underset{I}\oplaxcolim \; \eta_{A_i}}"{description}, from=1-6, to=1-4]
	\arrow["{T\underset{I}\oplaxcolim \; a_i}", shift left=5, from=1-4, to=1-6]
	\arrow["{\mu_{\underset{I}\oplaxcolim \; A_i}Ts}"', shift right=5, from=1-4, to=1-6]
\end{tikzcd}\]

\end{lemma}

\begin{proof}
We exhibited the lower coherence data corresponding to the $n_0$, $n_1$ in \cref{coherence data oplax bicolim : lower data}. Now let us construct the higher coherence data: we have to guess canonical invertible 2-cells between three combinations of lower and higher cells. \\

For the first 2-cell, corresponding to $\theta_{01}$ of \cref{codescent diagram}, observe that the data of the $ \alpha_i^s$ form actually a square of pseudonatural transformation in $[I,\mathcal{C}]_\ps$
\[\begin{tikzcd}
	{TTU_T\mathbb{A}} & {TU_T\mathbb{A}} \\
	{TU_T\mathbb{A}} & {U_T\mathbb{A}}
	\arrow["T\mu"', from=1-1, to=2-1]
	\arrow["{\mu_T}", from=1-1, to=1-2]
	\arrow["{\overline{a}}", from=1-2, to=2-2]
	\arrow["{\overline{a}}"', from=2-1, to=2-2]
	\arrow["{\overline{\alpha}^s \atop \simeq}"{description}, draw=none, from=1-1, to=2-2]
\end{tikzcd}\]
Applying the 2-functor $ \oplaxbicolim_I : [I,\mathcal{C}]_\ps \rightarrow \mathcal{C}$ returns then a canonical invertible 2-cell which we can take after applying $T$ once more as the witness of the second coherence condition
\[\begin{tikzcd}
	{T\underset{I}\oplaxbicolim \; TTA_i} && {T\underset{I}\oplaxbicolim \; TA_i} \\
	\\
	{T\underset{I}\oplaxbicolim \; TA_i} && {T\underset{I}\oplaxbicolim \; A_i}
	\arrow["{T\underset{I}\oplaxbicolim \; Ta_i}", from=1-1, to=1-3]
	\arrow["{T\underset{I}\oplaxbicolim \; a_i}"', from=3-1, to=3-3]
	\arrow["{T\underset{I}\oplaxbicolim \; \mu_{A_i}}"', from=1-1, to=3-1]
	\arrow["{T\underset{I}\oplaxbicolim \; a_i}", from=1-3, to=3-3]
	\arrow["{T\underset{I}\oplaxbicolim \; \alpha^s_i}"{description}, draw=none, from=1-1, to=3-3]
\end{tikzcd}\]

The second higher 2-cell, corresponding to $ \theta_{02}$, can be chosen as the following pasting, where $\theta$ was introduced in \cref{iterated comparison map}:
\[\begin{tikzcd}
	{T\underset{I}\oplaxbicolim \; TTA_i} && {T\underset{I}\oplaxbicolim \; TA_i} \\
	{TT\underset{I}\oplaxbicolim \; TA_i} && {TT\underset{I}\oplaxbicolim \; A_i} \\
	{T\underset{I}\oplaxbicolim \; TA_i} && {T\underset{I}\oplaxbicolim \; A_i}
	\arrow[""{name=0, anchor=center, inner sep=0}, "{T\underset{I}\oplaxbicolim \; Ta_i}", from=1-1, to=1-3]
	\arrow["{Tt}"', from=1-1, to=2-1]
	\arrow["{\mu_{\underset{I}\oplaxbicolim \; TA_i}}"', from=2-1, to=3-1]
	\arrow["{Ts}", from=1-3, to=2-3]
	\arrow["{\mu_{\underset{I}\oplaxbicolim \; A_i}}", from=2-3, to=3-3]
	\arrow["{T\underset{I}\oplaxbicolim \; a_i}"', from=3-1, to=3-3]
	\arrow[""{name=1, anchor=center, inner sep=0}, "{TT\underset{I}\oplaxbicolim \; a_i}", from=2-1, to=2-3]
	\arrow["{\mu_{\underset{I}\oplaxbicolim \; a_i} \atop \simeq}"{description}, draw=none, from=2-1, to=3-3]
	\arrow["{T\theta \atop \simeq}"{description, pos=0.3}, draw=none, from=0, to=1]
\end{tikzcd}\]

For the last 2-cell corresponding to $\theta_{12}$, we have first to construct an auxiliary 2-cell relating the multiplication at the oplax bicolimit and the two comparison maps. We need a canonical cell comparing the action of $t$ and the action of $s$. We can consider the following pasting, made of the pseudonaturality square of $\mu$ at the inclusion $ q_i$ together with the comparisons 2-cells $\sigma_i$, its image along $T$ and the 2-cell $\tau_i$ we constructed at \cref{iterated comparison map} for each $i$:
\[\begin{tikzcd}
	{\underset{I}\oplaxbicolim \; TTA_i} && {T\underset{I}\oplaxbicolim \; TA_i} \\
	{TTA_i} &&& {TT\underset{I}\oplaxbicolim \; A_i} \\
	{TA_i} &&& {T\underset{I}\oplaxbicolim \; A_i} \\
	&& {\underset{I}\oplaxbicolim \; TA_i}
	\arrow[""{name=0, anchor=center, inner sep=0}, "Ts", from=1-3, to=2-4]
	\arrow[""{name=1, anchor=center, inner sep=0}, "{\mu_{\underset{I}\oplaxbicolim \; A_i}}", from=2-4, to=3-4]
	\arrow[""{name=2, anchor=center, inner sep=0}, "{\mu_{A_i}}"', from=2-1, to=3-1]
	\arrow["t", from=1-1, to=1-3]
	\arrow["{q^{TT}_i}", from=2-1, to=1-1]
	\arrow[""{name=3, anchor=center, inner sep=0}, "{Tq^{T}_i}"{description}, from=2-1, to=1-3]
	\arrow[""{name=4, anchor=center, inner sep=0}, "s"', from=4-3, to=3-4]
	\arrow[""{name=5, anchor=center, inner sep=0}, "{q^{T}_i}"', from=3-1, to=4-3]
	\arrow["{TTq_i}"{description}, from=2-1, to=2-4]
	\arrow["{Tq_i}"{description}, from=3-1, to=3-4]
	\arrow["{\tau_i \atop \simeq}"{description, pos=0.3}, draw=none, from=1-1, to=3]
	\arrow["{\mu_{q_i} \atop \simeq}"{description}, draw=none, from=2, to=1]
	\arrow["{T\sigma_i \atop \simeq}"{description}, shift left=1, draw=none, from=0, to=3]
	\arrow["{\sigma_i^{-1} \atop \simeq}"{description}, draw=none, from=5, to=4]
\end{tikzcd}\]
Then using again the oplax bicolimit property in the 2-category $ \ps[2,\mathcal{C}]$ we induce a universal 2-cell $\tau = \langle \sigma_i^{-1}\mu_{q_i}T\sigma_i (\mu_{\underset{I}\oplaxbicolim \; A_i} Ts) *\tau_i \rangle_{i \in I} $ decomposing the 2-cell above as below
\[\begin{tikzcd}
	{TTA_i} & {\underset{I}\oplaxbicolim \; TTA_i} & {T\underset{I}\oplaxbicolim \; TA_i} & {TT\underset{I}\oplaxbicolim \; A_i} \\
	{TA_i} & {\underset{I}\oplaxbicolim \; TA_i} && {T\underset{I}\oplaxbicolim \; A_i}
	\arrow["Ts", from=1-3, to=1-4]
	\arrow["{\mu_{A_i}}"', from=1-1, to=2-1]
	\arrow["t", from=1-2, to=1-3]
	\arrow["{q^{TT}_i}", from=1-1, to=1-2]
	\arrow["s"', from=2-2, to=2-4]
	\arrow["{q^{T}_i}"', from=2-1, to=2-2]
	\arrow[""{name=0, anchor=center, inner sep=0}, "{\mu_{\underset{I}\oplaxbicolim \; A_i}}", from=1-4, to=2-4]
	\arrow[""{name=1, anchor=center, inner sep=0}, "{\underset{I}\oplaxbicolim \; \mu_{A_i}}", from=1-2, to=2-2]
	\arrow["{\mu_i \atop \simeq}"{description}, draw=none, from=1-1, to=2-2]
	\arrow["{\tau \atop \simeq}"{description}, Rightarrow, draw=none, from=1, to=0]
\end{tikzcd}\]
Then one can consider the following pasting 
\[\begin{tikzcd}
	{T\underset{I}\oplaxbicolim \; TTA_i} && {TT\underset{I}\oplaxbicolim \; TA_i} && {T\underset{I}\oplaxbicolim \; TA_i} \\
	&& {TTT\underset{I}\oplaxbicolim \; A_i} && {TT\underset{I}\oplaxbicolim \; A_i} \\
	{T\underset{I}\oplaxbicolim \; TA_i} && {TT\underset{I}\oplaxbicolim \; A_i} && {T\underset{I}\oplaxbicolim \; A_i}
	\arrow[""{name=0, anchor=center, inner sep=0}, "Tt", from=1-1, to=1-3]
	\arrow[""{name=1, anchor=center, inner sep=0}, "Ts"', from=3-1, to=3-3]
	\arrow["{T\underset{I}\oplaxbicolim \; \mu_{A_i}}"', from=1-1, to=3-1]
	\arrow[""{name=2, anchor=center, inner sep=0}, "{\mu_{\underset{I}\oplaxbicolim \; TA_i}}", from=1-3, to=1-5]
	\arrow["{\mu_{\underset{I}\oplaxbicolim \; A_i}}"', from=3-3, to=3-5]
	\arrow["Ts", from=1-5, to=2-5]
	\arrow["{\mu_{\underset{I}\oplaxbicolim \; A_i}}", from=2-5, to=3-5]
	\arrow["TTs"', from=1-3, to=2-3]
	\arrow["{T\mu_{\underset{I}\oplaxbicolim \; A_i}}"', from=2-3, to=3-3]
	\arrow[""{name=3, anchor=center, inner sep=0}, "{\mu_{T\underset{I}\oplaxbicolim \; A_i}}", from=2-3, to=2-5]
	\arrow["{\rho_{\underset{I}\oplaxbicolim \; A_i} \atop \simeq}"{description}, draw=none, from=3-3, to=2-5]
	\arrow["{\mu_s \atop \simeq}"{description, pos=0.4}, Rightarrow, draw=none, from=2, to=3]
	\arrow["{T\tau \atop \simeq}"{description}, draw=none, from=0, to=1]
\end{tikzcd}\]

\end{proof}

\begin{division}[Codescent object at the oplax bicolimit: pseudo-algebra structure]\label{Codescent object at the oplax bicolimit: pseudo-algebra structure}
Now we can lift the codescent object above in $T\hy\psAlg$, for each object has an obvious free pseudo-algebra structure on it, while we can guess convenient pseudomorphism structures on the 1-cells:\begin{itemize}
    \item for $T\oplaxbicolim_I \; a_i$ take the pseudonaturality square of the multiplication:
\[\begin{tikzcd}
	{TT\underset{I}\oplaxbicolim \; TA_i} && {TT\underset{I}\oplaxbicolim \; A_i} \\
	{T\underset{I}\oplaxbicolim \; TA_i} && {T\underset{I}\oplaxbicolim \; A_i}
	\arrow["{TT\underset{I}\oplaxbicolim \; a_i}", from=1-1, to=1-3]
	\arrow["{\mu_{\underset{I}\oplaxbicolim \; TA_i}}"', from=1-1, to=2-1]
	\arrow["{\mu_{\underset{I}\oplaxbicolim \; A_i}}", from=1-3, to=2-3]
	\arrow["{T\underset{I}\oplaxbicolim \; a_i}"', from=2-1, to=2-3]
	\arrow["{\mu_{\underset{I}\oplaxbicolim \; a_i} \atop \simeq}"{description}, draw=none, from=1-1, to=2-3]
\end{tikzcd}\]
\item similarly, for $ T\underset{I}\oplaxbicolim \; \eta_{A_i}$, take the pseudonaturality square:
\[\begin{tikzcd}[column sep=large]
	{TT\underset{I}\oplaxbicolim \; TA_i} && {TT\underset{I}\oplaxbicolim \; A_i} \\
	{T\underset{I}\oplaxbicolim \; TA_i} && {T\underset{I}\oplaxbicolim \; A_i}
	\arrow["{TT\underset{I}\oplaxbicolim \; \eta_{A_i}}"', from=1-3, to=1-1]
	\arrow["{\mu_{\underset{I}\oplaxbicolim \; TA_i}}"', from=1-1, to=2-1]
	\arrow["{\mu_{\underset{I}\oplaxbicolim \; A_i}}", from=1-3, to=2-3]
	\arrow["{T\underset{I}\oplaxbicolim \; \eta_{A_i}}", from=2-3, to=2-1]
	\arrow["{\mu_{\underset{I}\oplaxbicolim \; \eta_{A_i}} \atop \simeq}"{description}, draw=none, from=1-1, to=2-3]
\end{tikzcd}\]
\item for $\mu_{\underset{I}\oplaxbicolim \; A_i} Ts$ take the the pasting:
\[\begin{tikzcd}
	{TT\underset{I}\oplaxbicolim \; TA_i} && {TTT\underset{I}\oplaxbicolim \; A_i} && {TT\underset{I}\oplaxbicolim \; A_i} \\
	{T\underset{I}\oplaxbicolim \; TA_i} && {TT\underset{I}\oplaxbicolim \; A_i} && {T\underset{I}\oplaxbicolim \; A_i}
	\arrow["{\mu_{\underset{I}\oplaxbicolim \; TA_i}}"', from=1-1, to=2-1]
	\arrow["{\mu_{\underset{I}\oplaxbicolim \; A_i}}", from=1-5, to=2-5]
	\arrow["{TTs}", from=1-1, to=1-3]
	\arrow["{T\mu_{\underset{I}\oplaxbicolim \; A_i}}", from=1-3, to=1-5]
	\arrow["{Ts}"', from=2-1, to=2-3]
	\arrow["{\mu_{\underset{I}\oplaxbicolim \; A_i}}"', from=2-3, to=2-5]
	\arrow["{\mu_{T\underset{I}\oplaxbicolim \; A_i}}"{description}, from=1-3, to=2-3]
	\arrow["{\mu_{s} \atop \simeq}"{description}, draw=none, from=1-1, to=2-3]
	\arrow["{\rho_{\underset{I}\oplaxbicolim \; A_i}}"{description}, draw=none, from=1-3, to=2-5]
\end{tikzcd}\]
\item for $T\oplaxbicolim_I \mu_{A_i}$ take the pseudonaturality square
\[\begin{tikzcd}
	{TT\underset{I}\oplaxbicolim \; TTA_i} && {TT\underset{I}\oplaxbicolim \; TA_i} \\
	{T\underset{I}\oplaxbicolim \; TTA_i} && {T\underset{I}\oplaxbicolim \; TA_i}
	\arrow[""{name=0, anchor=center, inner sep=0}, "{T\underset{I}\oplaxbicolim \; \mu_{A_i}}"', from=2-1, to=2-3]
	\arrow["{\mu_{\underset{I}\oplaxbicolim \; T{A_i}}}", from=1-3, to=2-3]
	\arrow["{\mu_{\underset{I}\oplaxbicolim \; TT{A_i}}}"', from=1-1, to=2-1]
	\arrow[""{name=1, anchor=center, inner sep=0}, "{TT\underset{I}\oplaxbicolim \; \mu_{A_i}}", from=1-1, to=1-3]
	\arrow["{\mu_{\underset{I}\oplaxbicolim \; \mu_{A_i}} \atop \simeq}"{description}, Rightarrow, draw=none, from=1, to=0]
\end{tikzcd}\]
\item for $\mu_{\oplaxbicolim_I TA_i} Tt$ take the pasting 
\[\begin{tikzcd}
	{TT\underset{I}\oplaxbicolim \; TTA_i} & {TTT\underset{I}\oplaxbicolim \; TA_i} && {TT\underset{I}\oplaxbicolim \; TA_i} \\
	{T\underset{I}\oplaxbicolim \; TTA_i} & {TT\underset{I}\oplaxbicolim \; TA_i} && {T\underset{I}\oplaxbicolim \; TA_i}
	\arrow[""{name=0, anchor=center, inner sep=0}, "Tt"', from=2-1, to=2-2]
	\arrow[""{name=1, anchor=center, inner sep=0}, "{\mu_{\underset{I}\oplaxbicolim \; TA_i}}"', from=2-2, to=2-4]
	\arrow["{\mu_{\underset{I}\oplaxbicolim \; T{A_i}}}", from=1-4, to=2-4]
	\arrow["{\mu_{T\underset{I}\oplaxbicolim \; T{A_i}}}", from=1-2, to=2-2]
	\arrow["{\mu_{\underset{I}\oplaxbicolim \; TT{A_i}}}"', from=1-1, to=2-1]
	\arrow[""{name=2, anchor=center, inner sep=0}, "TTt", from=1-1, to=1-2]
	\arrow[""{name=3, anchor=center, inner sep=0}, "{T\mu_{\underset{I}\oplaxbicolim \; T{A_i}}}", from=1-2, to=1-4]
	\arrow["{\mu_{Tt} \atop \simeq}"{description}, Rightarrow, draw=none, from=2, to=0]
	\arrow["{\rho_{\underset{I}\oplaxbicolim \; T{A_i}} \atop \simeq}", Rightarrow, draw=none, from=3, to=1]
\end{tikzcd}\]
\item and finally for $ T\oplaxbicolim_I Ta_i$ take
\[\begin{tikzcd}
	{TT\underset{I}\oplaxbicolim \; TTA_i} && {TT\underset{I}\oplaxbicolim \; TA_i} \\
	{T\underset{I}\oplaxbicolim \; TTA_i} && {T\underset{I}\oplaxbicolim \; TA_i}
	\arrow[""{name=0, anchor=center, inner sep=0}, "{T\underset{I}\oplaxbicolim \; Ta_i}"', from=2-1, to=2-3]
	\arrow["{\mu_{\underset{I}\oplaxbicolim \; T{A_i}}}", from=1-3, to=2-3]
	\arrow["{\mu_{\underset{I}\oplaxbicolim \; TT{A_i}}}"', from=1-1, to=2-1]
	\arrow[""{name=1, anchor=center, inner sep=0}, "{TT\underset{I}\oplaxbicolim \; Ta_i}", from=1-1, to=1-3]
	\arrow["{\mu_{\underset{I}\oplaxbicolim \; Ta_i} \atop \simeq}"{description}, Rightarrow, draw=none, from=1, to=0]
\end{tikzcd}\]
\end{itemize}

\end{division}

\begin{lemma}
The following diagram, which we shall denote $ \mathscr{X}_\mathbb{A}$, is a codescent object in $T\hy\psAlg$:
\[\begin{tikzcd}[row sep=huge]
	{(T\underset{I}\oplaxcolim \; TTA_i, \mu_{\underset{I}\oplaxbicolim \; TTA_i},(\xi_{\underset{I}\oplaxbicolim \; TTA_i}, \rho_{\underset{I}\oplaxbicolim \; TTA_i}))} \\
	\\
	{(T\underset{I}\oplaxcolim \; TA_i, \mu_{\underset{I}\oplaxbicolim \; TA_i}, (\xi_{\underset{I}\oplaxbicolim \; TA_i},\rho_{\underset{I}\oplaxbicolim \; TA_i}))} \\
	\\
	{(T\underset{I}\oplaxcolim \; A_i, \mu_{\underset{I}\oplaxbicolim \; A_i}, (\xi_{\underset{I}\oplaxbicolim \; A_i},\rho_{\underset{I}\oplaxbicolim \; A_i}))}
	\arrow["{\small{(T\underset{I}\oplaxcolim \; \mu_{A_i}, \mu_{\underset{I}\oplaxbicolim \; \mu_{A_i})}}}"{description}, from=1-1, to=3-1]
	\arrow["{(T\underset{I}\oplaxcolim \; Ta_i,\mu_{\underset{I}\oplaxbicolim \; Ta_i})}"{description, pos=0.8}, shift left=50, from=1-1, to=3-1]
	\arrow[" {(\mu_{\underset{I}\oplaxcolim \; TA_i} Tt, \mu_{\underset{I}\oplaxbicolim \; T{A_i}}* \mu_{Tt}\rho_{\underset{I}\oplaxbicolim \; T{A_i}}*TTt)} "{description, pos=0.2}, shift right=50,  from=1-1, to=3-1]
	\arrow["{(T\underset{I}\oplaxcolim \; \eta_{A_i},\mu_{\underset{I}\oplaxbicolim \; \eta_{A_i}})}"{description}, from=5-1, to=3-1]
	\arrow["{(T\underset{I}\oplaxcolim \; a_i,\mu_{\underset{I}\oplaxbicolim \; a_i})}"{description, pos=0.8}, shift left=40, from=3-1, to=5-1]
	\arrow["{(\mu_{\underset{I}\oplaxcolim \; A_i}Ts,\mu_{\underset{I}\oplaxbicolim \; A_i}*\mu_s\rho_{\underset{I}\oplaxbicolim \; A_i}*TTs)}"{description, pos=0.2}, shift right=40,  from=3-1, to=5-1]
\end{tikzcd}\]
\end{lemma}

\begin{proof}
It is a tedious, yet straightforward manipulation of the coherence data of the pseudomonads to check the data of \cref{Codescent object at the oplax bicolimit: pseudo-algebra structure} define really pseudomorphisms; similarly for the verification that the underlying invertible 2-cell in the codescent structure of \cref{underlying codescent object} and \cref{coherence data oplax bicolim : lower data} satisfies the coherence conditions of transformation of pseudomorphisms of pseudo-algebras: we let this as an exercise for the careful reader.
\end{proof}

\begin{division}[Morphism of codescent object over the inclusions]
Beware that oplax bicolimit inclusions $ q_i : A_i \rightarrow \oplaxbicolim_I A_i$ do not bear pseudomorphism structure; however, it is possible to induce from the data we constructed above a morphism of codescent object 
\[\begin{tikzcd}
	{\mathscr{X}_{(A_i, a_i,(\alpha^t_i, \alpha^s_i))}} & {\mathscr{X}_\mathbb{A}}
	\arrow["{\overline{q_i} }", Rightarrow, from=1-1, to=1-2]
\end{tikzcd}\] from the data we constructed throughout this section: its pseudonaturality squares will be the following: \begin{itemize}
    \item at the lower data, take the following data (where $\theta_i$ as defined at \cref{Oplax colimit inclusions and structural maps} and $\mu_i$ at \cref{iterated comparison map}):
\[\begin{tikzcd}[column sep=small]
	{(TTA_i,\mu_{TA_i}, (\xi_{TA_i}, \rho_{TA_i}) )} && {(T\underset{I}\oplaxcolim \; TA_i, \mu_{\underset{I}\oplaxbicolim \; TA_i}, (\xi_{\underset{I}\oplaxbicolim \; TA_i},\rho_{\underset{I}\oplaxbicolim \; TA_i}))} \\
	{(TA_i,\mu_{A_i}, (\xi_{A_i}, \rho_{A_i}) )} && {(T\underset{I}{\oplaxbicolim} \; A_i, \mu_{\underset{I}{\oplaxbicolim} \; A_i}, (\xi_{\underset{I}{\oplaxbicolim} \; A_i}, \rho_{\underset{I}{\oplaxbicolim} \; A_i}))}
	\arrow[""{name=0, anchor=center, inner sep=0}, "{(Tq_i, \mu_{q_i})}"', from=2-1, to=2-3]
	\arrow["{(Ta_i, \mu_{a_i})}"', from=1-1, to=2-1]
	\arrow["{(T\underset{I}\oplaxcolim \; a_i,\mu_{\underset{I}\oplaxbicolim \; a_i})}", from=1-3, to=2-3]
	\arrow[""{name=1, anchor=center, inner sep=0}, "{(Tq^T_i), \mu_{q^T_i})}", from=1-1, to=1-3]
	\arrow["{T\theta_i \atop \simeq}"{description}, Rightarrow, draw=none, from=1, to=0]
\end{tikzcd}\]
\[\begin{tikzcd}[column sep=small]
	{(TTA_i,\mu_{TA_i}, (\xi_{TA_i}, \rho_{TA_i}) )} && {(T\underset{I}\oplaxcolim \; TA_i, \mu_{\underset{I}\oplaxbicolim \; TA_i}, (\xi_{\underset{I}\oplaxbicolim \; TA_i},\rho_{\underset{I}\oplaxbicolim \; TA_i}))} \\
	&& {(TT\underset{I}{\oplaxbicolim} \; A_i, \mu_{T\underset{I}{\oplaxbicolim} \; A_i}, (\xi_{T\underset{I}{\oplaxbicolim} \; A_i}, \rho_{T\underset{I}{\oplaxbicolim} \; A_i}))} \\
	{(TA_i,\mu_{A_i}, (\xi_{A_i}, \rho_{A_i}) )} && {(T\underset{I}{\oplaxbicolim} \; A_i, \mu_{\underset{I}{\oplaxbicolim} \; A_i}, (\xi_{\underset{I}{\oplaxbicolim} \; A_i}, \rho_{\underset{I}{\oplaxbicolim} \; A_i}))}
	\arrow["{(\mu_{A_i}, \rho_{A_i})}"', from=1-1, to=3-1]
	\arrow["{(Tq^T_i, \mu_{q^T_i})}", from=1-1, to=1-3]
	\arrow[""{name=0, anchor=center, inner sep=0}, "{(Tq_i, \mu_{q_i})}"', from=3-1, to=3-3]
	\arrow[""{name=1, anchor=center, inner sep=0}, "{(Ts,\mu_s)}", from=1-3, to=2-3]
	\arrow["{(\mu_{\underset{I}\oplaxcolim \; A_i}
,\rho_{\underset{I}\oplaxbicolim \; A_i})}", from=2-3, to=3-3]
	\arrow[""{name=2, anchor=center, inner sep=0}, "{(TTq_i, \mu_{Tq_i})}"{description}, from=1-1, to=2-3]
	\arrow["{T\sigma_i \atop \simeq}"{description}, Rightarrow, draw=none, from=2, to=1]
	\arrow["{\mu_i \atop \simeq}"{description}, Rightarrow, draw=none, from=2, to=0]
\end{tikzcd}\]
\item at their common retraction take the following (where $\overline{\theta_i}$ was defined at \cref{coherence data oplax bicolim : lower data})
\[\begin{tikzcd}[column sep=small]
	{(TTA_i,\mu_{TA_i}, (\xi_{TA_i}, \rho_{TA_i}) )} && {(T\underset{I}\oplaxcolim \; TA_i, \mu_{\underset{I}\oplaxbicolim \; TA_i}, (\xi_{\underset{I}\oplaxbicolim \; TA_i},\rho_{\underset{I}\oplaxbicolim \; TA_i}))} \\
	{(TA_i,\mu_{A_i}, (\xi_{A_i}, \rho_{A_i}) )} && {(T\underset{I}{\oplaxbicolim} \; A_i, \mu_{\underset{I}{\oplaxbicolim} \; A_i}, (\xi_{\underset{I}{\oplaxbicolim} \; A_i}, \rho_{\underset{I}{\oplaxbicolim} \; A_i}))}
	\arrow[""{name=0, anchor=center, inner sep=0}, "{(Tq^T_i, \mu_{q^T_i})}", from=1-1, to=1-3]
	\arrow[""{name=1, anchor=center, inner sep=0}, "{(Tq_i, \mu_{q_i})}"', from=2-1, to=2-3]
	\arrow["{(T\eta_{A_i}, \mu_{\eta_{A_i}})}", from=2-1, to=1-1]
	\arrow["{(T\underset{I}\oplaxcolim \; \eta_{A_i},\mu_{\underset{I}\oplaxbicolim \; \eta_{A_i}})}"', from=2-3, to=1-3]
	\arrow["{T\widetilde{\theta}_i \atop \simeq}"{description}, Rightarrow, draw=none, from=1, to=0]
\end{tikzcd}\]
\item at the higher data take the following (where $ \theta'_i$ was defined at \cref{iterated comparison map})

\footnotesize{

\[\begin{tikzcd}[column sep=small]
	{(TTTA_i,\mu_{TTA_i}, (\xi_{TTA_i}, \rho_{TTA_i}) )} && {\small{(T\underset{I}\oplaxcolim \; TTA_i, \mu_{\underset{I}\oplaxbicolim \; TTA_i},(\xi_{\underset{I}\oplaxbicolim \; TTA_i}, \rho_{\underset{I}\oplaxbicolim \; TTA_i}))}} \\
	{(TTA_i,\mu_{TA_i}, (\xi_{TA_i}, \rho_{TA_i}) )} && {\small{(T\underset{I}\oplaxcolim \; TA_i, \mu_{\underset{I}\oplaxbicolim \; TA_i}, (\xi_{\underset{I}\oplaxbicolim \; TA_i},\rho_{\underset{I}\oplaxbicolim \; TA_i}))}}
	\arrow[""{name=0, anchor=center, inner sep=0}, "{(Tq^{TT}_i, \mu_{q^{TT}_i})}", from=1-1, to=1-3]
	\arrow["{(T\mu_{A_i}, \mu_{\mu_{A_i}})}"', from=1-1, to=2-1]
	\arrow["{(T\underset{I}\oplaxcolim \; \mu_{A_i}, \mu_{\underset{I}\oplaxbicolim \; \mu_{A_i})}}", from=1-3, to=2-3]
	\arrow[""{name=1, anchor=center, inner sep=0}, "{(Tq^{T}_i, \mu_{q^{T}_i})}"', from=2-1, to=2-3]
	\arrow["{T\mu_i \atop \simeq}"{description}, Rightarrow, draw=none, from=0, to=1]
\end{tikzcd}\]
\[\begin{tikzcd}[column sep=small]
	{(TTTA_i,\mu_{TTA_i}, (\xi_{TTA_i}, \rho_{TTA_i}) )} && {\small{(T\underset{I}\oplaxcolim \; TTA_i, \mu_{\underset{I}\oplaxbicolim \; TTA_i},(\xi_{\underset{I}\oplaxbicolim \; TTA_i}, \rho_{\underset{I}\oplaxbicolim \; TTA_i}))}} \\
	&& {\small{(TT\underset{I}\oplaxcolim \; TA_i, \mu_{T\underset{I}\oplaxbicolim \; TA_i}, (\xi_{T\underset{I}\oplaxbicolim \; TA_i},\rho_{T\underset{I}\oplaxbicolim \; TA_i}))}} \\
	{(TTA_i,\mu_{TA_i}, (\xi_{TA_i}, \rho_{TA_i}) )} && {\small{(T\underset{I}\oplaxcolim \; TA_i, \mu_{\underset{I}\oplaxbicolim \; TA_i}, (\xi_{\underset{I}\oplaxbicolim \; TA_i},\rho_{\underset{I}\oplaxbicolim \; TA_i}))}}
	\arrow["{(Tq^{TT}_i, \mu_{q^{TT}_i})}", from=1-1, to=1-3]
	\arrow["{(\mu_{TA_i}, \rho_{{A_i}})}"', from=1-1, to=3-1]
	\arrow[""{name=0, anchor=center, inner sep=0}, "{(Tq^{T}_i, \mu_{q^{T}_i})}"', from=3-1, to=3-3]
	\arrow[""{name=1, anchor=center, inner sep=0}, "{(Tt, \mu_{t})}", from=1-3, to=2-3]
	\arrow["{\small{(\mu_{\underset{I}\oplaxcolim \; TA_i}, \rho_{\underset{I}\oplaxbicolim \; TA_i})}}", from=2-3, to=3-3]
	\arrow[""{name=2, anchor=center, inner sep=0}, "{(TTq^T_i, \mu_{Tq^T_i})}"{description}, from=1-1, to=2-3]
	\arrow["{T\tau_i \atop \simeq}"{description}, Rightarrow, draw=none, from=2, to=1]
	\arrow["{\mu_{q^T_i} \atop \simeq}", Rightarrow, draw=none, from=2, to=0]
\end{tikzcd}\]
\[\begin{tikzcd}
	{(TTTA_i,\mu_{TTA_i}, (\xi_{TTA_i}, \rho_{TTA_i}) )} && {\small{(T\underset{I}\oplaxcolim \; TTA_i, \mu_{\underset{I}\oplaxbicolim \; TTA_i},(\xi_{\underset{I}\oplaxbicolim \; TTA_i}, \rho_{\underset{I}\oplaxbicolim \; TTA_i}))}} \\
	\\
	{(TTA_i,\mu_{TA_i}, (\xi_{TA_i}, \rho_{TA_i}) )} && {\small{(T\underset{I}\oplaxcolim \; TA_i, \mu_{\underset{I}\oplaxbicolim \; TA_i}, (\xi_{\underset{I}\oplaxbicolim \; TA_i},\rho_{\underset{I}\oplaxbicolim \; TA_i}))}}
	\arrow[""{name=0, anchor=center, inner sep=0}, "{(Tq^{TT}_i, \mu_{q^{TT}_i})}", from=1-1, to=1-3]
	\arrow["{(TTa_i, \mu_{{TA_i}})}"', from=1-1, to=3-1]
	\arrow[""{name=1, anchor=center, inner sep=0}, "{(Tq^{T}_i, \mu_{q^{T}_i})}"', from=3-1, to=3-3]
	\arrow["{(T\underset{I}\oplaxcolim \; Ta_i,\mu_{\underset{I}\oplaxbicolim \; Ta_i})}", from=1-3, to=3-3]
	\arrow["{T\theta'_i \atop \simeq}"{description}, Rightarrow, draw=none, from=0, to=1]
\end{tikzcd}\]}
\end{itemize}

\end{division}

\begin{division}[Oplax cocone of pseudo-algebras]
Now suppose that arbitrary bicoequalizers of codescent objects exist in $ T\hy\psAlg$ and denote as 
\[\begin{tikzcd}
	{(T\underset{I}\oplaxcolim \; A_i, \mu_{\underset{I}\oplaxcolim \; A_i}, (\xi_{\underset{I}\oplaxcolim \; A_i}, \rho_{\underset{I}\oplaxcolim \; A_i}))} & {(C,c, (\gamma^t,\gamma^s))}
	\arrow["{(p,\pi)}", from=1-1, to=1-2]
\end{tikzcd}\]
the bicoequalizer of the codescent object $ \mathscr{X}_\mathbb{A}$.
Then by pseudofunctoriality of bicoequalizer of codescent objects, and using that each pseudo-algebra is the bicoequalizer of its own bar construction, the morphism of codescent object $\overline{q_i} $ at each $i$ of $I$ induces a canonical 2-cell in $ T\hy\psAlg$:
\[\begin{tikzcd}
	{(TA_i,\mu_{A_i}, (\xi_{A_i}, \rho_{A_i}) )} && {(T\underset{I}{\oplaxbicolim} \; A_i, \mu_{\underset{I}{\oplaxbicolim} \; A_i}, (\xi_{\underset{I}{\oplaxbicolim} \; A_i}, \rho_{\underset{I}{\oplaxbicolim} \; A_i}))} \\
	{(A_i, a_i,(\alpha^t_i, \alpha^s_i))} && {(C,c, (\gamma^t,\gamma^s))}
	\arrow[""{name=0, anchor=center, inner sep=0}, "{(Tq_i, \mu_{q_i})}", from=1-1, to=1-3]
	\arrow["{(a_i,\alpha_i^s)}"', from=1-1, to=2-1]
	\arrow["{(p, \pi)}", from=1-3, to=2-3]
	\arrow[""{name=1, anchor=center, inner sep=0}, "{(l_i, \lambda_i)}"', from=2-1, to=2-3]
	\arrow["{\chi_i \atop \simeq}"{description}, Rightarrow, draw=none, from=0, to=1]
\end{tikzcd}\]

\end{division}

\begin{proposition}\label{Oplax cocone of psalg}
The oplax cocone $ (l_i, \lambda_i)_{i \in I}$ is a an oplax bicolimit of $\mathbb{A}$ in $ T\hy\psAlg$.
\end{proposition}

\begin{proof}
Let be $(k_i, \kappa_i) : (A_i, a_i,(\alpha^t_i, \alpha_i^s)) \rightarrow (B,b, (\beta^t,\beta^s))$ an oplax cocone over $\mathbb{A}$ in $T\hy\psAlg$. Its underlying oplax cocone $(k_i: A_i \rightarrow B)_{i \in I}$ in $\mathcal{C}$ provides us with a universal factorization 
\[\begin{tikzcd}
	{A_i} && B \\
	& {\underset{I}\oplaxbicolim \; A_i}
	\arrow[""{name=0, anchor=center, inner sep=0}, "{k_i}", from=1-1, to=1-3]
	\arrow["{q_i}"', from=1-1, to=2-2]
	\arrow["{\langle k_i \rangle_{i \in I}}"', from=2-2, to=1-3]
	\arrow["{\nu_i \atop \simeq}"{description}, Rightarrow, draw=none, from=0, to=2-2]
\end{tikzcd}\]

Applying and iterating the free construction to this diagram produces not only an invertible 2-cell in the category of algebras
\[\begin{tikzcd}[column sep=tiny]
	{(TA_i,\mu_{A_i}, (\xi_{A_i}, \rho_{A_i}) )} && {(T\underset{I}{\oplaxbicolim} \; A_i, \mu_{\underset{I}{\oplaxbicolim} \; A_i}, (\xi_{\underset{I}{\oplaxbicolim} \; A_i}, \rho_{\underset{I}{\oplaxbicolim} \; A_i}))} \\
	& {(TB,\mu_{B}, (\xi_{B}, \rho_{B}) )}
	\arrow[""{name=0, anchor=center, inner sep=0}, "{(Tq_i, \mu_{q_i})}", from=1-1, to=1-3]
	\arrow["{(Tk_i, \mu_{k_i})}"', from=1-1, to=2-2]
	\arrow["{(T\langle k_i \rangle_{i \in I}, \mu_{\langle k_i \rangle_{i \in I}})}", from=1-3, to=2-2]
	\arrow["{T\nu_i \atop \simeq}"{description}, Rightarrow, draw=none, from=0, to=2-2]
\end{tikzcd}\]
but even a transformation of codescent object in $[\mathbb{X}, T\hy\psAlg]_\ps$ between the corresponding bar constructions
\[\begin{tikzcd}
	{\mathscr{X}_{(A_i, a_i,(\alpha^t_i, \alpha^s_i))}} && {\mathscr{X}_\mathbb{A}} \\
	& {\mathscr{X}_{(B,b, (\beta^t,\beta^s))}}
	\arrow[""{name=0, anchor=center, inner sep=0}, "{\overline{q_i}}", from=1-1, to=1-3]
	\arrow["{\overline{k_i}}"', from=1-1, to=2-2]
	\arrow["{\overline{k}}", from=1-3, to=2-2]
	\arrow["{\overline{\nu}_i \atop \simeq}"{description}, Rightarrow, draw=none, from=0, to=2-2]
\end{tikzcd}\]
Then, the pseudofunctoriality of the bicoequalizer construction provides us with a prism whose vertical 1-cells are the bicoequalizing inclusions and the bottom face is a 2-cell relating induced 1-cells between the bicoequalizers. But we know by \cref{an algebra coequalize its codescent object} that each $ (A_i, a_i,(\alpha^t_i, \alpha^s_i))$ is the bicoequalizer of the corresponding $\mathscr{X}_{(A_i, a_i,(\alpha^t_i, \alpha^s_i))}$, as well as ${(B,b, (\beta^t,\beta^s))} $ is the bicoequalizer of $\mathscr{X}_{(B,b, (\beta^t,\beta^s))}$, and that each $ (k_i, \kappa_i)$ is induced from the morphism of codescent object it induces at the level of the overlying bar construction $ \overline{k_i}$. Similarly the $ (l_i, \lambda_i)$ were induced by functoriality of the bicoequalizers from the morphism of codescent object induced by the $q_i$, while $ (C,c, (\gamma^t,\gamma^s))$ was defined as the bicoequalizer of $\mathscr{X}_\mathbb{A}$. Hence we have the following equality of 2-cells

\[\begin{tikzcd}[column sep=small]
	{(TA_i,\mu_{A_i}, (\xi_{A_i}, \rho_{A_i}) )} && {(T\underset{I}{\oplaxbicolim} \; A_i, \mu_{\underset{I}{\oplaxbicolim} \; A_i}, (\xi_{\underset{I}{\oplaxbicolim} \; A_i}, \rho_{\underset{I}{\oplaxbicolim} \; A_i}))} \\
	{(A_i, a_i,(\alpha^t_i, \alpha^s_i))} & {(TB,\mu_{B}, (\xi_{B}, \rho_{B}) )} & {(C,c, (\gamma^t,\gamma^s))} \\
	& {(B,b, (\beta^t,\beta^s))}
	\arrow[""{name=0, anchor=center, inner sep=0}, "{(Tq_i, \mu_{q_i})}", from=1-1, to=1-3]
	\arrow["{(a_i,\alpha_i^s)}"', from=1-1, to=2-1]
	\arrow["{(k_i, \kappa_i)}"', from=2-1, to=3-2]
	\arrow["{(p, \pi)}", from=1-3, to=2-3]
	\arrow[""{name=1, anchor=center, inner sep=0}, "{\bicoeq{(\overline{k})}}", from=2-3, to=3-2]
	\arrow["{(Tk_i, \mu_{k_i})}"{description}, from=1-1, to=2-2]
	\arrow["{(b, \beta^s)}", from=2-2, to=3-2]
	\arrow[""{name=2, anchor=center, inner sep=0}, "{(T\langle k_i \rangle_{i \in I}, \mu_{\langle k_i \rangle_{i \in I}})}"{description}, from=1-3, to=2-2]
	\arrow["{\kappa_i \atop \simeq}"{description}, draw=none, from=2-1, to=2-2]
	\arrow["{T\nu_i \atop \simeq}"{description}, Rightarrow, draw=none, from=0, to=2-2]
	\arrow["{\psi \atop \simeq}"{description}, Rightarrow, draw=none, from=2, to=1]
\end{tikzcd}\]
\[=\begin{tikzcd}[column sep=small]
	{(TA_i,\mu_{A_i}, (\xi_{A_i}, \rho_{A_i}) )} && {(T\underset{I}{\oplaxbicolim} \; A_i, \mu_{\underset{I}{\oplaxbicolim} \; A_i}, (\xi_{\underset{I}{\oplaxbicolim} \; A_i}, \rho_{\underset{I}{\oplaxbicolim} \; A_i}))} \\
	{(A_i, a_i,(\alpha^t_i, \alpha^s_i))} && {(C,c, (\gamma^t,\gamma^s))} \\
	& {(B,b, (\beta^t,\beta^s))}
	\arrow[""{name=0, anchor=center, inner sep=0}, "{(Tq_i, \mu_{q_i})}", from=1-1, to=1-3]
	\arrow["{(a_i,\alpha_i^s)}"', from=1-1, to=2-1]
	\arrow["{(k_i, \kappa_i)}"', from=2-1, to=3-2]
	\arrow["{(p, \pi)}", from=1-3, to=2-3]
	\arrow["{\bicoeq{(\overline{k})}}", from=2-3, to=3-2]
	\arrow[""{name=1, anchor=center, inner sep=0}, "{(l_i, \lambda_i)}"{description}, from=2-1, to=2-3]
	\arrow["{\chi_i \atop \simeq}"{description}, Rightarrow, draw=none, from=0, to=1]
	\arrow["{\bicoeq(\overline{\nu}) \atop \simeq}"{description}, Rightarrow, draw=none, from=1, to=3-2]
\end{tikzcd}\]

\end{proof}

Hence, from what precedes, assuming the existence of bicoequalizers in $T\hy\psAlg$ is sufficient to construct oplax bicolimits of pseudo-algebras thanks to the process above:
\begin{corollary}
Suppose that $T\hy\psAlg$ has bicoequalizers of codescent objects. Then $T\hy\psAlg$ has oplax bicolimits. 
\end{corollary}

Finally, to conclude this section, recall that we saw at \cref{Oplax bicolimit and bicoeq of codescent generates bicolimits} that oplax bicolimits and bicoequalizers of codescent objects were sufficient to generate all bicolimits. A twofold use of codescent objects, once for constructing oplax-bicolimits and twice to correct them into actual bicolimits (thanks to a codescent diagram made of oplax bicolimits as in \cref{codescent object at the oplax colimit}), provides hence the following categorification of \cite{linton1969coequalizers}:

\begin{theorem}[Linton theorem for pseudo-algebras]\label{Linton}
Let be $ (T, \eta, \mu, (\xi, \zeta, \rho)) $ a pseudomonad on a bicocomplete category $ \mathcal{C}$. Then $ T\hy \psAlg$ is bicocomplete if and only if it has bicoequalizers of codescent objects.
\end{theorem}

\section{The case of bifinitary pseudomonads}

Now we come to the main section of this paper. It a famous result of monad theory that categories of algebras of finitary 2-monads are cocomplete: the strategy, due to \cite{barr2000toposes}, consists in proving existence of coequalizers of algebras in an indirect way, by exhibiting a weakly initial coequalizing object through a filtered bicolimit in the underlying category, using the fact that this latter is preserved by the monad. Here we shall consider the case of \emph{bifinitary} pseudomonads, that are pseudomonads preserving \emph{bifiltered bicolimits} in the sense of \cite{ODL}. 

Recall first that, as stated at \cref{psalg have bicolimit that T preserves}, if a pseudomonad preserves a certain shape of bicolimit, then pseudo-algebras inherit those bicolimits and the forgetful functor preserves them. In particular we have the following:

\begin{proposition}
Let $(T, \eta, \mu, (\xi, \zeta, \rho))$ be a bifinitary pseudomonad on a 2-category $\mathcal{C}$ with bifiltered bicolimits. Then $ T\hy\psAlg$ has bifiltered bicolimits and $ U_T$ creates them.
\end{proposition}

Now, we want to establish that, more generally, 2-categories of pseudo-algebras of a bifinitary pseudomonad are bicocomplete. We proved in section 1 that conical $\sigma$-bicolimits are enough to construct arbitrary weighted bicolimits. Then we saw that oplax bicolimits and bicoequalizers of codescent objects were enough to construct $\sigma$-bicolimits. In the previous section we saw that bicocompleteness of the 2-category of pseudo-algebras of a pseudomonad was equivalent to existence of bicoequalizers of codescent objects. Hence it suffices to establish that the 2-category of pseudo-algebras of a bifinitary pseudomonad have bicoequalizers of codescent objects to ensure bicocompletness. Again our proof will be inspired by the classics of monad theory, as \cite{borceux1994handbook}[section 4.3].

\begin{division}[Setting of the theorem]
In this section we fix $(T, \eta, \mu, (\xi, \zeta, \rho))$ a bifinitary pseudomonad on a bicomplete 2-category $ \mathcal{C}$ with bifiltered bicolimits and bicoequalizers of codescent objects; again we denote as $U_T$ the associated forgetful 2-functor. We take a codescent diagram $ \mathscr{X}: \mathbb{X} \rightarrow T\hy\psAlg$ of pseudo-algebras, with vertices and arrows denoted as
\[\begin{tikzcd}
	{(A,a,(\alpha^t,\alpha^s))} && {(B,b,(\beta^t,\beta^s))} && {(C,c,(\gamma^t,\gamma^s))}
	\arrow["{(f,\phi)}", shift left=3, from=1-1, to=1-3]
	\arrow["{(h,\psi)}"', shift right=3, from=1-1, to=1-3]
	\arrow["{(g,\chi)}"{description}, from=1-1, to=1-3]
	\arrow["{(s,\sigma)}", shift left=3, from=1-3, to=1-5]
	\arrow["{(t,\tau)}"', shift right=3, from=1-3, to=1-5]
	\arrow["{(i,\iota)}"{description}, from=1-5, to=1-3]
\end{tikzcd}\]

Now, considering the weight we give at \cref{the weight}, we obtain a 2-functor
\[\begin{tikzcd}[column sep=large]
	{T\hy\psAlg} &&& \Cat
	\arrow["{[\mathbb{X}^{\op}, \Cat][\mathcal{J}, T\hy\psAlg[ \mathscr{X},-]]}", from=1-1, to=1-4]
\end{tikzcd}\]
and exhibiting a $\mathcal{J}$-weighted bicolimit of $\mathscr{X}$ (that is, a bicoequalizer for $\mathscr{X}$) amounts to prove this 2-functor to be birepresentable.\\

\begin{division}[Birepresentability and solution set]\label{Betti}
In \cite{betti1988complete}[Theorem 3.5] it is shown that a 2-functor $ T : \mathcal{A} \rightarrow \Cat $ with $\mathcal{A}$ bicomplete (having all weighted bilimits) is birepresentable if and only the the 2-category $ \int T$ of elements of $T$ has a small \emph{weakly co-initiall} pseudofully faithful sub-2-category $ \mathcal{H} \hookrightarrow \int T  $ (see \cite{betti1988complete}[Definition 3.1]), that is satisfying the condition that for any $A$ in $\mathcal{A}$ there exists some 1-cell $ H \rightarrow A$ with $H$ in $\mathcal{H}$. Beware that $\mathcal{A}$ needs to be bicomplete for this result to hold -- for the representation will be constructed through bilimits. \\

Getting back to our situation, this tells us that to construct a bicoequalizer of $\mathscr{X}$, that is, a birepresentation of $ [\mathbb{X}^{\op}, \Cat][\mathcal{J}, T\hy\psAlg[ \mathscr{X},-]]$, it suffices to exhibit such a small co-initial family in $\int [\mathbb{X}^{\op}, \Cat][\mathcal{J}, T\hy\psAlg[ \mathscr{X},-]]$. As in the 1-categorical case, the transfinite construction we perform here will actually produce such a co-initial family consisting of a \emph{single} object (the bifiltered bicolimit of some directed chain) which will possess this weak initialness property. Beware from above that $ \mathcal{C}$ has to be bicomplete -- in order $T\hy\psAlg$ to be so by \cref{pseudoaglebras inherit bilimits} to make use of \cite{betti1988complete}. 
\end{division}

\end{division}

\begin{division}[Initialization step: lower codescent data]\label{initialization codescent lower}
We shall construct a transfinite sequence of codescent diagrams. The first one, which will be hereafter denoted $ \mathscr{X}_0$, is constructed as follows from the data above. First take the bicoequalizer in $\mathcal{C}$ of the underlying codescent object $ U_T \mathscr{X} : X \rightarrow \mathcal{C}$: 
\[\begin{tikzcd}
	C & {Q_0 = \bicoeq (U_T\mathscr{X})}
	\arrow["{q_0}", from=1-1, to=1-2]
\end{tikzcd}\]
with $ \kappa_0 : q_0 s \simeq q_0t$ its inserted invertible 2-cell. \\

Then take the bicoequalizer of the free construction over the previous codescent diagram $ TU_T\mathscr{X}$:
\[\begin{tikzcd}
	TC & {P_0 = \bicoeq (TU_T\mathscr{X})}
	\arrow["{p_0}", from=1-1, to=1-2]
\end{tikzcd}\]
with $ \pi_0 : p_0 Ts \simeq p_0Tt$ its inserted invertible 2-cell. First, it is clear that the image $Tq_0$ also pseudocoequalizes $TU_T\mathscr{X}$: whence a canonical map
\[\begin{tikzcd}
	TC & {P_0 = \bicoeq (TU_T\mathscr{X})} \\
	& {TQ_0}
	\arrow["{p_0}", from=1-1, to=1-2]
	\arrow[""{name=0, anchor=center, inner sep=0}, "{Tq_0}"', from=1-1, to=2-2]
	\arrow["{v_0}", from=1-2, to=2-2]
	\arrow["{\upsilon_0 \atop \simeq}"{description, pos=0.7}, Rightarrow, draw=none, from=0, to=1-2]
\end{tikzcd}\]

For the second map, first observe that we can use the structure maps $ a,b,c$ and the pseudomorphism structure over the morphism in $\mathscr{X}$ to construct a morphism of codescent object $ \overline{x}: TU_T\mathscr{X} \Rightarrow U_T\mathscr{X}$ in $ [\mathscr{X}, \mathcal{C}]_{\ps} $. But for $P_0$ and $Q_0$ are the respective bicoequalizer of those codescent diagrams, this induces a canonical map $q_0 = \bicoeq(x) $ together with a canonical 2-cell
\[\begin{tikzcd}
	TC & {P_0} \\
	C & {Q_0}
	\arrow["{p_0}", from=1-1, to=1-2]
	\arrow["{u_0}", from=1-2, to=2-2]
	\arrow["{c}"', from=1-1, to=2-1]
	\arrow["{q_0}"', from=2-1, to=2-2]
	\arrow["{\xi_0 \atop \simeq}"{description}, draw=none, from=2-2, to=1-1]
\end{tikzcd}\]

This determines the lower data: set $ \mathscr{X}_0(0)= TQ_0$, $\mathscr{X}_0(1)= TP_0$, while as in the 1-dimensional case, the lower projections are to be taken as $ T(u_0)$ and $\mu_{Q_0}T(v_0)$. To construct the common pseudo-section, we can again exhibit a morphism of codescent diagrams $ \overline{i} : U_T\mathscr{X} \Rightarrow TU_T\mathscr{X}$ and applying pseudofunctoriality of the bicoequalizer construction to provide a unique arrow $ i_0 = \langle p_0\eta_C \rangle$ together with a canonical 2-cell 
\[\begin{tikzcd}
	TC & {P_0} \\
	C & {Q_0}
	\arrow["{q_0}"', from=2-1, to=2-2]
	\arrow["{i_0}"', from=2-2, to=1-2]
	\arrow["{\eta_C}", from=2-1, to=1-1]
	\arrow["{p_0}", from=1-1, to=1-2]
	\arrow["{\iota_0 \atop \simeq}"{description}, draw=none, from=1-1, to=2-2]
\end{tikzcd}\]
The unit for the lower data of our desired codescent diagram will be then chosen as $Ti_0$. 
\end{division}

\begin{division}[Initialization step : higher codescent data]\label{initialization higher codescent}
To determine the higher data, we can observe that our situation is very similar to the situation of \cref{underlying codescent object}: here the analogy is obtained by replacing the computation of oplax bicolimit by the computation of the pseudocoequalizer; the combinations of data coming from the pseudo-algebraic structure, the multiplication and the comparision maps follows the same pattern. The higher object is obtained as the free construction over the bicoequalizer of the twice iterated free construction at the codescent diagram. Step by step, define $ R_0$ as the following bicoequilizer
\[\begin{tikzcd}
	TTC & {R_0= \bicoeq(TTU_T\mathscr{X})}
	\arrow["{r_0}", from=1-1, to=1-2]
\end{tikzcd}\]
We are going to generate several maps $TR_0 \rightarrow TP_0$ from this universal property.\\

For the first one, just observe that $ Tp_0$ pseudocoqualizes $ TTU_T\mathscr{X}$ since $ p_0$ pseudocoequalizes $TU_T\mathscr{X}$ by definition. This induces a first comparison map $ t_0 = \langle Tp_0\rangle$:
\[\begin{tikzcd}
	TTC & {R_0= \bicoeq(TTU_T\mathscr{X})} \\
	& {TP_0}
	\arrow["{r_0}", from=1-1, to=1-2]
	\arrow[""{name=0, anchor=center, inner sep=0}, "{Tp_0}"', from=1-1, to=2-2]
	\arrow["t_0", from=1-2, to=2-2]
	\arrow["{\tau_0 \atop \simeq}"{description}, Rightarrow, draw=none, from=0, to=1-2]
\end{tikzcd}\]
Then one can take the following composite as the first higher map
\[\begin{tikzcd}
	{TR_0} & {TTP_0} & {TP_0}
	\arrow["Tt_0", from=1-1, to=1-2]
	\arrow["{\mu_{P_0}}", from=1-2, to=1-3]
\end{tikzcd}\]
\end{division}

 For the next map, observe that as above the data of the $Ta,Tb,Tc$ define a morphism of codescent diagram $ \overline{Tx} : TTU_T \mathscr{X} \Rightarrow TU_T \mathscr{X}  $, inducing a morphism between the corresponding bicoequalizers $ w_0 = \langle p_0 Tc \rangle$ together with a universal invertible 2-cell:
\[\begin{tikzcd}
	TTC & {R_0} \\
	TC & {P_0}
	\arrow["Tc"', from=1-1, to=2-1]
	\arrow["{r_0}", from=1-1, to=1-2]
	\arrow["{w_0}", from=1-2, to=2-2]
	\arrow["{p_0}"', from=2-1, to=2-2]
	\arrow["{\omega_0 \atop \simeq}"{description}, draw=none, from=1-1, to=2-2]
\end{tikzcd}\]
 
 Similarly, the data of the multiplications $ \mu_A, \mu_B, \mu_C$ define a morphism of codescent diagrams $ \overline{m} :TTU_T \mathscr{X} \Rightarrow TU_T \mathscr{X}   $ inducing a morphism $ m_0 = \langle p_0 \mu_C \rangle$ and a universal invertible 2-cell:
\[\begin{tikzcd}
	TTC & {R_0} \\
	TC & {P_0}
	\arrow["{\mu_C}"', from=1-1, to=2-1]
	\arrow["{r_0}", from=1-1, to=1-2]
	\arrow["{m_0}", from=1-2, to=2-2]
	\arrow["{p_0}"', from=2-1, to=2-2]
	\arrow["{\mu_0 \atop \simeq}"{description}, draw=none, from=1-1, to=2-2]
\end{tikzcd}\]

\begin{lemma}\label{first step: the codescent diagram}
The following diagram, which we shall denote $ \mathscr{X}_0$, defines a codescent object:
\[\begin{tikzcd}
	{TR_0} && {TP_0} && {TQ_0}
	\arrow["{Tm_0}"{description}, from=1-1, to=1-3]
	\arrow["{Tw_0}", shift left=4, from=1-1, to=1-3]
	\arrow["{\mu_{P_0}Tt_0}"', shift right=4, from=1-1, to=1-3]
	\arrow["{Ti_0}"{description}, from=1-5, to=1-3]
	\arrow["{Tu_0}", shift left=4, from=1-3, to=1-5]
	\arrow["{\mu_{Q_0}Tv_0}"', shift right=4, from=1-3, to=1-5]
\end{tikzcd}\]
\end{lemma}

\begin{proof}
We have to guess invertible 2-cells between the correct combinations of the higher and lower codescent data. All of them will be constructed thanks to some corresponding invertible 2-cells in $ [\mathbb{X}, \mathcal{C}]_\ps$ and then applying pseudofunctoriality of the bicoequalizer construction, though some will also involve further calculations.\\

One can check we have the following invertible 2-cell in $ [\mathbb{X}, \mathcal{C}]_\ps$:
\[\begin{tikzcd}
	{U_T\mathscr{X}} && {U_T\mathscr{X}} \\
	& {TU_T\mathscr{X}}
	\arrow[""{name=0, anchor=center, inner sep=0}, Rightarrow, no head, from=1-1, to=1-3]
	\arrow["\overline{i}"', from=1-1, to=2-2]
	\arrow["\overline{x}"', from=2-2, to=1-3]
	\arrow["{\overline{n_0} \atop\simeq}"{description}, Rightarrow, draw=none, from=0, to=2-2]
\end{tikzcd}\]
induced from the triangle 2-cells of the pseudo-algebra structure $ \alpha^t$, $ \beta^t$, $ \gamma^t$ (where $ \overline{w}$ was constructed in \cref{initialization codescent lower}). Applying the pseudofunctor $ \bicoeq(-)$ returns hence a 2-cell
\[\begin{tikzcd}
	{Q_0} && {Q_0} \\
	& {P_0}
	\arrow[""{name=0, anchor=center, inner sep=0}, Rightarrow, no head, from=1-1, to=1-3]
	\arrow["{i_0}"', from=1-1, to=2-2]
	\arrow["{u_0}"', from=2-2, to=1-3]
	\arrow["{n_0 \atop \simeq}"{description}, Rightarrow, draw=none, from=0, to=2-2]
\end{tikzcd}\]
and it suffices then to take $Tn_0$ as the first pseudo-retraction corresponding to $n_0$ of \cref{codescent diagram}.\\

For the second retraction, we saw we had a morphism of codescent object $ \overline{i} : U_T\mathscr{X} \rightarrow TU_T\mathscr{X}$: but observe that $ [2, [\mathscr{X}, \mathcal{C}]_\ps]_\ps \simeq [\mathscr{X}, [2, \mathcal{C}]_\ps]_\ps$: hence $ \overline{i}$ also defines a codescent diagram $ \mathscr{X}_i$ in $[2, \mathcal{C}]$. But on the other hand, the following invertible 2-cell
\[\begin{tikzcd}
	TC & {TQ_0} \\
	C & {Q_0}
	\arrow["{Tq_0}", from=1-1, to=1-2]
	\arrow[""{name=0, anchor=center, inner sep=0}, "{\eta_{Q_0}}"', from=2-2, to=1-2]
	\arrow[""{name=1, anchor=center, inner sep=0}, "{\eta_C}", from=2-1, to=1-1]
	\arrow["{q_0}"', from=2-1, to=2-2]
	\arrow["{\eta_{q_0} \atop \simeq}"{description}, Rightarrow, draw=none, from=1, to=0]
\end{tikzcd}\]
is pseudocoequalizing for $ \mathscr{X}_i $: for the bicoequalizer of $ \mathscr{X}_i$ in $ [2, \mathcal{C}]$ coincides with the universal 2-cell from which we induced $i_0$ and $\iota_0$ at the end of \cref{initialization codescent lower}, its universal property implies the existence of a unique 2-cell $ \nu $ satisfying the equations below
\[\begin{tikzcd}[row sep=small]
	& {P_0} \\
	TC && {TQ_0} \\
	\\
	C && {Q_0}
	\arrow[""{name=0, anchor=center, inner sep=0}, "{Tq_0}"{description}, from=2-1, to=2-3]
	\arrow[""{name=1, anchor=center, inner sep=0}, "{\eta_{Q_0}}"', from=4-3, to=2-3]
	\arrow[""{name=2, anchor=center, inner sep=0}, "{\eta_C}", from=4-1, to=2-1]
	\arrow["{q_0}"', from=4-1, to=4-3]
	\arrow["{p_0}", from=2-1, to=1-2]
	\arrow["{v_0}", from=1-2, to=2-3]
	\arrow["{\eta_{q_0} \atop \simeq}"{description}, Rightarrow, draw=none, from=2, to=1]
	\arrow["{\upsilon_0 \atop \simeq}"{description, pos=0.6}, Rightarrow, draw=none, from=0, to=1-2]
\end{tikzcd} =
\begin{tikzcd}[row sep=small]
	& {P_0} \\
	TC && {TQ_0} \\
	& {Q_0} \\
	C && {Q_0}
	\arrow[""{name=0, anchor=center, inner sep=0}, "{\eta_{Q_0}}"', from=4-3, to=2-3]
	\arrow["{\eta_C}", from=4-1, to=2-1]
	\arrow["{q_0}"', from=4-1, to=4-3]
	\arrow["{p_0}", from=2-1, to=1-2]
	\arrow["{v_0}", from=1-2, to=2-3]
	\arrow["{i_0}"{description}, ""{name=1, anchor=center, inner sep=0}, from=3-2, to=1-2]
	\arrow["{q_0}"{description}, from=4-1, to=3-2]
	\arrow[Rightarrow, no head, from=3-2, to=4-3]
	\arrow["{\iota_0 \atop \simeq}"{description}, draw=none, from=2-1, to=3-2]
	\arrow["{\nu_0 \atop \simeq}"{description}, Rightarrow, draw=none, from=1, to=0]
\end{tikzcd}\]
Then it suffices to paste $T\nu_0$ together with the free canonical triangular 2-cell as below to get the desired retraction corresponding to $n_1$
\[\begin{tikzcd}[row sep=large]
	{TP_0} && {TTQ_0} \\
	{TQ_0} && {TQ_0}
	\arrow["{Tv_0}", from=1-1, to=1-3]
	\arrow["{Ti_0}", from=2-1, to=1-1]
	\arrow[""{name=0, anchor=center, inner sep=0}, "{T\eta_{Q_0}}"{description}, from=2-1, to=1-3]
	\arrow["{\mu_{Q_0}}", from=1-3, to=2-3]
	\arrow[Rightarrow, no head, from=2-1, to=2-3]
	\arrow["{T\nu_0 \atop \simeq}"{description, pos=0.3}, Rightarrow, draw=none, from=1-1, to=0]
	\arrow["{\zeta_ {Q_0} \atop \simeq}"{pos=0.7}, Rightarrow, draw=none, from=0, to=2-3]
\end{tikzcd}\]

We turn now to the higher composition. First, one can check we have the following invertible 2-cell in $ [\mathscr{X}, \mathcal{C}]_\ps$ induced from the square 2-cells of the pseudo-algebra structure $ \alpha^s$, $ \beta^s$, $ \gamma^s$:
\[\begin{tikzcd}
	{TTU_T\mathscr{X}} & {TU_T\mathscr{X}} \\
	{TU_T\mathscr{X}} & {U_T\mathscr{X}}
	\arrow["\overline{x }"', from=2-1, to=2-2]
	\arrow["\overline{y }"', from=1-1, to=2-1]
	\arrow["\overline{Tx}", from=1-1, to=1-2]
	\arrow["\overline{x }", from=1-2, to=2-2]
	\arrow["{\overline{\theta} \atop \simeq}"{description}, draw=none, from=2-2, to=1-1]
\end{tikzcd}\]

Again, applying $ \bicoeq(-)$ returns an invertible 2-cell 
\[\begin{tikzcd}
	{R_0} & {P_0} \\
	{P_0} & {Q_0}
	\arrow["{w_0}", from=1-1, to=1-2]
	\arrow["{u_0}", from=1-2, to=2-2]
	\arrow["{m_0}"', from=1-1, to=2-1]
	\arrow["{u_0}"', from=2-1, to=2-2]
	\arrow["{\theta \atop \simeq}"{description}, draw=none, from=2-2, to=1-1]
\end{tikzcd}\]
and one just has to take $ T\theta$ as the first desired higher 2-cell corresponding to $\theta_{01}$. \\

Now observe again that the morphism of codescent object $ \overline{Tx}$ can also be seen as a codescent object $ \mathscr{X}_{Tx}$ in $ [2, \mathcal{C}]_\ps$, whose bicoequalizer is given by the pair $(w_0,\omega_0)$ as defined in \cref{initialization higher codescent}. But the pseudosquare $ T \xi_0 $ (where $\xi_0$ was defined at \cref{initialization codescent lower}), seen as a morphism in $[2, \mathcal{C}]_\ps$, pseudocoequalizes this codescent object $ \mathscr{X}_{Tx}$ as $ \xi_0$ was itself the bicoequalizer of $\mathscr{X}_{x}$, the codescent object associated to $\overline{x}$ in $ [2, \mathcal{C}]_\ps$: hence we have a 2-cell $ \langle T\xi_0 \rangle$ as below satisfying the following equation 
\[\begin{tikzcd}[row sep=small]
	& {R_0} \\
	TTC && {TP_0} \\
	\\
	TC && {TQ_0}
	\arrow["TTc"', from=2-1, to=4-1]
	\arrow[""{name=0, anchor=center, inner sep=0}, "{Tp_0}"{description}, from=2-1, to=2-3]
	\arrow["{Tq_0}"', from=4-1, to=4-3]
	\arrow["{Tu_0}", from=2-3, to=4-3]
	\arrow["{r_0}", from=2-1, to=1-2]
	\arrow["{t_0}", from=1-2, to=2-3]
	\arrow["{T\xi_0 \atop \simeq}"{description}, draw=none, from=2-1, to=4-3]
	\arrow["{\tau_0 \atop \simeq}"{description, pos=0.6}, Rightarrow, draw=none, from=0, to=1-2]
\end{tikzcd}
=\begin{tikzcd}[row sep=small]
	& {R_0} \\
	TTC && {TP_0} \\
	& {P_0} \\
	TC && {TQ_0}
	\arrow["TTc"', from=2-1, to=4-1]
	\arrow[""{name=0, anchor=center, inner sep=0}, "{Tq_0}"', from=4-1, to=4-3]
	\arrow[""{name=1, anchor=center, inner sep=0}, "{Tu_0}", from=2-3, to=4-3]
	\arrow["{r_0}", from=2-1, to=1-2]
	\arrow["{t_0}", from=1-2, to=2-3]
	\arrow[""{name=2, anchor=center, inner sep=0}, "{w_0}"{description}, from=1-2, to=3-2]
	\arrow["{p_0}"{description}, from=4-1, to=3-2]
	\arrow["{v_0}"{description}, from=3-2, to=4-3]
	\arrow["{\omega_0 \atop \simeq}"{description}, draw=none, from=2-1, to=3-2]
	\arrow["{\upsilon_0 \atop \simeq}"{description}, Rightarrow, draw=none, from=3-2, to=0]
	\arrow["{\langle T\xi_0 \rangle \atop \simeq}"{description}, Rightarrow, draw=none, from=2, to=1]
\end{tikzcd}\]
Now take the following pasting as the second higher 2-cell corresponding to $ \theta_{02}$:
\[\begin{tikzcd}
	{TR_0} & {TP_0} \\
	{TTP_0} & {TTQ_0} \\
	{TP_0} & {TQ_0}
	\arrow["{Tw_0}", from=1-1, to=1-2]
	\arrow["{Tv_0}", from=1-2, to=2-2]
	\arrow["{Tt_0}"', from=1-1, to=2-1]
	\arrow[""{name=0, anchor=center, inner sep=0}, "{TTu_0}"', from=2-1, to=2-2]
	\arrow["{\mu_{P_0}}"', from=2-1, to=3-1]
	\arrow[""{name=1, anchor=center, inner sep=0}, "{Tu_0}"', from=3-1, to=3-2]
	\arrow["{\mu_{Q_0}}", from=2-2, to=3-2]
	\arrow["{T\langle T\xi_0 \rangle \atop \simeq}"{description}, draw=none, from=1-1, to=2-2]
	\arrow["{\mu_{u_0} \atop \simeq}"{description}, Rightarrow, draw=none, from=0, to=1]
\end{tikzcd}\]
For the last 2-cell observe that the following pasting 
\[\begin{tikzcd}
	{R_0} && {TP_0} \\
	TTC &&& {TTQ_0} \\
	TC &&& {TQ_0} \\
	&& {P_0}
	\arrow[""{name=0, anchor=center, inner sep=0}, "{\mu_C}"', from=2-1, to=3-1]
	\arrow[""{name=1, anchor=center, inner sep=0}, "{Tp_0}"{description}, from=2-1, to=1-3]
	\arrow[""{name=2, anchor=center, inner sep=0}, "{Tv_0}", from=1-3, to=2-4]
	\arrow[""{name=3, anchor=center, inner sep=0}, "{\mu_{Q_0}}", from=2-4, to=3-4]
	\arrow[""{name=4, anchor=center, inner sep=0}, "{p_0}"', from=3-1, to=4-3]
	\arrow[""{name=5, anchor=center, inner sep=0}, "{v_0}"', from=4-3, to=3-4]
	\arrow["{r_0}", from=2-1, to=1-1]
	\arrow["{t_0}", from=1-1, to=1-3]
	\arrow["{TTq_0}"{description}, from=2-1, to=2-4]
	\arrow["{Tq_0}"{description}, from=3-1, to=3-4]
	\arrow["{\mu_{q_0} \atop \simeq}"{description}, draw=none, from=0, to=3]
	\arrow["{\tau_0 \atop \simeq}"{description}, draw=none, from=1-1, to=1]
	\arrow["{T\upsilon_0 \atop \simeq}"{description, pos=0.6}, shift right=1, draw=none, from=1, to=2]
	\arrow["{\upsilon_0^{-1} \atop \simeq}"{description, pos=0.6}, shift left=1, draw=none, from=4, to=5]
\end{tikzcd}\]
induces by functoriality of bicoequalizers the following 2-cell $ \tau$
\[\begin{tikzcd}
	TTC & {R_0} & {TP_0} & {TTQ_0} \\
	TC & {P_0} && {TQ_0}
	\arrow["{\mu_C}"', from=1-1, to=2-1]
	\arrow["{Tv_0}", from=1-3, to=1-4]
	\arrow["{\mu_{Q_0}}", from=1-4, to=2-4]
	\arrow["{p_0}"', from=2-1, to=2-2]
	\arrow["{v_0}"', from=2-2, to=2-4]
	\arrow["{r_0}", from=1-1, to=1-2]
	\arrow["{t_0}", from=1-2, to=1-3]
	\arrow["{m_0}"{description}, from=1-2, to=2-2]
	\arrow["{\mu_0 \atop \simeq}"{description}, draw=none, from=1-1, to=2-2]
	\arrow["{\tau \atop \simeq}"{description}, draw=none, from=1-2, to=2-4]
\end{tikzcd}\]
and now we can take as the last codescent 2-cell, corresponding to $\theta_{12}$, the following pasting
\[\begin{tikzcd}
	{TR_0} & {TTP_0} & {TP_0} \\
	& {TTTQ_0} & {TTQ_0} \\
	{TP_0} & {TQ_0} & {TQ_0}
	\arrow["{TTv_0}"{description}, from=1-2, to=2-2]
	\arrow["{T\mu_{Q_0}}"{description}, from=2-2, to=3-2]
	\arrow["{Tv_0}"', from=3-1, to=3-2]
	\arrow["{Tt_0}", from=1-1, to=1-2]
	\arrow["{Tm_0}"', from=1-1, to=3-1]
	\arrow["{T\tau \atop \simeq}"{description}, draw=none, from=1-1, to=3-2]
	\arrow["{\mu_{P_0}}", from=1-2, to=1-3]
	\arrow["{Tv_0}", from=1-3, to=2-3]
	\arrow["{\mu_{TQ_0}}"', from=2-2, to=2-3]
	\arrow["{\mu_{Q_0}}", from=2-3, to=3-3]
	\arrow["{\mu_{Q_0}}"', from=3-2, to=3-3]
	\arrow["{\rho_{Q_0} \atop \simeq}"{description}, draw=none, from=3-3, to=2-2]
	\arrow["{\mu_{v_0} \atop \simeq}"{description}, draw=none, from=1-2, to=2-3]
\end{tikzcd}\]
\end{proof}

We have hence constructed a first codescent diagram $ \mathscr{X}_0$ from the data of $\mathscr{X}$. Let us now turn to the induction step:

\begin{division}[Codescent diagram at an induction step]\label{induction step setting}
Suppose that the codescent diagram $ \mathscr{X}_n$ is given for some $n \in \mathbb{N}$, consisting of the data 
\[\begin{tikzcd}
	{TR_n} && {TP_n} && {TQ_n}
	\arrow["{\mu_{P_n}Tt_n}"{description}, from=1-1, to=1-3]
	\arrow["{Tw_n}", shift left=4, from=1-1, to=1-3]
	\arrow["{Tm_n}"', shift right=4, from=1-1, to=1-3]
	\arrow["{Ti_n}"{description}, from=1-5, to=1-3]
	\arrow["{Tu_n}", shift left=4, from=1-3, to=1-5]
	\arrow["{\mu_{Q_n}Tv_n}"', shift right=4, from=1-3, to=1-5]
\end{tikzcd}\]
with the different arrows $ u_n : P_n \rightarrow Q_n$, $ v_n : P_n \rightarrow TQ_n$, $ w_n : R_n \rightarrow P_n$, $t_n : R_n \rightarrow TP_n $ and $ m_n : R_n \rightarrow P_n$ being given at the induction step; suppose moreover, for $ n \geq 1$, that it is related to its predecessor step through the condition that $ v_n = Tu_n T\eta_{Q_{n-1}}$ (for the corresponding relation will be established at the next step later on). We must now construct a new codescent object $ \mathscr{X}_{n+1}$ from the data above. 
\end{division}

\begin{division}[Induction step: lower data]
First define $ Q_{n+1}  $ and $u_{n+1}$ as the bicoequalizer
\[\begin{tikzcd}
	{TQ_n} & {Q_{n+1}=\bicoeq({\mathscr{X}_n})}
	\arrow["{u_{n+1}}", from=1-1, to=1-2]
\end{tikzcd}\]
and set the following identity $ P_{n+1} = TQ_n$, and take $ v_{n+1} = Tq_{n+1}$ where the next $ q_{n+1}$ is defined as the composite $ q_{n+1} = u_{n+1}\eta_{Q_n}$: that is we take $v_{n+1}$ as the composite
\[\begin{tikzcd}
	{P_{n+1}=TQ_n} && {TQ_{n+1}}
	\arrow["{u_{n+1}\eta_{Q_n}}", from=1-1, to=1-3]
\end{tikzcd}\]

\begin{division}[Common pseudosection - first step]
To construct the common pseudosection, we need to get back to the step $n-1$ (we suppose here $n \geq 1$, addressing the case $n=1$ in a further remark) to extract intermediate comparison maps from which we will induce $i_{n+1}$ through a sequence of universal properties. It would be tempting to use directly $ T\eta_{Q_n}$ to induce $i_{n+1}$ through the universal of $u_{n+1}$, but the unit by itself fails to pseudocoequalize the codescent diagram $ \mathscr{X}_n$. We must fist compose it with a map that is both expected to vanish when we shall achieve the transfinite induction, but also ensure pseudocoequalization. \\

As a first step, if both $ \mathscr{X}_{n-1}$ and $\mathscr{X}_n$ have been defined in such a way that $v_n = Tu_n T\eta_{Q_{n-1}}$, then observe that $v_n$ pseudocoequalizes $ \mathscr{X}_{n-1}$ thanks to the whiskering $ T\xi_n * T\eta_{Q_{n-1}}$. As a consequence, the universal property of $ u_n$, which is the pseudocoequalizing 1-cell of $ \mathscr{X}_{n-1}$, ensures the existence of a canonical invertible 1-cell 

\[\begin{tikzcd}
	{TQ_{n-1}} & {TT{Q_{n-1}}} && {TQ_n} \\
	{Q_n}
	\arrow["{u_n}"', from=1-1, to=2-1]
	\arrow[""{name=0, anchor=center, inner sep=0}, "{\overline{v_n}}"', from=2-1, to=1-4]
	\arrow["{T\eta_{Q_{n-1}}}", from=1-1, to=1-2]
	\arrow["{Tu_n}", from=1-2, to=1-4]
	\arrow["{\nu_n \atop \simeq}"{description}, draw=none, from=1-1, to=0]
\end{tikzcd}\]


\end{division}



\begin{lemma}\label{intermediate pseudocoequalization}
For any $n \geq 1$, the 2-cell $\mu_{Q_n}T\overline{v_n}$ pseudocoequalizes the codescent diagram $ \mathscr{X}_n$. 
\end{lemma}

\begin{proof}

Consider the following pasting:
\[\begin{tikzcd}[row sep=large]
	{TTQ_{n-1}} &[8pt] &&& {TQ_n} \\
	& {TTTQ_{n-1}} \\
	{TTTQ_{n-1}} & {TTTTQ_{n-1}} && {TTTQ_{n-1}} \\
	{TTQ_n} && {TTTQ_n} && {TTQ_n} \\
	{TQ_n} && {TTQ_n} && {TQ_n}
	\arrow["{Tu_n}", from=1-1, to=1-5]
	\arrow["{\mu_{Q_n}}"', from=4-1, to=5-1]
	\arrow["{T\overline{v_n}}", from=1-5, to=4-5]
	\arrow[""{name=0, anchor=center, inner sep=0}, "{TT\eta_{Q_{n-1}}}"', from=1-1, to=3-1]
	\arrow[""{name=1, anchor=center, inner sep=0}, "{TTu_{n}}"', from=3-1, to=4-1]
	\arrow[""{name=2, anchor=center, inner sep=0}, "{T\overline{v_n}}"', from=5-1, to=5-3]
	\arrow["{\mu_{Q_n}}", from=4-5, to=5-5]
	\arrow["{\mu_{Q_n}}"', from=5-3, to=5-5]
	\arrow[""{name=3, anchor=center, inner sep=0}, "{TTu_n}"{description}, from=3-4, to=4-5]
	\arrow["{\mu_{TQ_n}}"{description}, from=4-3, to=5-3]
	\arrow["{T\mu_{Qn}}"{description}, from=4-3, to=4-5]
	\arrow[""{name=4, anchor=center, inner sep=0}, "{TT\overline{v_n}}"{description}, from=4-1, to=4-3]
	\arrow["{TTTu_n}"{description}, from=3-2, to=4-3]
	\arrow["{T\mu_{TQ_{n-1}}}"{description}, from=3-2, to=3-4]
	\arrow["{\rho_{Q_n} \atop \simeq}"{description}, draw=none, from=4-3, to=5-5]
	\arrow["{T\nu_n \atop \simeq}"{description}, draw=none, from=3-4, to=1-5]
	\arrow["{TT\eta_{Q_{n-1}}}"{description}, from=1-1, to=2-2]
	\arrow[""{name=5, anchor=center, inner sep=0}, "{TT\eta_{TQ_{n-1}}}"{description}, from=2-2, to=3-2]
	\arrow["{TTT\eta_{Q_{n-1}}}", from=3-1, to=3-2]
	\arrow[""{name=6, anchor=center, inner sep=0}, Rightarrow, no head, from=2-2, to=3-4]
	\arrow["{\mu_{\nu_n} \atop \simeq}"{description}, draw=none, from=2, to=4]
	\arrow["{T\mu_{u_n} \atop \simeq}"{description}, draw=none, from=3-2, to=3]
	\arrow["{TT\eta_{\eta_{Q_{n-1}}} \atop \simeq}"{description}, draw=none, from=0, to=5]
	\arrow["{TT\nu_n \atop \simeq}"{description}, draw=none, from=1, to=4]
	\arrow["{T\zeta_{TQ_n} \atop \simeq}"{description, pos=0.4}, draw=none, from=3-2, to=6]
\end{tikzcd}\]

\end{proof}

\begin{division}[Common pseudosection - second step]
From \cref{intermediate pseudocoequalization}, the arrow $ T\overline{v_n}$ pseudocoequalizes the codescent diagram $\mathscr{X}_n$; then so does its post-composition with $\mu_{Q_n}$, inducing hence a factorization through its bicoequalizer $ u_{n+1}$ as below, which we can take as the desired $i_{n+1}$:

\[\begin{tikzcd}
	{TQ_n} & {TTQ_n} & {TQ_n=P_{n+1}} \\
	{Q_{n+1}}
	\arrow["{u_{n+1}}"', from=1-1, to=2-1]
	\arrow["{T\overline{v_n}}", from=1-1, to=1-2]
	\arrow["{\mu_{Q_n}}", from=1-2, to=1-3]
	\arrow[""{name=0, anchor=center, inner sep=0}, "{{i_{n+1}}}"', from=2-1, to=1-3]
	\arrow["{\iota_n \atop \simeq}"{description}, draw=none, from=1-1, to=0]
\end{tikzcd}\]

\end{division}


%
\end{division}

\begin{division}[Induction step: higher data]
Concerning higher data, one must now declare $ R_{n+1} = TP_{n}$. Now we construct the following higher maps as follows:\begin{itemize}
    \item the $n$th step gives us a map $ u_n : P_n \rightarrow Q_n$, whence $ Tu_n : R_{n+1} = TP_n \rightarrow P_{n+1}=TQ_n$. Then takes simply $w_{n+1} = Tu_n$, and then its image along $T$:
\[\begin{tikzcd}
	{TR_{n+1}} & {TP_{n+1}}
	\arrow["{TTu_n}", from=1-1, to=1-2]
\end{tikzcd}\]
\item we also have the composite $ \mu_{Q_n} Tv_n : R_{n+1} = TP_n \rightarrow P_{n+1} =TQ_n$. Then define $ t_{n+1}= Tv_n$ and take its image along $T$ followed by the multiplication: 
\[\begin{tikzcd}[sep=huge]
	{TR_{n+1}} && {TP_{n+1}}
	\arrow["{\mu_{TQ_n}Tt_{n+1}}", from=1-1, to=1-3]
\end{tikzcd}\]
\item for the last one we use the multiplication of $Q_n$ to determine $m_{n+1}= \mu_{Q_n}Tv_n$: then take the map:
\[\begin{tikzcd}
	{TR_{n+1}} && {TP_{n+1}=TTQ_n}
	\arrow["{Tm_{n+1}}", from=1-1, to=1-3]
\end{tikzcd}\]
\end{itemize}
\end{division}

\begin{lemma}\label{reccurence step returns a codescent diagram}
Those data define altogether a codescent object we can take as $ \mathscr{X}_{n+1}$:
\[\begin{tikzcd}
	{TR_{n+1}} &&& {TP_{n+1}} &&& {TQ_{n+1}}
	\arrow["{\mu_{Q_{n+1}}Tv_{n+1}}"', shift right=4, from=1-4, to=1-7]
	\arrow["{Tu_{n+1}}", shift left=4, from=1-4, to=1-7]
	\arrow["{Tw_{n+1}}", shift left=4, from=1-1, to=1-4]
	\arrow["{Ti_{n+1}}"{description}, from=1-7, to=1-4]
	\arrow["{Tm_{n+1}}"', shift right=4, from=1-1, to=1-4]
	\arrow["{\mu_{P_{n+1}}Tt_{n+1}}"{description}, from=1-1, to=1-4]
\end{tikzcd}\]
\end{lemma}

\begin{proof}
We have to exhibit the invertible 2-cells corresponding to the coherence data. First, we must construct the invertible 2-cells exhibiting $ Ti_{n+1}$ as a common pseudosection of $ Tu_{n+1}$ and $ \mu_{Q_{n+1}} Tv_{n+1}$. For the first one, consider the pasting 
\[\begin{tikzcd}
	{TTQ_{n-1}} & {TQ_n} && {Q_{n+1}} \\
	{TQ_n} & {TTTQ_{n-1}} & {TTQ_n} & {P_{n+1}=TQ_n} \\
	{Q_{n+1}} &&& {Q_{n+1}}
	\arrow["{i_{n+1}}", from=1-4, to=2-4]
	\arrow["{u_{n+1}}", from=2-4, to=3-4]
	\arrow["{u_{n+1}}", from=1-2, to=1-4]
	\arrow["{T\overline{v_n}}"{description}, from=1-2, to=2-3]
	\arrow["{\mu_{Q_n}}"', from=2-3, to=2-4]
	\arrow["{Tu_n}", from=1-1, to=1-2]
	\arrow["{TT\eta_{Q_{n-1}}}"{description}, from=1-1, to=2-2]
	\arrow[""{name=0, anchor=center, inner sep=0}, "{TTu_{n}}"', from=2-2, to=2-3]
	\arrow["{\iota_n \atop \simeq}"{description}, draw=none, from=2-3, to=1-4]
	\arrow["{T\nu_n \atop \simeq}"{description}, draw=none, from=2-2, to=1-2]
	\arrow["{Tu_n}"', from=1-1, to=2-1]
	\arrow["{u_{n+1}}"', from=2-1, to=3-1]
	\arrow[""{name=1, anchor=center, inner sep=0}, Rightarrow, no head, from=3-1, to=3-4]
	\arrow["{\xi_n \atop \simeq}"{description}, draw=none, from=0, to=1]
\end{tikzcd}\]


This pasting inserts an invertible 2-cell $ \alpha : u_{n+1}i_{n+1} u_{n+1} Tu_n \Rightarrow u_{n+1} Tu_n $. But we know from the assumption that $ \mathscr{X}_{n}$ is a codescent object that $ Tu_n$ admits a pseudosection $Ti_n $: hence by general properties of pseudoretracts, $Tu_n$ is co-fully faithful, hence the 2-cell $\alpha$ comes uniquely from a 2-cell $ \alpha' : u_{n+1}i_{n+1} u_{n+1} \Rightarrow u_{n+1}$, and then one can use the universal property of $u_{n+1}$ as the bicoequalizer of $ \mathscr{X}_n$ to retrieve from $\alpha'$ the desired 2-cell $  u_{n+1} i_{n+1} \simeq 1_{Q_{n+1}}$ corresponding to $n_0$ of \cref{codescent diagram}.\\

For the second one, we precompose again $ \mu_{Q_{n+1}} T_{v_{n+1}} Ti_{n+1} $ with another combination of canonical, this time with the sequence $ \eta_{Q_{n+1}} u_{n+1} \eta_{Q_n} u_n \eta_{Q_{n+1}}$ and construct a first invertible 2-cell as the following pasting:
\[\begin{tikzcd}
	{Q_{n-1}} & {TQ_{n-1}} & {Q_n} & {TQ_n} & {Q_{n+1}} & {TQ_{n+1}} \\
	{TQ_{n-1}} & {TTQ_{n-1}} && {TTQ_n} && {TTQ_n} \\
	{Q_n} && {TQ_n} && {TQ_n} \\
	&&& {TTTQ_n} && {TTTQ_n} \\
	{TQ_n} && {TQ_n} && {TTQ_n} \\
	&&& {TTQ_{n+1}} && {TQ_{n+1}} \\
	{Q_{n+1}} && {TQ_{n+1}} && {TQ_{n+1}} & {TQ_{n+1}}
	\arrow["{\eta_{Q_{n+1}}}", from=1-5, to=1-6]
	\arrow["{Ti_{n+1}}", from=1-6, to=2-6]
	\arrow["{TT\eta_{Q_n}}", from=2-6, to=4-6]
	\arrow["{TTu_{n+1}}", from=4-6, to=6-6]
	\arrow["{\mu_{Q_{n+1}}}", from=6-6, to=7-6]
	\arrow["{i_{n+1}}"{description}, from=1-5, to=3-5]
	\arrow["{\eta_{TQ_{n}}}"{description}, from=3-5, to=2-6]
	\arrow[""{name=0, anchor=center, inner sep=0}, "{T\eta_{Q_n}}"{description}, from=3-5, to=5-5]
	\arrow["{\eta_{TTQ_{n}}}"{description}, from=5-5, to=4-6]
	\arrow[""{name=1, anchor=center, inner sep=0}, "{Tu_{n+1}}"{description}, from=5-5, to=7-5]
	\arrow[""{name=2, anchor=center, inner sep=0}, "{\eta_{TQ_{n+1}}}", from=7-5, to=6-6]
	\arrow[Rightarrow, no head, from=7-5, to=7-6]
	\arrow[""{name=3, anchor=center, inner sep=0}, "{u_{n+1}}", from=1-4, to=1-5]
	\arrow["{\eta_{Q_n}}", from=1-3, to=1-4]
	\arrow["{\mu_{TQ_n}}"{description}, from=4-4, to=5-5]
	\arrow[""{name=4, anchor=center, inner sep=0}, "{\mu_{Q_n}}"{description}, from=2-4, to=3-5]
	\arrow[""{name=5, anchor=center, inner sep=0}, "{T\overline{v_n}}"{description}, from=1-4, to=2-4]
	\arrow[""{name=6, anchor=center, inner sep=0}, "{TT\eta_{Q_n}}"{description}, from=2-4, to=4-4]
	\arrow[""{name=7, anchor=center, inner sep=0}, "{TTu_{n+1}}"{description}, from=4-4, to=6-4]
	\arrow["{\mu_{Q_{n+1}}}"{description}, from=6-4, to=7-5]
	\arrow[""{name=8, anchor=center, inner sep=0}, "{\overline{v_n}}"{description}, from=1-3, to=3-3]
	\arrow["{\eta_{TQ_n}}"{description}, from=3-3, to=2-4]
	\arrow[""{name=9, anchor=center, inner sep=0}, "{T\eta_{Q_n}}"{description}, from=3-3, to=5-3]
	\arrow["{\eta_{TQ_n}}"{description}, from=5-3, to=4-4]
	\arrow[""{name=10, anchor=center, inner sep=0}, "{Tu_{n+1}}"{description}, from=5-3, to=7-3]
	\arrow["{\eta_{TQ_{n+1}}}"{description}, from=7-3, to=6-4]
	\arrow[""{name=11, anchor=center, inner sep=0}, Rightarrow, no head, from=7-3, to=7-5]
	\arrow["{u_n}", from=1-2, to=1-3]
	\arrow[""{name=12, anchor=center, inner sep=0}, "{T\eta_{Q_{n-1}}}"{description}, from=1-2, to=2-2]
	\arrow["{Tu_n}"{description}, from=2-2, to=3-3]
	\arrow["{\eta_{Q_{n-1}}}", from=1-1, to=1-2]
	\arrow[""{name=13, anchor=center, inner sep=0}, "{\eta_{TQ_{n-1}}}"', from=2-1, to=2-2]
	\arrow[""{name=14, anchor=center, inner sep=0}, "{\eta_{Q_{n-1}}}"', from=1-1, to=2-1]
	\arrow["{u_n}"', from=2-1, to=3-1]
	\arrow[""{name=15, anchor=center, inner sep=0}, "{\eta_{Q_{n}}}"{description}, from=3-1, to=3-3]
	\arrow["{\eta_{Q_{n}}}"', from=3-1, to=5-1]
	\arrow[""{name=16, anchor=center, inner sep=0}, "{\eta_{TQ_{n}}}"{description}, from=5-1, to=5-3]
	\arrow["{u_{n+1}}"', from=5-1, to=7-1]
	\arrow[""{name=17, anchor=center, inner sep=0}, "{\eta_{Q_{n+1}}}"', from=7-1, to=7-3]
	\arrow["{\eta_{i_{n+1}} \atop \simeq}"{description}, draw=none, from=1-5, to=2-6]
	\arrow["{\eta_{T\eta_{Q_n}} \atop \simeq}"{description}, draw=none, from=3-5, to=4-6]
	\arrow["{\nu_{n} \atop \simeq}"{description}, draw=none, from=2-2, to=1-3]
	\arrow["{\eta_{Tu_{n+1}} \atop \simeq}"{description}, draw=none, from=5-5, to=6-6]
	\arrow["{\mu_{u_{n+1}} \atop \simeq}"{description}, draw=none, from=7, to=1]
	\arrow["{\mu_{\eta_{Q_n}} \atop \simeq}"{description}, draw=none, from=6, to=0]
	\arrow["{\eta_{\overline{v_n}} \atop \simeq}"{description}, draw=none, from=8, to=5]
	\arrow["{\eta_{T\eta_{Q_n}} \atop \simeq}"{description}, draw=none, from=9, to=6]
	\arrow["{\eta_{u_{n+1}} \atop \simeq}"{description}, draw=none, from=16, to=17]
	\arrow["{\eta_{\eta_{Q_n}} \atop \simeq}"{description}, draw=none, from=15, to=16]
	\arrow["{\eta_{\eta_{Q_{n-1}}} \atop \simeq}"{description}, draw=none, from=14, to=12]
	\arrow["{\xi_{Q_{n+1}} \atop \simeq}"{description}, draw=none, from=6-4, to=11]
	\arrow["{\xi_{Q_{n+1}} \atop \simeq}"{description, pos=0.6}, draw=none, from=2, to=7-6]
	\arrow["{\iota_{n} \atop \simeq}"{description}, draw=none, from=3, to=4]
	\arrow["{\eta_{u_n} \atop \simeq}"{description}, draw=none, from=13, to=15]
	\arrow["{\eta_{Tu_{n+1}} \atop \simeq}"{description}, draw=none, from=10, to=7]
\end{tikzcd}\]

Then one use the universal properties of the maps in the precomposite part, either as units or as bicoequalizers, to infer the existence of the desired 2-cell $  \mu_{Q_{n+1}} TTu_{n+1} TT\eta_{Q_n} Ti_{n+1} \Rightarrow 1_{TQ_{n+1}}$ corresponding to the $n_1$ of \cref{codescent diagram}. \\

Now let us construct the higher coherence data. The first one is provided by applying $T$ to the bicoequalizing 2-cell provided with $ u_{n+1}$:
\[\begin{tikzcd}[row sep=small]
	{TR_{n+1} = TTP_n} & {TP_{n+1} =TTQ_n} \\
	\\
	{TTTQ_n} \\
	{TTQ_n} & {TQ_{n+1}}
	\arrow["{TTv_n}"', from=1-1, to=3-1]
	\arrow["{TTu_n}", from=1-1, to=1-2]
	\arrow["{T\mu_{Q_n}}"', from=3-1, to=4-1]
	\arrow["{Tu_{n+1}}"', from=4-1, to=4-2]
	\arrow["{Tu_{n+1}}", from=1-2, to=4-2]
	\arrow["{T\xi_{n+1} \atop \simeq}"{description}, draw=none, from=1-1, to=4-2]
\end{tikzcd}\]

For the second 2-cell, we first construct the following auxiliary 2-cell we will call $ \upsilon_n$:
\[\begin{tikzcd}[sep=small]
	{P_n} &&&& {Q_n} \\
	&&& {TP_n} && {TQ_n} \\
	& {TQ_n} && {TTQ_n} \\
	{TQ_n} &&& {TQ_n} && {Q_{n+1}}
	\arrow["{u_n}", from=1-1, to=1-5]
	\arrow["{\eta_{P_n}}"{description}, from=1-1, to=2-4]
	\arrow["{Tu_n}"{description}, from=2-4, to=2-6]
	\arrow[""{name=0, anchor=center, inner sep=0}, "{u_{n+1}}", from=2-6, to=4-6]
	\arrow["{\eta_{Q_n}}", from=1-5, to=2-6]
	\arrow["{\eta_{u_n} \atop \simeq}"{description}, draw=none, from=2-4, to=1-5]
	\arrow[""{name=1, anchor=center, inner sep=0}, "{Tv_n}", from=2-4, to=3-4]
	\arrow[""{name=2, anchor=center, inner sep=0}, "{v_n}"', from=1-1, to=4-1]
	\arrow[""{name=3, anchor=center, inner sep=0}, "{v_n}"{description}, from=1-1, to=3-2]
	\arrow["{\eta_{TQ_n}}", from=3-2, to=3-4]
	\arrow[Rightarrow, no head, from=3-2, to=4-1]
	\arrow["{\mu_{Q_n}}", from=3-4, to=4-4]
	\arrow["{u_{n+1}}"', from=4-4, to=4-6]
	\arrow[Rightarrow, no head, from=4-1, to=4-4]
	\arrow["{\zeta_{Q_n} \atop \simeq}"{description}, draw=none, from=3-2, to=4-4]
	\arrow["{\eta_{v_n} \atop \simeq}"{description}, Rightarrow, draw=none, from=3, to=1]
	\arrow["{\xi_{n+1} \atop \simeq}"{description}, Rightarrow, draw=none, from=3-4, to=0]
\end{tikzcd}\]

Now we construct the second higher coherence 2-cell by pasting the previously constructed 2-cell together with the component of the multiplication at $ u_{n+1}$:
\[\begin{tikzcd}
	{TTP_n} && {TTQ_n} \\
	&& {TTTQ_n} \\
	{TTTQ_n} && {TTQ_{n+1}} \\
	{TTQ_n} && {TQ_{n+1}}
	\arrow["{TTu_n}", from=1-1, to=1-3]
	\arrow["{TTu_{n+1}}", from=2-3, to=3-3]
	\arrow["{TT\eta_{Q_n}}", from=1-3, to=2-3]
	\arrow["{TTv_n}"', from=1-1, to=3-1]
	\arrow["{TTu_{n+1}}"{description}, from=3-1, to=3-3]
	\arrow["{TT\upsilon_n \atop \simeq}"{description}, draw=none, from=1-1, to=3-3]
	\arrow["{\mu_{TQ_n}}"', from=3-1, to=4-1]
	\arrow["{Tu_{n+1}}"', from=4-1, to=4-3]
	\arrow["{\mu_{Q_{n+1}}}", from=3-3, to=4-3]
	\arrow["{\mu_{u_{n+1}} \atop \simeq}"{description}, draw=none, from=3-1, to=4-3]
\end{tikzcd}\]

Finally the last 2-cell is obtained by pasting the following structures 2-cells associated to the multiplication and their images along $T$
\[\begin{tikzcd}
	{TTP_n} && {TTTQ_n} & {TTQ_n} \\
	&& {TTTTQ_n} & {TTTQ_n} \\
	{TTTQ_n} & {TTTTQ_n} & {TTTQ_{n+1}} & {TTQ_{n+1}} \\
	{TTQ_n} & {TTTQ_n} & {TTQ_{n+1}} & {TQ_{n+1}}
	\arrow["{TTv_n}"', from=1-1, to=3-1]
	\arrow[""{name=0, anchor=center, inner sep=0}, "{\mu_{TQ_n}}"', from=3-1, to=4-1]
	\arrow["{TTv_n}", from=1-1, to=1-3]
	\arrow["{TTT\eta_{Q_n}}"', from=1-3, to=2-3]
	\arrow["{TTTu_{n+1}}"', from=2-3, to=3-3]
	\arrow["{TTT\eta_{Q_n}}", from=3-1, to=3-2]
	\arrow["{TTTu_{n+1}}", from=3-2, to=3-3]
	\arrow[""{name=1, anchor=center, inner sep=0}, "{T\mu_{Q_n}}", from=1-3, to=1-4]
	\arrow["{TT\eta_{Q_n}}", from=1-4, to=2-4]
	\arrow["{TTu_{n+1}}", from=2-4, to=3-4]
	\arrow["{TT\eta_{Q_n}}"', from=4-1, to=4-2]
	\arrow["{TTu_{n+1}}"', from=4-2, to=4-3]
	\arrow[""{name=2, anchor=center, inner sep=0}, "{\mu_{TQ_{n+1}}}"', from=3-3, to=4-3]
	\arrow["{\mu_{Q_{n+1}}}"', from=4-3, to=4-4]
	\arrow["{\mu_{Q_{n+1}}}", from=3-4, to=4-4]
	\arrow["{\rho_{Q_{n+1}} \atop \simeq}"{description}, draw=none, from=3-3, to=4-4]
	\arrow[""{name=3, anchor=center, inner sep=0}, "{T\mu_{Q_{n+1}}}", from=3-3, to=3-4]
	\arrow["{\mu_{TTu_{n+1}\eta_{Q_n}} \atop \simeq}"{description}, Rightarrow, draw=none, from=0, to=2]
	\arrow["{T\mu_{Tu_{n+1}\eta_{Q_n}} \atop \simeq}"{description, pos=0.3}, Rightarrow, draw=none, from=1, to=3]
\end{tikzcd}\]


\end{proof}

\begin{remark}
The construction of the step $\mathscr{X}_1$ is actually covered by the generic recurrence step, except that the composite $ Tu_1 T\eta_{Q_0}$ should not be called $v_0$ for this later name was already used. 
\end{remark}

\begin{division}[Transfinite induction]
We hence have at each $ n \in \mathbb{N}$ a codescent diagram $\mathscr{X}_n$ which we would like to see as the bar construction of a pseudo-algebras: but it is not yet true for a finite $n$, as the bicoequalizer of this step only is the next $Q_{n+1}$ rather than $Q_n$ itself. However this stabilizes after iterating the construction $\omega$ times. Those last steps are essentially the same as in 1-dimension, but for the sake of completeness and clarity we recall them. \\



The direct sequence $(Q_{n})_{n \in \mathbb{N}}$ is trivially a bifiltered diagram: hence we can compute in $\mathcal{C}$ its bifiltered bicolimit $ Q_{\omega} = \bicolim_{n \in \mathbb{N}} \, Q_n$. Moreover, observing that the set of all $ \{n+1 \mid n \in \mathbb{N} \}$ is cofinal in $\mathbb{N}$ we have 
\[
     P_\omega = \underset{n \in \mathbb{N}}{\bicolim} \; P_{n+1} 
     = \underset{n \in \mathbb{N}}{\bicolim} \; TQ_n
     \simeq T \underset{n \in \mathbb{N}}{\bicolim} \; Q_n 
     = TQ_\omega
\]
But now, the transition maps $ u_{n+1} : P_{n+1}=TQ_n \rightarrow Q_{n+1}$ induces a structure map 
\[\begin{tikzcd}
	{P_\omega \simeq TQ_\omega} & {Q_\omega}
	\arrow["{u_\omega}", from=1-1, to=1-2]
\end{tikzcd}\]

Similarly, one has 
\[ R_\omega = \underset{n \in \mathbb{N}}{\bicolim} \; R_{n+1} 
     = \underset{n \in \mathbb{N}}{\bicolim} \; TP_n
     \simeq T \underset{n \in \mathbb{N}}{\bicolim} \; P_n 
     = TP_\omega = TTQ_{\omega}  \]
     
Moreover one can check that the map $ Tv_n : TP_n \rightarrow TTQ_n$ is sent by the bicolimit to the identity $ 1_{R_\omega}$ because it is induced from the transition morphisms.   \\   

Let us describe also what becomes of the higher and lower data of the codescent diagrams $\mathscr{X}_n$ after the transfinite induction. In each case, we infer this from the expression of the $ n+1$ step for a given $n \geq 1$, using again the same cofinality argument. While we saw that $ u_n$ induces a structure map $ u_\omega$ above, we can tell that the bicolimit arrow of the $(v_n)_{n \in \mathbb{N}}$ is the identity of $P_\omega$ for $v_{n+1}$: we are going to prove indeed at \cref{structure map at transfinite induction} $u_\omega$ to be a pseudoretraction of $ \eta_{Q_\omega}$, but we know that $ v_{n+1} = T(u_{n+1} \eta_{Q_n})$, which will after transfinite induction induce an invertible 2-cell $ v_\omega \simeq 1_{P_\omega} $. As to their common section, it will be $ T\eta_{Q_\omega}$. The higher data will be determined the same way: we will use $ w_\omega = Tu_\omega$, $ \mu_{P_\omega}$ (as $ t_\omega = Tv_\omega \simeq 1_{R_\omega}$), and $ m_\omega = \mu_{Q_\omega}$. We end up with the following codescent diagram $ \mathscr{X}_\omega$:

\[\begin{tikzcd}
	{TTTQ_\omega} && {TTQ_\omega} && {TQ_\omega}
	\arrow["{\mu_{TQ_\omega}}"{description}, from=1-1, to=1-3]
	\arrow["{TTu_\omega}", shift left=4, from=1-1, to=1-3]
	\arrow["{T\mu_{Q_\omega}}"', shift right=4, from=1-1, to=1-3]
	\arrow["{T\eta_{Q_\omega}}"{description}, from=1-5, to=1-3]
	\arrow["{Tu_\omega}", shift left=4, from=1-3, to=1-5]
	\arrow["{\mu_{Q_\omega}}"', shift right=4, from=1-3, to=1-5]
\end{tikzcd}\]
\end{division}

\newpage

\begin{lemma}\label{structure map at transfinite induction}
The map $ u_\omega$ defines a structure of pseudo-algebra on $Q_\omega$. 
\end{lemma}

\begin{proof}
We exhibit triangle and square data forming together with $ u_\omega$ a structure of pseudo-algebra. For the triangle, recall that we defined at each $ n \in \mathbb{N}$ the transition map $ q_{n+1}$ as the composite
\[\begin{tikzcd}
	{Q_n} && {Q_{n+1}} \\
	& {TQ_n=P_{n+1}}
	\arrow[""{name=0, anchor=center, inner sep=0}, "{q_{n+1}}", from=1-1, to=1-3]
	\arrow["{\eta_{Q_n}}"', from=1-1, to=2-2]
	\arrow["{u_{n+1}}"', from=2-2, to=1-3]
	\arrow["{=}"{description}, Rightarrow, draw=none, from=0, to=2-2]
\end{tikzcd}\]
Then, for $ \bicolim_{n \in \mathbb{N}} Q_{n} \simeq \bicolim_{n \in \mathbb{N}} Q_{n+1} $ by cofinality, we end up with the following triangle 
\[\begin{tikzcd}
	{Q_{\omega}} && {Q_{\omega}} \\
	& {TQ_\omega = P_\omega}
	\arrow["{\eta_{Q_{\omega}}}"', from=1-1, to=2-2]
	\arrow["{u_{\omega}}"', from=2-2, to=1-3]
	\arrow[""{name=0, anchor=center, inner sep=0}, Rightarrow, no head, from=1-1, to=1-3]
	\arrow["{u^t \atop\simeq }"{description}, Rightarrow, draw=none, from=2-2, to=0]
\end{tikzcd}\]

For the square data, one has at each $n \in \mathbb{N}$ an inserted 2-cell from the bicoequalizing property of $ u_{n+1} : TQ_n \rightarrow Q_{n+1}$
\[\begin{tikzcd}
	{R_{n+1} = TP_n} & {TQ_n = P_{n+1}} \\
	{TTQ_n} \\
	{TQ_n = P_{n+1}} & {Q_{n+1}}
	\arrow["{w_{n+1}}", from=1-1, to=1-2]
	\arrow["{u_{n+1}}", from=1-2, to=3-2]
	\arrow[draw=none, from=1-1, to=3-2]
	\arrow["{u_{n+1}}"', from=3-1, to=3-2]
	\arrow["{Tv_n}"', from=1-1, to=2-1]
	\arrow["{\mu_{Q_n}}"', from=2-1, to=3-1]
	\arrow["{\xi_n \atop \simeq}"{description}, draw=none, from=1-1, to=3-2]
\end{tikzcd}\]

This 2-cell is sent by the bicolimit functor to an invertible 2-cell 
\[\begin{tikzcd}
	{R_{\omega} = TP_\omega =TTQ_{\omega}} & {TQ_{\omega} = P_{{\omega}}} \\
	{TQ_{\omega} = P_{\omega}} & {Q_{\omega}}
	\arrow["{w_{n+1}}", from=1-1, to=1-2]
	\arrow["{u_{\omega}}", from=1-2, to=2-2]
	\arrow[draw=none, from=1-1, to=2-2]
	\arrow["{u_{\omega}}"', from=2-1, to=2-2]
	\arrow["{\xi_{\omega} \atop \simeq}"{description}, draw=none, from=1-1, to=2-2]
	\arrow["{\mu_{Q_\omega}}"', from=1-1, to=2-1]
\end{tikzcd}\]
and it suffices to take $ u^s=\xi_\omega$. Then $(Q_\omega, u_\omega, (u^t, u^s)) $ is a pseudoalgebra. 
\end{proof}

\begin{remark}
One can check that the codescent diagram $ \mathscr{X}_\omega$ is the underlying diagram of the bar construction at $(Q_\omega, u_\omega, (u^t, u^s)) $, that is 
\[  \mathscr{X}_\omega = U_T \mathscr{X}_{ (Q_\omega, u_\omega, (u^t, u^s))} \]

This is not a coincidence but the very aim of the way the diagrams $\mathscr{X}_n$ were designed. The first bicoequalizer $Q_0$ of $\mathscr{X}$ does not yet bear a structure of pseudo-algebra, though being equipped with different structure maps that could be assembled in a codescent diagram $\mathscr{X}_0$ close to be a bar construction. This situation reproduced at each recurrence step, with in particular the transition map $ q_{n+1} : Q_n \rightarrow Q_{n+1}$ measuring the failure of $Q_n$ to be the bicoequalizer of the bar-like construction $\mathscr{X}_n$, and at the same time, of $u_{n+1}$ to define a structure of pseudo-algebra. Eventual stabilisation at the transfinite step simultaneously corrects those two defects at once.

\end{remark}

\begin{division}[Canonical pseudomorphism]
Now each for each $n \in \mathbb{N}$ we have not only a bicolimit inclusion $ q^\omega_0 : Q_n \rightarrow Q_\omega $ but also a bicolimit inclusion in $ [2, \mathcal{C}]_\ps$ of the arrows $u_n$
\[\begin{tikzcd}
	{P_n} & {P_\omega= TQ_\omega} \\
	{Q_n} & {Q_\omega}
	\arrow["{u_n}"', from=1-1, to=2-1]
	\arrow["{q^\omega_n}"', from=2-1, to=2-2]
	\arrow["{p^\omega_n}", from=1-1, to=1-2]
	\arrow["{u_{\omega}}", from=1-2, to=2-2]
	\arrow["{l^\omega_n \atop \simeq}"{description}, draw=none, from=1-1, to=2-2]
\end{tikzcd}\]
In particular, composing the inclusion at $0$ with the canonical 2-cell obtained at \cref{initialization codescent lower} gives us the following 2-cell which is the desired pseudomorphism 
\[\begin{tikzcd}
	TC & {P_0} & {P_\omega= TQ_\omega} \\
	C & {Q_0} & {Q_\omega}
	\arrow["{p_0}", from=1-1, to=1-2]
	\arrow["{u_0}"{description}, from=1-2, to=2-2]
	\arrow["c"', from=1-1, to=2-1]
	\arrow["{q_0}"', from=2-1, to=2-2]
	\arrow["{q^\omega_0}"', from=2-2, to=2-3]
	\arrow["{p^\omega_0}", from=1-2, to=1-3]
	\arrow["{u_{\omega}}", from=1-3, to=2-3]
	\arrow["{\xi_0 \atop \simeq}"{description}, draw=none, from=1-1, to=2-2]
	\arrow["{l^\omega_0 \atop \simeq}"{description}, draw=none, from=1-2, to=2-3]
\end{tikzcd}\]
In the following we shall denote it as $(q_\omega, \kappa_\omega)$. It is routine, yet tedious, to check that this defines a pseudomorphism. 
\end{division}

\begin{proposition}\label{transfinite coeq}
{The map $(q_\omega,\kappa_\omega) : (C,c, (\gamma^s,\gamma^t)) \rightarrow (Q_\omega, u_\omega, (u^s, u^t))$} pseudocoequalizes the codescent diagram $\mathscr{X}$ in $T\hy\psAlg$. Moreover, any other data $(l, \lambda) : (C,c, (\gamma^s,\gamma^t))  \rightarrow (D,d, (\delta^s,\delta^t)) $ pseudocoequalizing $ \mathscr{X}$ admits a factorization through $(q_\omega,\kappa_\omega) $.
\end{proposition}

\begin{proof}
The argument is similar to \cite{borceux1994handbook}. First, it is clear that this arrow $ (q_\omega, \kappa_\omega)$ pseudocoequalizes $ \mathscr{X}$ because it factorizes in $[2, \mathcal{C}]$ through the bicoequalizer of $\overline{x} : TU_T\mathscr{X} \rightarrow U_T\mathscr{X}$ (see \cref{initialization codescent lower}). \\

Now take a pseudomorphism $(l, \lambda) : (C,c, (\gamma^s,\gamma^t))  \rightarrow (D,d, (\delta^s,\delta^t)) $ pseudocoequalizing $ \mathscr{X}$. Then the underlying $ l : C \rightarrow D$ pseudocoequalizes $ U_T \mathscr{X}$ while $ Tl$ also bicoequalizes $ TU_T\mathscr{X}$: we have then a canonical 2-cell satisfying the equation below
\[\begin{tikzcd}
	& {P_0} \\
	TC && TD \\
	C && D
	\arrow[""{name=0, anchor=center, inner sep=0}, "Tl"{description}, from=2-1, to=2-3]
	\arrow["c"', from=2-1, to=3-1]
	\arrow["d", from=2-3, to=3-3]
	\arrow[""{name=1, anchor=center, inner sep=0}, "l"', from=3-1, to=3-3]
	\arrow[""{name=2, anchor=center, inner sep=0}, "{p_0}", from=2-1, to=1-2]
	\arrow[""{name=3, anchor=center, inner sep=0}, "{\langle Tl \rangle_0}", from=1-2, to=2-3]
	\arrow["{\pi_0 \atop \simeq}"{description}, Rightarrow, draw=none, from=2, to=3]
	\arrow["{\lambda \atop \simeq}"{description}, Rightarrow, draw=none, from=0, to=1]
\end{tikzcd} =
\begin{tikzcd}
	& {P_0} \\
	TC & {Q_0} & TD \\
	C && D
	\arrow["c"', from=2-1, to=3-1]
	\arrow["d", from=2-3, to=3-3]
	\arrow[""{name=0, anchor=center, inner sep=0}, "l"', from=3-1, to=3-3]
	\arrow["{p_0}", from=2-1, to=1-2]
	\arrow["{\langle Tl \rangle_0}", from=1-2, to=2-3]
	\arrow["{u_0}"{description}, from=1-2, to=2-2]
	\arrow["{q_0}"{description}, from=3-1, to=2-2]
	\arrow["{\langle l \rangle_0}"{description}, from=2-2, to=3-3]
	\arrow["{\langle \lambda \rangle_0 \atop \simeq}"{description}, draw=none, from=2-2, to=2-3]
	\arrow["{\xi_0 \atop \simeq}"{description}, draw=none, from=2-1, to=2-2]
	\arrow["{\pi_0' \atop \simeq}"{description}, Rightarrow, draw=none, from=2-2, to=0]
\end{tikzcd}\]

Now suppose you have a factorization of $ (l, \lambda)$ trough a 2-cell
\[\begin{tikzcd}
	{P_n} && TD \\
	{Q_n} && D
	\arrow["d", from=1-3, to=2-3]
	\arrow["{\langle Tl \rangle_n}", from=1-1, to=1-3]
	\arrow["{u_n}"', from=1-1, to=2-1]
	\arrow["{\langle l \rangle_n}"', from=2-1, to=2-3]
	\arrow["{\langle \lambda \rangle_n \atop \simeq}"{description}, draw=none, from=2-1, to=1-3]
\end{tikzcd}\] 
Then we must show that $ d T\langle l \rangle_n$ pseudocoequalizes the codescent diagram $ \mathscr{X}_n$. The inserted 2-cell will be constructed as the pasting 
\[\begin{tikzcd}[column sep=large]
	{TP_n} &[5pt] & {TQ_n} \\
	{TTQ_n} & TTD & TD \\
	{TQ_n} & TD & D
	\arrow["d", from=2-3, to=3-3]
	\arrow["{Tu_n}", from=1-1, to=1-3]
	\arrow["{Tv_n}"', from=1-1, to=2-1]
	\arrow["{\mu_{Q_n}}"', from=2-1, to=3-1]
	\arrow["{T\langle l \rangle_n}"', from=3-1, to=3-2]
	\arrow["d"', from=3-2, to=3-3]
	\arrow[""{name=0, anchor=center, inner sep=0}, "{T\langle l \rangle_n}", from=1-3, to=2-3]
	\arrow["{TT\langle l \rangle_n}"{description}, from=2-1, to=2-2]
	\arrow["{\mu_{D}}"{description}, from=2-2, to=3-2]
	\arrow["Td"{description}, from=2-2, to=2-3]
	\arrow["{\delta^s \atop \simeq}"{description}, draw=none, from=2-2, to=3-3]
	\arrow["{\mu_{\langle l \rangle_n} \atop \simeq}"{description}, draw=none, from=2-1, to=3-2]
	\arrow[""{name=1, anchor=center, inner sep=0}, "{T\langle Tl \rangle_n}"{description},  bend left=20, from=1-1, to=2-2]
	\arrow["{ T\nu_n \atop \simeq}"{description}, Rightarrow, draw=none, from=2-1, to=1]
	\arrow["{T\langle \lambda \rangle_n \atop \simeq}"{description}, draw=none, from=1, to=0]
\end{tikzcd}\]
where $\nu_n$ can be obtained from the universal property of the unit and bicoequalizer of codescent object hidden in $v_n$. Once one has proven this 2-cell to satisfy the coherence conditions, we are provided with a canonical 2-cell which is the next factorization posting $ \langle Tl \rangle_{n +1} = T\langle l \rangle_n$:
\[\begin{tikzcd}
	{TQ_{n} = P_{n+1}} & TD \\
	{Q_{n+1}} & D
	\arrow["{u_{n+1}}"', from=1-1, to=2-1]
	\arrow["{T\langle l \rangle_n}", from=1-1, to=1-2]
	\arrow["d", from=1-2, to=2-2]
	\arrow["{\langle l \rangle_{n+1}}"', from=2-1, to=2-2]
	\arrow["{\langle \lambda\rangle_{n +1} \atop \simeq}"{description}, draw=none, from=1-1, to=2-2]
\end{tikzcd}\]
Having constructed such a factorization for any $n \in \mathbb{N}$, we can induce an ultimate factorization
\[\begin{tikzcd}
	{TQ_{\omega} } & TD \\
	{Q_{\omega}} & D
	\arrow["{u_{\omega}}"', from=1-1, to=2-1]
	\arrow["{T\langle l \rangle_\omega}", from=1-1, to=1-2]
	\arrow["d", from=1-2, to=2-2]
	\arrow["{\langle l \rangle_{\omega}}"', from=2-1, to=2-2]
	\arrow["{\langle \lambda\rangle_{\omega} \atop \simeq}"{description}, draw=none, from=1-1, to=2-2]
\end{tikzcd}\]
which is the desired 2-cell establishing the factorization of $(l, \lambda)$ through $(q_\omega, k_\omega)$. 
\end{proof}

\begin{remark}
In fact one can check that the factorization provided by $(\langle l \rangle_\omega, \langle \lambda \rangle_\omega) $ is a morphism in the 2-category of elements $[\mathbb{X}^{\op}, \Cat][\mathcal{J}, T\hy\psAlg[ \mathscr{X},-]] $, that is, is a morphism between pseudocoequalizing data for $\mathscr{X}$. 
\end{remark}

\begin{corollary}\label{existence of codescent objec for a finitary pseudomonad}
The category $ T\hy\psAlg$ has bicoequalizers of codescent objects. 
\end{corollary}

\begin{proof}
We shown at \cref{transfinite coeq} that $(q_\omega,\kappa_\omega) : (C,c, (\gamma^s,\gamma^t)) \rightarrow (Q_\omega, u_\omega, (u^s, u^t))$ provides a single-object co-initial family in the 2-category of elements of $[\mathbb{X}^{\op}, \Cat][\mathcal{J}, T\hy\psAlg[ \mathscr{X},-]]  : T\hy\psAlg \rightarrow \Cat$. Moreover, $T\hy\psAlg$ has bilimits inherited from $ \mathcal{C}$ as established at \cref{pseudoaglebras inherit bilimits}, which are clearly preserved by this 2-functor from its expression. Hence, as explained in \cref{Betti}, \cite{betti1988complete}[Theorem 3.5] ensures hence that this 2-functor is birepresentable: the representing object is the desired bicoequalizer of $\mathscr{X}$. 
\end{proof}

\begin{theorem}\label{The theorem}
Let $(T, \eta, \mu, (\xi, \zeta, \rho))$ be a bifinitary pseudomonad on a bicomplete and bicocomplete 2-category $\mathcal{C}$. Then the 2-category of pseudo-algebras and pseudomorphisms $ T\hy\psAlg$ is bicomplete and bicocomplete. 
\end{theorem}

\begin{proof}
First, bicompleteness is inherited from $ \mathcal{C}$ as established at \cref{pseudoaglebras inherit bilimits}. Hence
From \cref{transfinite coeq}, bicoequalizers exist in $T\hy\psAlg$. But from \cref{Linton}, this ensures existence of oplax bicolimits, and thus of arbitrary bicolimits thanks to \cref{Oplax bicolimit and bicoeq of codescent generates bicolimits}.
\end{proof}

We finish with a list of examples, which are also investigated in further details in \cite{ODL}:

\begin{example}
Recall that the 2-category $ \Lex$ of small lex categories and lex functors is the 2-category of pseudo-algebras for the (strict) KZ-monad $\Lex[-]$ on $\Cat$ sending a small category on its free completion under finite limits. In \cite{ODL}[proposition 5.3.2] we prove the KZ-monad $ \Lex[-]$ to be bifinitary. For $\Cat $ is bicocomplete, so is $ \Lex$ by \cref{The theorem}.
\end{example} 

\begin{example}
In the same reference \cite{ODL}[section 5.4] we investigate 2-categories of \emph{$\Phi$-exact categories} for a class of finite weights $ \Phi$ - see \cite{garner2012lex} for the general theory of \emph{lex colimits}. The 2-category of $\Phi$-exact categories and $\Phi$-exact functors is the 2-category of pseudo-algebras of a pseudomonad $ \Phi_l$ on $\Lex$. In \cite{ODL}[Lemma 5.4.4] we prove the pseudomonad $\Phi_l$ to be bifinitary. As $ \Lex$ is bicocomplete, so is the 2-category $ \Phi\hy\textbf{Ex}$ of $\Phi$-exact categories. This includes the following examples of 2-categories, which are hence bicocomplete: \begin{itemize}
    \item $ \textbf{Reg}$, the 2-category of small regular categories and regular functors;
    \item $ \textbf{Ex}$, the 2-category of small (Barr)-exact categories and exact functors;
    \item $ \textbf{Coh}$, the 2-category of small coherent categories and coherent functors;
    \item $ \textbf{Ext}_\omega$, the 2-category of small finitely-extensive categories and functors preserving finite coproducts;
    \item $ \textbf{Adh}$, the 2-category of small adhesive categories and adhesive functors;
    \item $ \textbf{Pretop}_\omega$, the 2-category of small finitary pretopoi and coherent functors.
\end{itemize}
\end{example}

\bibliography{Bib}
\bibliographystyle{alpha}

\vfill

 \begin{minipage}{0.49\textwidth}
 
 \textsc{Axel Osmond} 

\vspace{0.2cm}
{\small \textsc{Istituto Grothendieck,
		Corso Statuto 24, 12084 Mondovì, Italy.}\\
	\emph{E-mail address:} \texttt{axel.osmond@igrothendieck.org}}
 
\end{minipage}

\end{document}